\documentclass[11pt]{article}     

\usepackage[bbgreekl]{mathbbol}

\usepackage[latin1]{inputenc}
\usepackage[swedish, english]{babel}
\usepackage{epsfig}
\usepackage{float}
\usepackage{subfigure}
\usepackage{latexsym}
\usepackage{amssymb, amsmath, amsthm}
\usepackage{color}
\usepackage{multirow}

\newtheorem{theorem}{Theorem}[section]
\newtheorem*{theorem*}{Theorem}

\newtheorem*{corollary*}{Corollary}
 
\newtheorem*{lemma*}{Lemma} 
\newtheorem{proposition}[theorem]{Proposition} 
\newtheorem*{proposition*}{Proposition} 
 
\newtheorem*{assumption*}{Assumption} 
\theoremstyle{remark} 
\newtheorem{remark}[theorem]{Remark} 
\theoremstyle{definition} 
 
\theoremstyle{definition} 
\newtheorem*{definition*}{Definition}

\title{A dual consistent  finite difference  method with \\
narrow stencil second derivative operators}

\usepackage{fancyhdr}
\begin{document}
\thispagestyle{plain}

\date{}
\author{Sofia Eriksson\\
\\
{\it\small Department of Mathematics, TU Darmstadt, Germany}}

\maketitle


\definecolor{kommentar}{rgb}{.8, 0, 0.4}%
\newcommand{\kom}{\color{kommentar}\bf}

\newcommand{\PH}{H}

 \newcommand{\JutL}{\beta}
 \newcommand{\JutR}{\alpha}
 \newcommand{\gldual}{\dual{\hspace{1pt}\gl}}
\newcommand{\grdual}{\dual{\gr}}



\newcount\tmpnum \newdimen\tmpdim
{\lccode`\?=`\p \lccode`\!=`\t  \lowercase{\gdef\ignorept#1?!{#1}}}
\edef\widecharS{\expandafter\ignorept\the\fontdimen1\textfont1}

\def\widebar#1{\futurelet\next\widebarA#1\widebarA}
\def\widebarA#1\widebarA{%
   \def\tmp{0}\ifcat\noexpand\next A\def\tmp{1}\fi
   \widebarE
   \ifdim\tmp pt=0pt \overline{#1}%
   \else {\mathpalette\widebarB{#1}}\fi
}

\def\widebarB#1#2{%
   \setbox0=\hbox{$#1\overline{#2}$}%
   \tmpdim=\tmp\ht0 \advance\tmpdim by-.4pt
   \tmpdim=\widecharS\tmpdim
   \kern\tmpdim\overline{\kern-\tmpdim#2}%
}
\def\widebarC#1#2 {\ifx#1\end \else 
   \ifx#1\next\def\tmp{#2}\widebarD 
   \else\expandafter\expandafter\expandafter\widebarC
   \fi\fi
}
\def\widebarD#1\end. {\fi\fi}
\def\widebarE{\widebarC A1.4 J1.2 L.6 O.8 T.5 U.7 V.3 W.1 Y.2 
   a.5 b.2 d1.1 h.5 i.5 k.5 l.3 m.4 n.4 o.6 p.4 r.5 t.4 v.7 w.7 x.8 y.8
   \alpha1 \beta1 \gamma.6 \delta.8 \epsilon.8 \varepsilon.8 \zeta.6 \eta.4
   \theta.8 \vartheta.8 \iota.5 \kappa.8 \lambda.5 \mu1 \nu.5 \xi.7 \pi.6
   \varpi.9 \rho1 \varrho1 \sigma.7 \varsigma.7 \tau.6 \upsilon.7 \phi1
   \varphi.6 \chi.7 \psi1 \omega.5 \cal1 \end. }

\def\test#1{$\let\.=#1 \.M, \.A, \.g, \.\beta, \.{\cal A}^q, \.{AB}^\sigma, 
  \.H^C, \.{\sin z}, \.W_{\!n}$}



\newcommand{\dx}{\,\mathrm {d} x}
\newcommand{\ddt}{\frac{\mathrm {d} }{\mathrm {d} t}}
\newcommand{\ddtau}{\frac{\mathrm {d} }{\mathrm {d} \tau}}
\newcommand{\xl}{x_{\scalebox{.6}{$L$\hspace{1pt}}} }
\newcommand{\xr}{x_{\scalebox{.6}{$R$\hspace{1pt}}} }

\newcommand{\func}{\mathcal{J}}
\newcommand{\linop}{\mathcal{L}}
\newcommand{\inprod}[1]{\langle{#1}\rangle}
\newcommand{\discmark}{{\hspace{-1pt}{}_\PH}}

\newcommand{\sol}{\hspace{1pt}\mathcal{U}}
\newcommand{\force}{\mathcal{F}}
\newcommand{\weight}{\mathcal{G}}
\newcommand{\dualvar}{\mathcal{V}}

\newcommand{\nsol}{U}
\newcommand{\fh}{F}
\newcommand{\gh}{G}
\newcommand{\dualdisc}{V}
\newcommand{\funcdisc}{J}
\newcommand{\lindisc}{L}

\newcommand{\hypR}{\mathcal{R}\hspace{0.5pt}} 
\newcommand{\hypA}{\mathcal{A}\hspace{0.5pt}} 
\newcommand{\hypBL}{\mathcal{B}_{\scalebox{.7}{$L$\hspace{1pt}}} }
\newcommand{\hypBR}{\mathcal{B}_{\scalebox{.7}{$R$\hspace{1pt}}} }
\newcommand{\gl}{g_{\scalebox{.6}{$L$\hspace{1pt}}} }
\newcommand{\gr}{g_{\scalebox{.6}{$R$\hspace{1pt}}} }

\newcommand{\scalingJ}{P}
\newcommand{\ratioR}{R}

\newcommand{\Rot}{\Delta}
\newcommand{\X}{Z}
\newcommand{\RL}{\ratioR_L}
\newcommand{\RR}{\ratioR_R}
\newcommand{\gsnokL}{\widetilde{g}_L}
\newcommand{\gsnokR}{\widetilde{g}_R}

\newcommand{\JL}{\scalingJ_L}
\newcommand{\JR}{\scalingJ_R}
\newcommand{\CL}{\mathcal{C}_L}
\newcommand{\CR}{\mathcal{C}_R}

\newcommand{\dual}[1]{\widetilde{#1}}
\newcommand{\RLdual}{\dual{\RL}}

\newcommand{\hypBLdual}{\dual{\hypBL}}
\newcommand{\hypBRdual}{\dual{\hypBR}}
\newcommand{\JLdual}{\dual{\JL}}
\newcommand{\dualCL  }{\dual{\CL }}
\newcommand{\dualCR  }{\dual{\CR }}

\newcommand{\SigL}{\Sigma_0}
\newcommand{\SigR}{\Sigma_N}
\newcommand{\Pbar}{\hspace{2pt}\widebar{\hspace{-2pt}\PH}}
\newcommand{\Dbar}{\widebar{D}}
\newcommand{\Sbar}{\widebar{S}}
\newcommand{\compA}{A_{\scalebox{.7}{$S$\hspace{1pt}}} }
\newcommand{\compAsnok}{\widetilde{A}_{\scalebox{.7}{$S$\hspace{1pt}}} }
\newcommand{\Msnok}{\widetilde{M}}
\newcommand{\pert}{\delta}
\newcommand{\RHS}{\text{RHS}}
\newcommand{\BTmat}{C}

\newcommand{\cpL}{\Pi_0}
\newcommand{\cmL}{\Gamma_0}
\newcommand{\cpR}{\Gamma_N}
\newcommand{\cmR}{\Pi_N}

\newcommand{\GG}{\mathcal{G}}
\newcommand{\factor}{\hspace{1.5pt}\mathcal{K}}

\newcommand{\paraA}{\mathcal{A}}
\newcommand{\paraE}{\mathcal{E}}
\newcommand{\HL}{\mathcal{H}_L}
\newcommand{\GL}{\GG_L}
\newcommand{\HR}{\mathcal{H}_R}
\newcommand{\GR}{\GG_R}
\newcommand{\factorL}{\factor_L}
\newcommand{\factorR}{\factor_R}

\newcommand{\sizebig}{m}
\newcommand{\sizepara}{n}

\newcommand{\bigI}{\widebar{\mathcal{I}} }
\newcommand{\bigR}{\hspace{1pt}\widebar{\hspace{-1pt}\hypR}} 
\newcommand{\bigA}{\hspace{3pt}\widebar{\hspace{-3pt}\hypA}} 
\newcommand{\bigU}{\hspace{2pt}\widebar{\hspace{-2pt}\sol }}
\newcommand{\bigF}{\hspace{2.5pt}\widebar{\hspace{-2.5pt}\force\hspace{0.5pt}}\hspace{-0.5pt}} 
\newcommand{\bigBL}{\hspace{1pt}\widebar{\hspace{-1pt}\mathcal{B}\hspace{1pt}}\hspace{-2pt}_L}
\newcommand{\bigBR}{\hspace{1pt}\widebar{\hspace{-1pt}\mathcal{B}\hspace{1pt}}\hspace{-2pt}_R}

\newcommand{\bigX}{\hspace{2pt}\widebar{\hspace{-2pt}\X}}
\newcommand{\bigRot}{\hspace{1pt}\widebar{\hspace{-1pt}\Rot\hspace{-1pt}}\hspace{1pt}}
\newcommand{\bigJL}{\hspace{2pt}\widebar{\hspace{-2pt}\scalingJ\hspace{1pt}}_{\hspace{-2pt}L}}
\newcommand{\bigJR}{\hspace{2pt}\widebar{\hspace{-2pt}\scalingJ\hspace{1pt}}_{\hspace{-2pt}R}}
\newcommand{\bigRL}{\hspace{2pt}\widebar{\hspace{-2pt}\ratioR}_L}
\newcommand{\bigRR}{\hspace{2pt}\widebar{\hspace{-2pt}\ratioR}_R}
\newcommand{\bigNL}{0_{\sizebig_+,\sizebig_0}}
\newcommand{\bigNR}{0_{\sizebig_-,\sizebig_0}}

\newcommand{\bigV}{\hspace{1pt}\widebar{\hspace{-1pt}\nsol }}
\newcommand{\bigFdisc}{\hspace{2pt}\widebar{\hspace{-2pt}\fh}} %
\newcommand{\bigSigL}{\widebar{\Sigma}_\noll}
\newcommand{\bigSigR}{\widebar{\Sigma}_N}

\newcommand{\wide}[1]{\widehat{#1}}
\newcommand{\sigltop}{\sigma_0}
\newcommand{\siglbot}{\tau_0}
\newcommand{\sigrtop}{\sigma_N}
\newcommand{\sigrbot}{\tau_N}
\newcommand{\p}{\wide{\q}}

\newcommand{\penI}{\mu}
\newcommand{\penS}{\nu}

\newcommand{\taul}{\penI_0}
\newcommand{\sigmal}{\hspace{0pt}\penS_{\hspace{0.5pt}0\hspace{0.5pt}}}
\newcommand{\taur}{\penI_N}
\newcommand{\sigmar}{\penS_N}

\newcommand{\compact}[1]{{\color{compact}\widehat{#1}}}
\newcommand{\q}{q}
\newcommand{\wpen}{\kappa}
\newcommand{\Wtilde}{\widetilde{W}}
\newcommand{\VCOMP}{\hspace{-1pt}\widetilde{\hspace{1pt}U\hspace{1pt}}\hspace{-1pt}}

\newcommand{\al}{\alpha_{{}_L}}
\newcommand{\ar}{\alpha_{{}_R}}
\newcommand{\bl}{\beta_{{}_L}}
\newcommand{\br}{\beta_{{}_R}}
\newcommand{\rot}{\omega}

\newcommand{\noll}{{\hspace{1pt}0\hspace{0.5pt}}}

  \newcommand{\maxeig}{\rho}
  \newcommand{\mineig}{\eta}

\newcommand{\BTsnok}{\widetilde{\hspace{0.5pt}\text{BT}\hspace{0.5pt}}{}^{Disc.\hspace{-1pt}}}

\newcommand{\Jfel}{\mathtt{E}}
\newcommand{\fel}{\mathtt{e}}

\newcommand{\ett}{\mathbf{1}} %
\newcommand{\stympadetta}{\vec{1}}

\newcommand{\storidentitet}{I}
\newcommand{\litenidentitet}{\bar{I}}

\newcommand{\Atop}{a}
\newcommand{\Avec}{\vec{a}}
\newcommand{\Adel}{\bar{A}}
\newcommand{\Kdel}{\Adel^{-1}}

\newcommand{\stympadnolla}{\vec{0}}
\newcommand{\stornollmatris}{0}



\newcommand{\nondual}[1]
{\tilde{#1}}
\newcommand{\kappasnok}{\tilde{\kappa}}

\newcommand{\scalaru}{\sol}
\newcommand{\scalarv}{\nsol}
\newcommand{\abar}{a}
\newcommand{\bbar}{b}
\newcommand{\HRdual}{\dual{\HR\hspace{-1pt}}\hspace{1pt}}
\newcommand{\ansatz}{k}


\begin{abstract}

We study the numerical solutions of 
 time-dependent systems of partial differential equations, focusing on the implementation of boundary conditions. 
 The numerical method considered is a
 finite  difference scheme constructed by high order  summation by parts operators, combined with a boundary procedure using penalties (SBP-SAT). 

Recently 
it was shown that 
SBP-SAT finite difference methods can
yield superconvergent  functional output  if the boundary conditions are imposed such that the discretization is dual consistent.
We generalize these 
 results so that they include a 
 broader
  range of boundary conditions and penalty parameters.
The results are also
generalized to hold for 
narrow-stencil
second derivative operators.
The derivations are supported by numerical experiments. 

\end{abstract}


\section{Introduction}
\label{intro}

In this paper we consider a summation by parts (SBP)
 finite difference method,  which is
combined with a penalty  technique denoted simultaneous approximation term (SAT) for the 
boundary conditions. 
The main advantages of the SBP-SAT finite difference methods are  high accuracy, computational efficiency and provable stability. 
 For a 
 background on 
the history and 
the newer developments 
of SBP-SAT,
see \cite{Magnus201417,DelReyFernandez2014171}.

A discrete differential operator $D_1$ is said to be 
a SBP-operator if it can be factorized by the inverse of a positive definite matrix $\PH $ and a 
difference operator $Q$, as specified later in equation \eqref{SBPprop1}.
When $\PH $ is diagonal, $D_1$ consists of a $2p$-order accurate central difference approximation in the interior, 
but
at the boundaries, the accuracy is limited to $p$th order. The global accuracy of the numerical solution can then be shown to be $p+1$, see \cite{Magnus201417,ref:SVAD06}.

In many applications functionals are of interest, sometimes they are even more important than the primary solution itself (one example is  lift or drag coefficients in computational fluid dynamics).
It could be expected that functionals computed from the numerical solution would have the same order of accuracy as the solution itself.
However, recently  Hicken and Zingg \cite{HickenNo17} showed that 
when 
computing the numerical solution
in a dual consistent way, the order of accuracy of the output functional is higher than the FD solution itself, in fact, the full $2p$ accuracy can be recovered. 
Related papers are \cite{Hicken2013111,HickenNo18} which includes interesting work on SBP operators as quadrature rules and  error estimators for functional errors.
 Note 
 that this kind of superconvergent behavior
was  already known for example for finite element  and discontinuous Galerkin methods, but it had not been  proven for finite difference schemes before, see \cite{HickenNo17}.
Later  
Berg and Nordstr\"{o}m \cite{Berg20126846,Berg201341,Berg2014135}
showed that the results hold
 also for time-dependent problems.

 In \cite{HickenNo17,HickenNo18} and \cite{Berg20126846} boundary conditions of Dirichlet type are considered (in \cite{HickenNo17} Neumann boundary conditions are included but are rewritten on first order form), and  in \cite{Berg201341,Berg2014135}  boundary conditions of far-field type are derived.
In this paper, we generalize these results  by deriving penalty 
parameters
that yield dual consistency for all energy stable boundary conditions of Robin type (including the special cases Dirichlet and Neumann). 
In contrast to \cite{Berg201341,Berg2014135}, where the boundary conditions were adapted to get the penalty in a certain form, we  adapt the penalty after the boundary conditions
instead. 
Furthermore, we extend
the results  such that they hold also for narrow-stencil second derivative operators (sometimes also denoted compact second derivative operators), 
where the
term narrow  is used to define explicit finite difference schemes
with a minimal stencil width.
In fact, the results even carry over to narrow-stencil  second derivatives operators for variable coefficients (of the type considered for example in \cite{Mattsson2012}).


To keep things simple we consider  linear 
 problems 
 in one spatial dimension, however, note that this is not due to a limitation of the method.
In 
  \cite{HickenNo17,HickenNo18} the extension to higher dimensions, curvilinear  grids and non-linear problems are discussed and implemented for stationary problems 
  and in \cite{Berg2014135} the 
  theory is 
applied to the 
time-dependent
Navier--Stokes and Euler equations in two dimensions.

The paper is organized as follows: 
In Section~\ref{SectHyp} we consider  hyperbolic systems of  partial differential equations
and derive a family of SAT parameters which guarantees
a stable and dual consistent discretization.
Since
higher order
differential equations can  always  be rewritten
as first order systems, 
this result directly leads to penalty parameters for parabolic problems, when using wide-stencil second derivative operators. Next, 
these parameters are generalized 
such that they hold also for narrow-stencil second derivative operators. This is all done in Section~\ref{SectPara}.
In Section~\ref{AppRinv} a special aspect of the stability for the narrow operators is discussed. The derivations are then followed by examples and numerical simulations in Section~\ref{ExNum} and a summary is given  in Section~\ref{conclusions}.

\subsection{Preliminaries}

We consider time-dependent partial differential equations (PDE) as
\begin{align}
\label{primalprototype}
\begin{split}
\sol _t+\linop(\sol )&=\force ,\hspace{20pt} t\in[0,T],
\hspace{20pt}x\in\Omega,
\end{split}
\end{align}
where  $\linop$ represents a linear, spatial differential operator and $\force (x,t)$ is a forcing function. For simplicity, we will assume that the sought solution $\sol (x,t)$ satisfies homogeneous initial and boundary conditions.
To derive the dual equations  we follow \cite{HickenNo17,Berg20126846,Berg201341}
 and  pose  the  problem in a variational framework:
Given a functional  $\func(\sol )=\inprod{ \weight ,\sol }$, where  $\weight (x,t)$ is  smooth weight function and where $\inprod{ \weight ,\sol }=\int_\Omega  \weight^T\sol  \dx$
refers to the standard $L^2$ inner product, we seek a function $\dualvar(x,t)$ such that $\func(\sol )=\func^*(\dualvar)=\inprod{ \dualvar,\force }$.
This defines the dual problem as
\begin{align}
\label{dualprototype}\begin{split}
\dualvar_\tau+\linop^*(\dualvar)&=\weight ,\hspace{20pt} \tau\in[0, T],
\hspace{20pt}x\in\Omega,
\end{split}
\end{align}
where $\linop^*$ is the 
adjoint  operator, 
given by $\inprod{\dualvar,\linop \sol }=\inprod{ \linop ^*\dualvar,\sol }$, 
and where
 $\dualvar$ also satisfies homogeneous initial and boundary
conditions.
Note that the dual problem actually goes "backward" in time;  the expression in \eqref{dualprototype} is obtained using the transformation $\tau=T-t$.

Let $\nsol $ and $\dualdisc$ be discrete vectors approximating $\sol $ and $\dualvar$, respectively, and let  $\fh $ and $\gh $ be projections  of $\force$ and $\weight$ onto a spatial grid. We discretize  \eqref{primalprototype}
using a stable and consistent SBP-SAT scheme, leading to
\begin{align}\label{discprototype}\begin{split}
\nsol _t+\lindisc \nsol &=\fh ,\hspace{20pt} t\in[0,T].
\end{split}
\end{align}
The SBP scheme has an associated  matrix $\PH $ which defines a discrete 
inner product, as $\inprod{ \gh,\nsol}_\discmark= \gh^T\PH \nsol $ (when $\sol $ is vector-valued,  $\PH $ must be replaced by $\Pbar$, which is defined later in the paper).
Now 
the discrete adjoint operator 
is given by $\lindisc ^*= \PH ^{-1}\lindisc ^T\PH $, since  this leads to
$\inprod{\dualdisc,\lindisc  \nsol }_\discmark=\inprod{ \lindisc ^*\dualdisc,\nsol }_\discmark$ which mimics the continuous relation above. 

If
 $\lindisc ^*$ happens to be a consistent approximation of $\linop ^*$, 
then the discretization \eqref{discprototype}  is said to be {\it dual consistent} (if considering the stationary case) or {\it spatially dual consistent}, see \cite{HickenNo17,Berg20126846} respectively. 
When \eqref{discprototype} is a stable and dual consistent discretization of \eqref{primalprototype}, then the linear functional $\funcdisc(\nsol )=\inprod{ \gh ,\nsol }_\discmark$ is a $2p$-order accurate approximation of $\func(\sol )$, that is $\funcdisc(\nsol )=\func(\sol )+\mathcal{O}(h^{2p})$, and we thus have superconvergent functional output. 
To obtain such high accuracy it is necessary with compatible and sufficiently smooth data, 
see \cite{HickenNo17} for more details.

\section{Hyperbolic systems}

\label{SectHyp}

We start by
considering a hyperbolic system of PDEs 
of reaction-advection type, namely
\begin{align}
\label{HypSyst}
\begin{array}{rll}\sol_t+\hypR \sol + \hypA \sol _x=&\hspace{-7pt}\force ,\hspace{20pt}&x\in[\xl,\xr],\hspace{20pt}\vspace{4pt}\\\hypBL \sol =&\hspace{-7pt}\gl,&x=\xl,\vspace{4pt}\\\hypBR \sol =&\hspace{-7pt}\gr,&x=\xr,\end{array}
\end{align}
valid for $t\geq0$ and augmented with initial data 
$\sol (x,0)=\sol _0(x)$.
We let $\hypR$ and $\hypA$ be real-valued, symmetric $n\times n $ matrices.
Further,  $\hypR$ is positive semi-definite, that is $\hypR\geq0$.
The operators $\hypBL$ and $\hypBR$  
define the form of
the boundary conditions and 
their properties are specified in \eqref{RLrelation} 
below.
The 
forcing function $\force (x,t)$, the initial data 
$\sol _0(x)$ and the boundary data $\gl\hspace{-1pt}(t)$ and $\gr\hspace{-1pt}(t)$ are assumed to be
compatible and sufficiently smooth 
 such that  the solution $\sol (x,t)$ exists.
We will refer to \eqref{HypSyst} as our primal problem.

\subsection{Well-posedness using the energy method}

We call  \eqref{HypSyst} well-posed if it has a unique solution and is stable. 
Existence is guaranteed by using the right number of boundary conditions, and uniqueness then follows from the stability, \cite{ref:NORD052,GKO}. 
Next we  show stability, using the energy method.

The PDE in the first row of \eqref{HypSyst} is multiplied by $\sol ^T$ from the left and integrated over the domain $\Omega=[\xl,\xr]$.  
Using integration by parts we obtain
\begin{align}\label{GrowthRate}
\ddt\|\sol \|^2+2\inprod{ \sol ,\hypR \sol }&=2\inprod{ \sol ,\force }+
\text{\text{BT}}_L+\text{\text{BT}}_R
\end{align}
where $\|\sol \|^2=\inprod{ \sol , \sol }=\int_{\xl}^{\xr}\sol  ^T \sol \dx$ and where
\begin{align}
\text{\text{BT}}_L=\left.\sol  ^T\hypA \sol  ^{}\right|_{\xl},&&
\text{\text{BT}}_R=-\left.\sol  ^T\hypA \sol  ^{}\right|_{\xr}.\notag
\end{align}%
To bound the growth of the solution, we must ensure that the boundary conditions make $\text{\text{BT}}_L$ and $\text{\text{BT}}_R$  non-positive for zero data. 
We consider the matrix $\hypA$ above and assume that we have found a factorization such that 
\begin{align}
\label{ALambdaX}
\hypA=\X\Rot \X^T, &&\Rot=\left[\begin{array}{ccc}\Rot^{ }_+\\&\Rot^{ }_0\\&&\Rot^{ }_-\end{array}\right],&&\X=\left[\X_+,\X_0,\X_-\right],
\end{align}
where 
$\X$ is non-singular. The parts of $\Rot$ are arranged such that $\Rot_+>0$, $\Rot_0=0$ and $\Rot_-<0$. 
According to Sylvester's law of inertia, the matrices $\hypA$ and $\Rot$ 
have the same number of positive ($n_+$), negative ($n_-$) and zero ($n_0$) eigenvalues (for a non-singular $\X$), where
$n=n_++n_0+n_-$.
To bound
the terms $\text{\text{BT}}_L$ and $\text{\text{BT}}_R$, we have to give $n_+$ boundary conditions at $x=\xl$ and $n_-$ boundary conditions at $x=\xr$. 
We note that 
\begin{align}   
\label{rot}
\hypA=\X^{ }_+\Rot^{ }_+\X_+^T+\X_-^{ }\Rot^{ }_-\X_-^T,
\end{align}
which gives
\begin{align}
\text{\text{BT}}_L
=\left.\sol ^T\hspace{-2pt}\left(\X^{ }_+\Rot^{ }_+\X_+^T+\X_-^{ }\Rot^{ }_-\X_-^T\right)\sol \right|_{\xl},&&
\text{BT}_R
=-\left.\sol ^T\left(\X^{ }_+\Rot^{ }_+\X_+^T+\X_-^{ }\Rot^{ }_-\X_-^T\right)\sol \right|_{\xr}
\notag
\end{align}
where
$\X_+^T\sol $ represents the right-going variables (ingoing at the left boundary), and $\X_-^T\sol $ represents the left-going variables (ingoing at the right boundary).
The ingoing variables are given data in terms of known functions and outgoing variables, 
as 
\begin{align}
\label{charbc}
\left.\X_+^T\sol\right|_{\xl} =\left.\gsnokL-\RL\X_-^T\sol\right|_{\xl},&&\left.\X_-^T\sol\right|_{\xr} =\left.\gsnokR-\RR\X_+^T\sol\right|_{\xr},
\end{align}
where $\gsnokL$, $\gsnokR$ are the known data and
where the matrices $\RL$ and $\RR$
must be sufficiently small. 
Using the boundary conditions %
in \eqref{charbc},
the boundary terms $\text{BT}_L$ and $\text{BT}_R$ become
\begin{align}
\label{ContBT}
\begin{split}
\text{BT}_L
&=\hspace{10pt}\left.\sol ^T\X_-^{ }\left(\Rot^{ }_-+\RL^T\Rot^{ }_+\RL\right)\X_-^T\sol \right|_{\xl}
\hspace{2pt}-\left.2\gsnokL^T\Rot^{ }_+\RL\X_-^T\sol \right|_{\xl}
\hspace{2pt}+\gsnokL^T\Rot^{ }_+\gsnokL
\\
\text{BT}_R&=-\left.\sol ^T\X_+^{ }\left(\Rot^{ }_++\RR^T\Rot^{ }_-\RR\right)\X_+^T\sol \right|_{\xr}
+\left.2
\gsnokR^T\Rot^{ }_-\RR\X_+^T\sol \right|_{\xr}
-\gsnokR^T\Rot^{ }_-\gsnokR.\end{split}
\end{align}
We define
\begin{align*}
\CL =\Rot^{ }_-+\RL^T\Rot^{ }_+\RL,&&\CR =-\Rot^{ }_+-\RR^T\Rot^{ }_-\RR
\end{align*}
and note that
if  $\CL$,   $\CR\leq0$,  the boundary terms in \eqref{ContBT} will be non-positive for zero data. 
By integrating \eqref{GrowthRate} in time 
we can now obtain 
a bound on $\|\sol \|^2$.
With boundary conditions on the form \eqref{charbc}, we also know that  the correct number of boundary conditions are specified at each boundary, which yields existence.
Our problem is thus well-posed.

To relate the original
 boundary conditions %
 in \eqref{HypSyst}
 to  the ones in \eqref{charbc},
 we let
\begin{align}
\label{RLrelation}
\hypBL=\JL(\X_+^T+\RL\X_-^T),\hspace{30pt}\hypBR=\JR(\X_-^T+\RR\X_+^T),
\end{align}
where $\JL$ and $\JR$ are invertible 
scaling and/or permutation matrices.
The data in \eqref{charbc} is identified as $\gsnokL=\JL^{-1}\gl$ and $\gsnokR=\JR^{-1}\gr$.
We assume that the boundary conditions in \eqref{HypSyst} are properly chosen such that
$\RL$ and $\RR$ are sufficiently small and hence $\CL$,
$\CR \leq0$.

\begin{remark}
Note that the energy method is a sufficient but not necessary condition for stability
and that it is rather restrictive with respect to the admissible boundary conditions.
By rescaling
 the problem we could allow $\RL$ and $\RR$ to be larger,
see \cite{KL,GKO}.
We will not consider this complication but simply require that $\CL\leq0$, $\CR\leq0$. 
\end{remark}

\begin{remark}
\label{remarkBound}

In the homogeneous case, 
with boundary conditions such that $\CL$, $\CR \leq0$,
the growth rate in
\eqref{GrowthRate} 
becomes $\ddt\|\sol \|^2\leq0$.
Integrating this in time we obtain the energy estimate $\|\sol \|^2\leq\|\sol _0\|^2$
and 
\eqref{HypSyst} is 
well-posed.
Since \eqref{HypSyst} is an one-dimensional hyperbolic problem it is also possible to show strong well-posedness, i.e. that $\|\sol \|$ is bounded by the data $\gl$, $\gr$, $\force $ and $\sol _0$. See \cite{KL,GKO} for different definitions of well-posedness.
\end{remark}

\subsection{The semi-discrete problem}

We discretize  in space 
  using $N+1$ equidistant grid points 
$x_i=\xl+hi$, where $h=(\xr-\xl)/N$ and $i=0,1,\hdots,N$. 
The semi-discrete scheme approximating \eqref{HypSyst} 
 is written
\begin{align}
\label{HypSystDisc}
\begin{split}\nsol _t+(I_N\otimes \hypR)\nsol +(D_1\otimes \hypA)\nsol =\fh &+(\PH ^{-1}e_\noll \otimes\SigL )(\hypBL{\nsol }_\noll -\gl)\\
&+(\PH ^{-1}e_N \otimes\SigR )(\hypBR{\nsol }_N -\gr),
\end{split}
\end{align}
where
$\nsol =[\nsol _0^T,\nsol _1^T,\hdots,\nsol _N^T]^T$ is a
 vector  of length $n(N+1)$, such that   $\nsol _i(t)\approx \sol (x_i,t)$,
 and
where 
$\fh_i(t)=\force (x_i,t)$. 
The symbol $\otimes$ refers to the Kronecker product. 
The finite difference operator $D_1$ approximates 
$\partial/\partial x$ and
satisfies the SBP-properties
\begin{align}\label{SBPprop1}
\begin{array}{lclcl}\vspace{4pt}
D_1=\PH ^{-1}Q,&&\hspace{5.5pt}\PH =\PH ^T>0,&&Q+Q^T=E_N-E_0\\
\end{array}
\end{align}
where  $E_0=e_0e_0^T$, $E_N=e_Ne_N^T$ and 
$e_0=[1, 0, \hdots, 0]^T$ and $e_N=[ 0, \hdots, 0, 1]^T$.
Note that
$\nsol _0=(e_0^T\otimes I_n)\nsol $ and $\nsol _N=(e_N^T\otimes I_n)\nsol $.
By
$I_N$ and $I_n$
we refer to
  identity matrices of size $N+1$ and $n$, respectively.
The boundary conditions 
are imposed 
using the 
SAT technique
which is a penalty method.
The penalty parameters $\SigL$ and $\SigR$ in \eqref{HypSystDisc} 
are at this point unknown, but  are derived in the next subsections and presented in 
Theorem~\ref{HypSystPen}.

In this paper, 
we require 
that $\PH $ is diagonal,
and in this case $D_1$ consists of a $2p$-order accurate central difference approximation in the interior 
and one-sided, $p$-order 
accurate approximations at the boundaries.
Examples  of SBP
 operators 
can  be found  in \cite{ref:STRA94, Mattsson2004503}. 
For 
more details about
SBP-SAT, see \cite{ref:SVAD06} and references therein.

\subsection{Numerical stability using the energy method}

Just as in the continuous case we use the energy method to show stability.
We multiply \eqref{HypSystDisc} by 
$\nsol ^T\Pbar$
 from the left, where $\Pbar=\PH \otimes I_n$, and then add the  transpose of the result. 
Thereafter using the
SBP-properties in \eqref{SBPprop1} 
we obtain
\begin{align*}
\ddt\| \nsol \|^2_\discmark+2{\nsol }^T(\PH \otimes \hypR){\nsol }&=2\inprod{\nsol , \fh }_\discmark+\text{BT}_L^{Disc.}+\text{BT}_R^{Disc.},
\end{align*}
 where
 $\|\nsol \|^2_\discmark=\inprod{ \nsol ,\nsol }_\discmark=\nsol ^T\Pbar \nsol $
is the discrete $L^2$-norm and where
 \begin{align}\label{HypBTInnan}
\begin{split}
\text{BT}_L^{Disc.}&=\nsol _\noll^T\left(\hspace{9pt}\hypA+\SigL \hspace{2pt} \hypBL\hspace{1pt}+\hspace{1pt}\hypBL^T\SigL^T \hspace{3pt}\right){\nsol }_\noll\hspace{2pt} -\nsol _\noll^T \SigL   \hspace{2pt}\gl  \hspace{1pt}-   \gl^T\SigL^T \hspace{2pt}\nsol _\noll,\\
\text{BT}_R^{Disc.}&=\nsol _N^T\left(- \hypA+\SigR \hypBR+\hypBR^T\SigR^T\right) \nsol _N-\nsol _N^T\SigR \gr-   \gr^T\SigR^T\nsol _N.
\end{split}
\end{align}
We define
$\BTmat_0= \hypA+\SigL  \hypBL+\hypBL^T\SigL^T$ and $ \BTmat_N= - \hypA+\SigR \hypBR+\hypBR^T\SigR^T$.
For stability 
 $\text{BT}_L^{Disc.}$ and $\text{BT}_R^{Disc.}$ must be non-positive for zero boundary
  data, i.e. $\BTmat_0\leq0$ and $ \BTmat_N\leq0$. 
 We make the following ans\"{a}tze for the penalty parameters:
\begin{align}\label{HypPrimalPenAnsatz}
\SigL&=(\X_+\cpL+\X_-\cmL)\JL^{-1},
&
\SigR&=(\X_+\cpR+\X_-\cmR)\JR^{-1}.
\end{align}
Taking the left boundary as example 
and
using \eqref{rot}, \eqref{RLrelation} and \eqref{HypPrimalPenAnsatz} 
we obtain %
  \begin{align}\label{stabXi0}
  \BTmat_0
  &=\left[\begin{array}{c}\X_+^T\\\X_-^T\end{array}\right]^T\left[\begin{array}{cc}\Rot_++\cpL+\cpL^T&\cpL\RL+\cmL^T\\
\cmL+  \RL^T\cpL^T&\Rot_-+\cmL\RL+\RL^T\cmL^T\end{array}\right]\left[\begin{array}{c}\X_+^T\\\X_-^T\end{array}\right].
\end{align}

 \subsection{The dual problem}

Given the functional  %
$\func(\sol )=\inprod{ \weight ,\sol }  $, 
the
dual problem of \eqref{HypSyst} is
\begin{align}
\label{HypSystDual}
\begin{array}{rll}\dualvar_\tau+\hypR \dualvar- \hypA \dualvar_x=&\hspace{-7pt}\weight ,\hspace{20pt}&x\in[\xl,\xr],\hspace{20pt}\vspace{4pt}\\
\dual{\hypBL} \dualvar=&\hspace{-7pt}\gldual,&x= \xl,\vspace{4pt}\\
\dual{\hypBR} \dualvar=&\hspace{-7pt}\grdual,&x= \xr,
\end{array}
\end{align}
which holds for $\tau\geq0$ and is complemented  with the initial condition $\dualvar(x,0)=\dualvar_0(x)$.
The boundary 
operators in \eqref{HypSystDual} 
have the form
 \begin{align} \label{RLrelationDual} \hypBLdual=\JLdual(\X_-^T+\RLdual\X_+^T),&& \hypBRdual=\dual{\JR}(\X_+^T+\dual{\RR}\X_-^T), \end{align}
where $\JLdual$ and $\dual{\JR}$ are arbitrary invertible matrices and
 $\RLdual$ and $\dual{\RR}$ depend on the primal boundary conditions as
 \begin{align}\label{PrimalDualBCrel}
 \RLdual=-\Rot_-^{-1}\RL^T\Rot_+,&&
 \dual{\RR}=-\Rot_+^{-1}\RR^T\Rot_-.\end{align}
The claim that \eqref{HypSystDual}, \eqref{RLrelationDual} and \eqref{PrimalDualBCrel} describes the dual problem 
is motivated below:
Using the notation in \eqref{primalprototype} and \eqref{dualprototype} we identify the spatial operators of \eqref{HypSyst} and \eqref{HypSystDual} as
\begin{align}
\label{LbothHyp}
\linop=\hypR+\hypA\frac{\partial}{\partial x},&&
\linop^*=\hypR-\hypA\frac{\partial}{\partial x},
\end{align}
respectively.
For \eqref{HypSystDual} to be the dual problem of \eqref{HypSyst}, 
$\linop $ and $\linop ^*$ must fulfill the relation $\inprod{\dualvar,\linop \sol }=\inprod{ \linop ^*\dualvar,\sol }$.
Using integration by parts we obtain
\begin{align}
\inprod{\dualvar,\linop \sol }
&=\inprod{ \linop ^*\dualvar,\sol }+[\dualvar^T\hypA \sol ]_{\xl}^{\xr}\notag
\end{align}
and %
we see that
$\dualvar ^T\hypA \sol  ^{}$
must be zero at both boundaries.
(The boundary conditions for the dual problem
are defined as the minimal set of homogeneous conditions such that all boundary terms vanish after  that that the  homogeneous boundary conditions for the primal problem have been applied, see \cite{Berg20126846}.)
Using the boundary conditions of the primal problem, \eqref{charbc}, %
followed by the dual boundary conditions, \eqref{HypSystDual}, \eqref{RLrelationDual}, yields  (for zero data)
\begin{align}
\left.\dualvar ^T\hypA \sol  ^{}\right|_{\xl}&=\left.-\dualvar^T\X^{ }_+\left(\Rot^{ }_+\RL+\RLdual^T\Rot^{ }_-\right)\X_-^T\sol \right|_{\xl}\notag\\
\left.\dualvar ^T\hypA \sol  ^{}\right|_{\xr}&=-\left.\dualvar^T\X^{ }_-\left(\dual{\RR}^T\Rot^{ }_++\Rot^{ }_-\RR\right)\X_+^T\sol \right|_{\xr}\notag\end{align}and 
 if 
\eqref{PrimalDualBCrel} holds, 
then  
$\dualvar ^T\hypA \sol  ^{}=0$
 at both boundaries and the above claim is confirmed.

 \begin{remark}
A functional can also include outgoing solution terms from the boundary, as $\func(\sol )=\inprod{ \weight ,\sol } +\JutR\Rot_+^{ }\X_+^T\sol |_{\xr}+\JutL\Rot_-^{ }\X_-^T\sol |_{\xl} $.  This would specify the  boundary data in \eqref{HypSystDual} to
 $\gldual=-\JLdual\JutL^T$ and $\grdual=\dual{\JR}^{}\JutR^T$, compare with \cite{HickenNo17}. For simplicity we  consider 
 $\func(\sol )=\inprod{ \weight ,\sol }$ 
 which means that
actually, the boundary data in \eqref{HypSystDual}  is zero.
 \end{remark}

\subsubsection{Well-posedness of the dual problem}

The growth rate for the dual problem is given by
\begin{align*}
\frac{d}{d\tau}\|\dualvar \|^2+2\inprod{\dualvar,\hypR \dualvar }=\text{BT}_L^{dual}+\text{BT}_R^{dual}
\end{align*}
where
 the boundary terms are 
 (the homogeneous boundary conditions  have been applied)
\begin{align*}\text{BT}_L^{dual}=\dualvar^T\X^{ }_+\dualCL  \left.\X_+^T \dualvar\right|_{\xl},&&\text{BT}_R^{dual}=\dualvar^T\X_-^{ }\dualCR \left.\X_-^T \dualvar\right|_{\xr} \end{align*}
and
where
$\dualCL  =-\Rot^{ }_+-\Rot_+\RL\Rot_-^{-1}\RL^T\Rot_+$ and $\dualCR =\Rot^{ }_- +\Rot_-\RR\Rot_+^{-1}\RR^T\Rot_-$.
For well-posedness of the dual problem $\dualCL  \leq0$ and $\dualCR \leq0$ are necessary.

Recall that the primal problem is well-posed if $\CL $, $\CR \leq0$. 
The dual demand
$\dualCL  \leq0$ is directly fulfilled if $\CL \leq0$ and 
 $\dualCR \leq0$ follows from $\CR \leq0$.
(In the special case when $\RL,\RR$ are square, invertible matrices, this is trivial. For general $\RL,\RR$ it can be shown
with the help of Sylvester's determinant theorem.)
We conclude that the dual problem \eqref{HypSystDual} with \eqref{RLrelationDual}, \eqref{PrimalDualBCrel} is well-posed if the primal problem \eqref{HypSyst} with \eqref{RLrelation} is well-posed.

\begin{remark}

In \cite{Berg201341,Berg2014135}
the  dual consistent schemes are constructed by first designing the boundary conditions (for incompletely parabolic problems) such that both the primal and the dual problem are well-posed. 
Their different approach can partly be explained by 
their wish to have the boundary conditions in the special form $H_{L,R}U \mp BU_x=G_{L,R}$.
Looking e.g. at Eq. (30) in \cite{Berg201341}, we note that after applying the boundary conditions, $U ^TM_LU \geq0$ is needed for stability. However, if $B$ is singular, 
replacing  $BU _x$ by $\pm H_{L,R}U $ 
 does not guarantee that all conditions have been completely used, and $u$ and $p$ in $U =[p,u]^T$ in $U ^TM_LU$ can be linearly dependent. Therefore the demand $M_L\geq0$ in (31) is  unnecessarily strong and gives some extra restrictions on the boundary conditions.
\end{remark}

\subsubsection{Discretization of the dual problem}

The semi-discrete scheme approximating the dual problem \eqref{HypSystDual} 
 is written 
\begin{align}
\label{HypSystDiscDual}
\begin{split}\dualdisc_\tau+(I_N\otimes \hypR)\dualdisc-(D_1\otimes \hypA)\dualdisc=\gh &+(\PH ^{-1}e_0 \otimes\dual{\SigL} )(\dual{\hypBL}{\dualdisc}_0 -\gldual)\\
&+(\PH ^{-1}e_N \otimes\dual{\SigR} )(\dual{\hypBR}\dualdisc_N -\grdual),
\end{split}
\end{align}
where   
 $\dualdisc_i(\tau)$ represents $\dualvar(x_i,\tau)$.
The SAT parameters $\dual{\SigL}$ and $\dual{\SigR}$ 
are yet unknown.

\subsection{Dual consistency}
\label{SecDualConsHyp}

The semi-discrete scheme \eqref{HypSystDisc} is rewritten as $\nsol _t+\lindisc \nsol =\RHS$,
where  
\begin{align}
\lindisc 
&=(I_N\otimes \hypR)+(D_1\otimes \hypA)-(\PH ^{-1}E_0 \otimes\SigL  \hypBL )-(\PH ^{-1}E_N \otimes\SigR \hypBR)\notag\end{align}
and where 
$\RHS$ only depends on known data.
In contrast to the  continuous counterpart $\linop$, $\lindisc $ 
 includes the boundary conditions explicitly. 
 According to \cite{Berg201341},
the discrete adjoint operator is 
given by $\lindisc ^*= \Pbar^{-1}\lindisc ^T\Pbar$, 
which, using \eqref{SBPprop1}, leads to
\begin{align}
\label{Lh*comp}\begin{split}
\lindisc ^*
&\hspace{-1pt}=\hspace{-1pt}(I_N\otimes \hypR)\hspace{-1pt}-\hspace{-1pt}(D_1\otimes \hypA)\hspace{-1pt}-\hspace{-1pt}(\PH ^{-1}E_0 \otimes \hypBL^T\SigL ^T+\hypA )\hspace{-1pt}-\hspace{-1pt}(\hspace{-1pt}\PH ^{-1}E_N \otimes \hypBR^T\SigR^T\hspace{-1pt}-\hypA)\end{split}
\end{align}
If $\lindisc ^*$  is a consistent approximation of $\linop ^*$ in \eqref{LbothHyp}, then the scheme \eqref{HypSystDisc} is dual consistent.
Looking at \eqref{HypSystDiscDual}, we see that $\lindisc ^*$ must have the form
\begin{align}
\label{Lh*goal}
(\lindisc ^*)^{goal}
&=(I_N\otimes \hypR)-(D_1\otimes \hypA)-(\PH ^{-1}E_0 \otimes\dual{\SigL} \dual{ \hypBL} )-(\PH ^{-1}E_N \otimes\dual{\SigR}\dual{ \hypBR}).\end{align}
Thus we   have dual consistency if the expressions in \eqref{Lh*comp} and \eqref{Lh*goal} are equal.
 This gives us the following requirements:
\begin{align}
 \hypBL^T\SigL ^T+\hypA-\dual{\SigL} \dual{ \hypBL}=0&
&
  \hypBR^T\SigR^T-\hypA-\dual{\SigR}\dual{ \hypBR}=0.\notag
\end{align}
Similarly to
the penalty parameters \eqref{HypPrimalPenAnsatz} for the primal problem,
we make the ans\"{a}tze
\begin{align}\label{HypDualPenAnsatz}
\dual{\SigL}&=(\X_+\dual{\cmL}+\X_-\dual{\cpL})\dual{\JL}^{-1}&\dual{\SigR}&=(\X_+\dual{\cmR}+\X_-\dual{\cpR})\dual{\JR}^{-1}
\end{align}
for the penalty parameters of the dual problem.
We consider the left boundary
and
use
\eqref{HypPrimalPenAnsatz} and \eqref{HypDualPenAnsatz},
together with
\eqref{rot},
 \eqref{RLrelation} and
  \eqref{RLrelationDual},
to write
\begin{align*}
\hypBL^T\SigL ^T+\hypA-\dual{\SigL} \dual{ \hypBL}
&=\hspace{-1pt}\left[\hspace{-2pt}\begin{array}{c}\X_+^T\\\X_-^T\end{array}\hspace{-2pt}\right]^T\hspace{-2pt}\left[\hspace{-1pt}\begin{array}{cc}\Rot^{ }_++\cpL^T-\dual{\cmL}\RLdual&\cmL^T-\dual{\cmL}\\\RL^T\cpL^T-\dual{\cpL}\RLdual&\Rot^{ }_-+\RL^T\cmL^T-\dual{\cpL}\end{array}\hspace{-1pt}\right]\hspace{-2pt}\left[\hspace{-2pt}\begin{array}{c}\X_+^T\\\X_-^T\end{array}\hspace{-2pt}\right]
\end{align*}
which is zero if 
and only if
the four entries of the matrix 
 are zero.
These four demands are rearranged 
to the more convenient form
\begin{subequations} 
\label{kravL}
\begin{align}
\label{kravLnyprim}\cpL&=-\Rot^{ }_+-\Rot^{ }_+\RL\Rot^{-1 }_-\cmL\\
\label{kravLnycont}\RLdual&=-\Rot^{ -1}_-\RL^T\Rot^{ }_+\\
\label{kravLnyrel}\dual{\cmL}&=\cmL^T\\
\label{kravLnydual}\dual{\cpL}&=\Rot^{ }_--\Rot^{ }_-\RLdual\Rot^{-1 }_+\dual{\cmL}.
\end{align}
\end{subequations}
Note that \eqref{kravLnyprim} only depends on parameters from  the primal problem, 
while \eqref{kravLnydual} only depends on parameters from  the dual problem.
Interestingly enough, \eqref{kravLnycont} is 
nothing 
but
the duality demand  \eqref{PrimalDualBCrel} for the continuous problem.
The demand \eqref{kravLnyrel}  relates the penalty of the dual problem to the primal penalty.

Unless we actually want to solve the dual problem, it is enough to consider the first demand,
\eqref{kravLnyprim}. 
We repeat the above derivation also for the right boundary and get the following result: 
The penalty parameters %
$\SigL$ and $\SigR$
in  \eqref{HypPrimalPenAnsatz}  with
\begin{align}\label{DualOnlyPrimal}
\cpL=-\Rot^{ }_+-\Rot^{ }_+\RL\Rot^{-1 }_-\cmL,&&\cmR=\Rot^{ }_--\Rot_-\RR\Rot_+^{-1}\cpR,
\end{align}
makes the discretization \eqref{HypSystDisc} dual consistent.

\begin{remark}

If the discrete primal problem \eqref{HypSystDisc} is dual consistent there is no need to check if the discrete dual problem \eqref{HypSystDiscDual} is stable -- in \cite{HickenNo18} it is stated that
stability of the primal
problem implies stability of the dual problem, because the system matrix for the dual problem is the transpose of the system
matrix for the primal problem -- 
that is
 the primal and dual discrete problems have exactly the same growth rates for zero data.
\end{remark}

\subsection{Penalty parameters for the hyperbolic problem}

Consider the %
penalty parameter ansatz for the left boundary,
 $\SigL=(\X_+\cpL+\X_-\cmL)\JL^{-1}$, which is given in
\eqref{HypPrimalPenAnsatz}.
 From a stability point of view, we must choose $\cpL$ and $\cmL$  such that
$\BTmat_0$  in \eqref{stabXi0} 
becomes non-positive. 
In addition, 
for dual consistency the
constraint in \eqref{DualOnlyPrimal} must be 
fulfilled. 
By inserting the duality constraint $\cpL=-\Rot^{ }_+-\Rot^{ }_+\RL\Rot^{-1 }_-\cmL$ from \eqref{DualOnlyPrimal} into $\BTmat_0$ 
we obtain, after some rearrangements, the expression
  \begin{align*} \BTmat_0  
  &\hspace{-1pt}=\hspace{-1pt}
  \left[\hspace{-4pt}\begin{array}{c}\JL^{-1}\hypBL\\\X_-^T\end{array}\hspace{-4pt}\right]^T\hspace{-1pt}
 \left[\hspace{-2pt}\begin{array}{cc}-\Rot_+-\Rot^{ }_+\RL\Rot^{-1 }_-\cmL-(\Rot^{ }_+\RL\Rot^{-1 }_-\cmL)^T&\cmL^T\Rot^{-1 }_-\CL \\\CL \Rot^{-1 }_-\cmL&\CL 
\end{array}\hspace{-2pt}\right]\hspace{-1pt}
 \left[\hspace{-4pt}\begin{array}{c}\JL^{-1}\hypBL\\\X_-^T\end{array}\hspace{-4pt}\right].\end{align*}
The most obvious strategy to make $\BTmat_0\leq0$ is to cancel the off-diagonal entries by putting
$\cmL=0$,
but note that other choices exist. 
To single out the 
  optimal  (in a certain sense)
 candidate, we use another approach.
With \eqref{rot}, \eqref{RLrelation} and $\gsnokL=\JL^{-1}\gl$,
the left boundary term in \eqref{HypBTInnan} can be rearranged as
\begin{align}\label{BTrearranged}
\begin{split}
\text{BT}_L^{Disc.}
&=\nsol _0^T\X_-^{ }\left(\Rot^{ }_-+\RL^T\Rot^{ }_+\RL\right)\X_-^T{\nsol }_0-2\gsnokL^T\Rot^{ }_+\RL\X_-^T\nsol _0+\gsnokL^T\Rot^{ }_+\gsnokL\\&-\left(\hypBL \nsol _0-\gl\right)^T\JL^{-T}\Rot^{ }_+\JL^{-1}\left(\hypBL \nsol _0-\gl\right)\\&+2\left(\hypBL \nsol _0-   \gl\right)^T\left(\SigL+\X_+\Rot^{ }_+\JL^{-1}\right)^T\nsol _0
\end{split}
\end{align}
where we see that the first row corresponds exactly to the continuous boundary term  $\text{BT}_L$ in \eqref{ContBT}.
The second row is a damping term that is quadratically proportional 
to the 
solution's deviation from data
at the boundary, $\hypBL \nsol _0-\gl$.
The term in the last row is only linearly proportional to this deviation, so we would prefer  it to be zero. This is possible if the penalty parameter is chosen exactly as $\SigL=-\X_+\Rot^{ }_+\JL^{-1}$. 
Luckily  this choice  fulfills
both the stability requirement and
 the duality constraint. 
We repeat the above derivation also for the right boundary and 
summarize our findings in Theorem~\ref{HypSystPen}.
\begin{theorem}
\label{HypSystPen}
Consider the problem \eqref{HypSyst} with an associated factorization  \eqref{ALambdaX} where $\X$ is non-singular. 
With the particular choice of 
penalty parameters
\begin{align}\label{OptPenHyp}
\SigL&=-\X_+\Rot^{ }_+\JL^{-1}
,&
\SigR&=\X_-\Rot^{ }_-\JR^{-1},
\end{align}
the scheme \eqref{HypSystDisc} 
is 
a stable and dual consistent discretization of  \eqref{HypSyst}.
The matrices $\JL$ and $\JR$ are specified through \eqref{RLrelation}.
\end{theorem}

\begin{proof}
Comparing  with \eqref{HypPrimalPenAnsatz}, we note that 
$\SigL$
 in \eqref{OptPenHyp} 
is obtained using $\cpL=-\Rot_+$ and $\cmL=0$. 
These values 
fulfill the 
left
duality constraint 
in \eqref{DualOnlyPrimal}.
Inserting $\cmL=0$ into 
 $ \BTmat_0$ above,
we obtain
 $ \BTmat_0=\X_-\CL\X_-^T-\hypBL^T\JL^{-T}\Rot_+\JL^{-1}\hypBL$, which is negative semi-definite if the continuous problem is well-posed (in the 
  $\CL\leq0$ sense).
Thus the stability demand   $\BTmat_0\leq0$ is fulfilled. 
The same  is done for the right boundary, completing the proof.
\end{proof}

\begin{remark} \label{remarkNotDep}The seemingly very specific choice of penalty parameters in Theorem~\ref{HypSystPen}
 is, in fact, a family of penalty parameters, depending on the factorization used.
Note that it is not  
 necessary to use the same factorization for the left and the right boundary.
\end{remark}

\begin{remark}
If characteristic boundary conditions (in the sense $\RL, \RR=0$) 
are used, 
the  scheme  \eqref{HypSystDisc} together 
with the SATs from
Theorem~\ref{HypSystPen} simplifies to
\begin{align*}
\begin{split}\nsol _t+(I_N\otimes \hypR)\nsol +(D_1\otimes \hypA)\nsol &=\fh +(\PH ^{-1}E_0 \otimes-\hypA_+ )\nsol +(\PH ^{-1}E_N \otimes\hypA_- )\nsol 
\end{split}
\end{align*}
in the homogeneous case,
where
$\hypA_+=\X^{ }_+\Rot^{ }_+\X_+^T$ 
and $\hypA_-=\X_-^{ }\Rot^{ }_-\X_-^T$. 
When the factorization refers to the eigendecomposition, this corresponds to the SAT used for the characteristic boundary conditions of the nonlinear Euler equations in \cite{HickenNo17}.
\end{remark}

\section{Parabolic systems}
\label{SectPara}

Consider the parabolic (or 
 incompletely parabolic) system of partial differential equations 
\begin{align}
\label{ParaSyst}
\begin{array}{rll}\sol _t+ \paraA \sol _x-\paraE \sol _{xx}=&\hspace{-7pt}\force ,\hspace{20pt}&x\in[\xl,\xr],\hspace{20pt}\vspace{4pt}\\\HL \sol +\GL \sol _x=&\hspace{-7pt}\gl,&x= \xl,\vspace{4pt}\\\HR \sol +\GR \sol _x=&\hspace{-7pt}\gr,&x= \xr,\end{array}
\end{align}
for $t\geq0$, augmented with the initial condition $\sol (x,0)=\sol _0(x)$.
The matrices 
$\paraA$ and $\paraE\geq0$ are symmetric $\sizepara\times\sizepara$ matrices, and we assume that 
$\GL$ and $\GR$ scales 
as $\GL=\factorL\paraE$ and $\GR=\factorR\paraE$, respectively.
Treating $\sol _x$ as a separate variable,
we can
rewrite \eqref{ParaSyst} as a  first order system (as was also done in  \cite{HickenNo17,Berg20126846}), arriving at
\begin{align}
\label{ParaSystAsHypSyst}
\begin{array}{rlll}\bigI\bigU_t+\bigR\bigU+ \bigA \bigU_x=&\hspace{-7pt}\bigF,\hspace{20pt}&x\in[\xl,\xr],\hspace{20pt}\vspace{4pt}\\\bigBL \bigU=&\hspace{-7pt}\gl,&x= \xl,\vspace{4pt}\\\bigBR \bigU=&\hspace{-7pt}\gr,&x= \xr,\end{array}
\end{align}
where %
\begin{align}
\bigI=\left[\begin{array}{cc}I_\sizepara&0\\0&0\end{array}\right],&&
\bigR=\left[\begin{array}{cc}0&0\\0&\paraE
\end{array}\right],&&
\bigU=\left[\begin{array}{c}\sol \\\sol _x\end{array}\right],&&
\bigF=\left[\begin{array}{c}\force \\0\end{array}\right]%
\notag\end{align}
and
\begin{align}
\label{LiteViktigare}
\bigA=\left[\begin{array}{cc}\paraA&-\paraE\\-\paraE&0
\end{array}\right], &&\bigBL=\left[\begin{array}{cc}\HL&\GL
\end{array}\right],&&\bigBR=\left[\begin{array}{cc}\HR&\GR
\end{array}\right].
\end{align}
The system \eqref{ParaSystAsHypSyst} has almost 
the same form as \eqref{HypSyst} since $\bigR\geq0$ and $\bigA$ are symmetric $ \sizebig\times \sizebig$  matrices, where  $\sizebig= 2\sizepara$.
Thus we can use the results from the hyperbolic case.

\begin{remark}\label{RankDim}
In  \cite{Berg201341,Berg2014135} the  operators corresponding to $\HL$, $\GL$, $\HR$ and $\GR$ are square $\sizepara\times \sizepara$ 
matrices and their ranks are changed to suit the number of boundary conditions. We 
adapt the matrix  dimensions instead. Both approaches have their respective advantages.
\end{remark}

\subsection{Discretization using wide-stencil second derivative operators}

To discretize the parabolic problem, we first  consider the reformulated problem \eqref{ParaSystAsHypSyst}, and use the
results from the hyperbolic section. Then we  rearrange the terms such that we get an equivalent scheme but in 
a form corresponding to \eqref{ParaSyst}. 
These steps, which are done in
Appendix~\ref{ReformWide},  lead to
\begin{align}
\label{FinalWide}
\begin{split}\nsol _t
+(D_1\otimes \paraA)\nsol -(D_1
^2\otimes \paraE)\nsol =\fh &+\Pbar^{-1}(e_0 \otimes\wide\penI_0+D_1^Te_0 \otimes\wide\penS_0)\wide{\xi}_0\\
&+\Pbar^{-1}(e_N \otimes\wide\penI_N+D_1^Te_N \otimes\wide\penS_N)
\wide{\xi}_N
\end{split}
\end{align}
where 
\begin{align}
\label{xi}
\wide{\xi}_\noll&=\HL{\nsol }_\noll +\GL (\Dbar \nsol )_\noll -\gl,&\wide{\xi}_N&= \HR{\nsol }_N+\GR (\Dbar \nsol )_N -\gr,
\end{align}
and $\Pbar=(\PH \otimes I_\sizepara)$ and $\Dbar=(D_1\otimes I_\sizepara)$.
The penalty parameters in \eqref{FinalWide} are
\begin{align}\label{WIDEtauRel}
\begin{split}\wide\penI_\noll&=(-\bigX_1+\p\bigX_2)\bigRot^{ }_+(\bigJL+\p \factorL\bigX_2\bigRot^{ }_+ )^{-1},\hspace{10pt}\wide\penS_\noll=\hspace{9pt}\bigX_2\bigRot^{ }_+(\bigJL+\p \factorL\bigX_2\bigRot^{ }_+)^{-1}\\
\wide\penI_N&=\hspace{8pt}(\bigX_3+\p\bigX_4)\bigRot^{ }_-(\bigJR-\p\factorR \bigX_4\bigRot^{ }_-)^{-1},\hspace{7pt}\wide\penS_N=-\bigX_4\bigRot^{ }_-(\bigJR-\p\factorR \bigX_4\bigRot^{ }_-)^{-1}\end{split}
\end{align}
where the matrices  $\bigX_{1,2,3,4}$ are defined through 
\begin{align}
\label{X1234}
\bigX_+=\left[\begin{array}{c}\bigX_1\\\bigX_2\end{array}\right],&&\bigX_-=\left[\begin{array}{c}\bigX_3\\\bigX_4\end{array}\right].\end{align}
As before, $\bigRot^{ }_\pm$, $\bigX_\pm$ and $\bigJL$, $\bigJR$ are described in \eqref{ALambdaX} and \eqref{RLrelation}, respectively, but are now obtained using  $\bigA$ and $\bigBL$, $\bigBR$ from \eqref{LiteViktigare}. 
Finally, the quantity 
 $\p$ in \eqref{WIDEtauRel} is given by
\begin{align}
\label{defp}
\p=e_0^T\PH ^{-1}e_0=e_N^T\PH ^{-1}e_N.
\end{align}
The matrix $\PH $ is positive definite and proportional to the grid size $h$, and thus
 $\p$ is a positive scalar proportional to $1/h$.

\subsection{Discretization using narrow-stencil second derivative operators}

In \cite{Berg20126846}, it was 
suggested that dual consistency might require wide-stencil second derivative operators, 
but next
 we will show  that 
this is not necessary.
The  semi-discrete scheme approximating \eqref{ParaSyst}
is now  written,
analogously to \eqref{FinalWide}, as
\begin{align}
\label{PrimalCOMPACT}
\begin{split}\nsol _t+(D_1\otimes \paraA)\nsol -(D_2\otimes \paraE)\nsol =\fh &+\Pbar^{-1}(e_\noll \otimes\taul +S^Te_\noll \otimes\sigmal )\xi_\noll\\
&+\Pbar^{-1}(e_N \otimes\taur+S^Te_N \otimes\sigmar)
\xi_N.
\end{split}
\end{align}
The operator $D_2$, which approximates the  second derivative operator, is 
no longer 
limited to the previous form $D_1^2$,  where the   first derivative is used twice. However,  $D_2$ must still fulfill the 
SBP  relations
\begin{align}\label{SBPprop2}
D_2=\PH ^{-1}(-\compA+(E_N-E_0)S),&&\hspace{23pt}\compA=\compA^T=S^TMS\geq0.
\end{align}
The first and last row
of the matrix
$S$  are consistent difference stencils, see e.g.  \cite{Mattsson2004503}. For 
dual consistency, $\compA$ must be symmetric.
Further, we have
\begin{align}
\label{xiCOMPACT}
\xi_\noll&=\HL{\nsol }_0 +\GL (\Sbar \nsol )_0 -\gl,&\xi_N&= \HR{\nsol }_N+\GR (\Sbar \nsol )_N -\gr,
\end{align}
where
\begin{align}
\Sbar=S\otimes I_\sizepara,&&
(\Sbar  \nsol )_0=(e_0^TS\otimes I_\sizepara)\nsol ,&&(\Sbar  \nsol )_N=(e_N^TS\otimes I_\sizepara)\nsol .
\notag
\end{align}
We also define
\begin{align}
\label{qdef}
\q\equiv\q_0+|\q_c|=\q_N+|\q_c|
\end{align}where\begin{align}
\label{qdefparts}
\q_0=e_0^TM^{-1}e_0,&&\q_N=e_N^TM^{-1}e_N,&& \q_c=e_0^TM^{-1}e_N=e_N^TM^{-1}e_0,
\end{align}
and where $M$ is a part of $D_2$ as stated in \eqref{SBPprop2}.
In Section~\ref{AppRinv} we provide
 $\q$  for various $D_2$ matrices.
The penalty parameters $\taul$, $\sigmal$, $\taur$ and $\sigmar$ in \eqref{PrimalCOMPACT}
are  now given by: 
 
\begin{theorem}
\label{ParaSystPen}
Consider the problem \eqref{ParaSyst}
 with $\GL=\factorL\paraE$ and $\GR=\factorR\paraE$.
Further, let $\bigA$, which is specified in \eqref{LiteViktigare}, be factorized as $\bigA=\bigX\bigRot \bigX^T$ as described in \eqref{ALambdaX}.
Then 
the particular choice of 
penalty parameters
\begin{align}
\label{TauUtanMu}
\begin{split}
\taul&=(-\bigX_1+\q\bigX_2)\bigRot^{ }_+
(\bigJL+\q \factorL\bigX_2\bigRot^{ }_+ )^{-1},\hspace{9pt}\sigmal=\hspace{11pt}\bigX_2\bigRot^{ }_+(\bigJL+\q \factorL\bigX_2\bigRot^{ }_+)^{-1}
\\
\taur&=\hspace{8pt}(\bigX_3+\q\bigX_4)\bigRot^{ }_-
(\bigJR-\q\factorR \bigX_4\bigRot^{ }_-)^{-1},\hspace{7pt}
\sigmar=-\bigX_4\bigRot^{ }_-(\bigJR-\q\factorR \bigX_4\bigRot^{ }_-)^{-1}
\end{split}
\end{align}
makes the scheme in \eqref{PrimalCOMPACT} stable and dual consistent.
The matrices $\bigX_{1,2,3,4}$ are given in
\eqref{X1234},
$\bigJL$, $\bigJR$ are obtained from \eqref{RLrelation} (using $\bigBL$, $\bigBR$ in \eqref{LiteViktigare}) and
$\q$ is defined in \eqref{qdef}. %
\end{theorem}

Note that 
$\q$ in \eqref{qdef} is a 
generalization of $\p$ in \eqref{defp},
and that the penalty parameters in \eqref{TauUtanMu} 
and \eqref{WIDEtauRel} 
 are identical if $\q=\p$.
Hence the narrow-stencil scheme \eqref{PrimalCOMPACT} is
a generalization of the wide-stencil scheme in \eqref{FinalWide}, since the schemes are identical if we choose $D_2=D_1^2$, $S=D_1$ and $M=\PH $.
In the rest of this section we will 
justify these generalizations and
prove Theorem~\ref{ParaSystPen} by showing that
 the penalties given in 
\eqref{TauUtanMu}
 indeed make the scheme
\eqref{PrimalCOMPACT} stable and dual consistent.

\subsection{Stability when using narrow-stencil second derivative operators}

We multiply the scheme \eqref{PrimalCOMPACT} by $\nsol ^T\Pbar$ from the left and add the  transpose of the result. 
Thereafter 
using the
SBP-properties in \eqref{SBPprop1} and \eqref{SBPprop2} 
yields
\begin{align}
\label{FirstEnergyCompact}
\ddt\| \nsol \|^2_\discmark+2{\nsol }^T(S^TMS \otimes \paraE){\nsol }&=2\inprod{\nsol , \fh }_\discmark+\text{BT}_L^{Disc.}+\text{BT}_R^{Disc.},\end{align}
 where
\begin{align}
\label{FirstBTCompact}
\begin{split}
\text{BT}_L^{Disc.}&=\hspace{10pt}\nsol _0^T\paraA{\nsol }_0 \hspace{2pt}-\hspace{1pt}2{\nsol }_0^T\paraE(\Sbar{\nsol })_0 \hspace{2pt}+\hspace{1pt}2(\nsol _0^T \taul \hspace{2pt}+(\Sbar \nsol )^T_0 \sigmal \hspace{2pt})\xi_\noll\\
\text{BT}_R^{Disc.}&=-{\nsol }_N^T\paraA{\nsol }_N+2{\nsol }_N^T\paraE(\Sbar{\nsol })_N+2(\nsol _N^T\taur+(\Sbar \nsol )_N^T\sigmar)\xi_N
\end{split}\end{align}
where $\xi_{0,N}$ are given in \eqref{xiCOMPACT}.
If $\text{BT}_L^{Disc.}$ and $\text{BT}_R^{Disc.}$ are non-positive for zero data the scheme is stable.
This can be achieved
if $\taul$, $\sigmal$, $\taur$ and $\sigmar$ are chosen freely, 
but the scheme should also be dual consistent. 
It turns out 
that in some cases these requirements are
impossible to combine, for example when having 
 Dirichlet boundary conditions. 
We therefore need an alternative way to show stability.

\label{StabComp}

First, we  assume 
that the penalty parameters $\sigmal$ and $\sigmar$ 
scales with $\paraE$.
Let 
 \begin{align}\label{wpensigma}\sigmal=-\paraE\wpen_0,&&
 \sigmar=-\paraE\wpen_N.\end{align}
Next, we take a look at the wide case (which is partly presented in Appendix~\ref{ReformWide}). 
Using a wide counterpart to \eqref{wpensigma}, $\wide\penS_0=-\paraE\wide\wpen_0$ and $\wide\penS_N=-\paraE\wide\wpen_N$, and the later relations in \eqref{xichirel} and \eqref{WIDEtaumuRel},
we can rewrite \eqref{Approach2b} 
as
\begin{align*}
\wide{W}
&=\Dbar \nsol +(\PH ^{-1}e_0 \otimes\wide\wpen_0)\wide\xi_\noll+(\PH ^{-1}e_N \otimes\wide\wpen_N )\wide\xi_N.
\end{align*}
We return to  the narrow-stencil scheme \eqref{PrimalCOMPACT}. 
Inspired by
the wide case, we define 
\begin{align}\label{WdefCompact}
W&\equiv\Sbar \nsol +( M^{-1} e_0\otimes \wpen_0)\xi_\noll+(M^{-1}  e_N\otimes\wpen_N)\xi_N.
\end{align} 
From \eqref{WdefCompact} we compute
\begin{align}
W^T(M\otimes \paraE)W
=\nsol ^T(S^TMS\otimes \paraE)\nsol &+\left(2(\Sbar \nsol )_0 \hspace{3pt}+\q_0\hspace{1pt}\wpen_0\hspace{1pt}\xi_\noll\hspace{3pt}+\q_c\wpen_N\xi_N\right)^T\paraE\wpen_0\hspace{2pt}\xi_\noll\notag\\
&+\left(2(\Sbar \nsol )_N +\q_N\wpen_N\xi_N+\q_c\wpen_0\hspace{1pt}\xi_\noll\right)^T\paraE\wpen_N\xi_N\notag
\end{align}
where $\q_0$, $\q_N$ and $\q_c$ are given in \eqref{qdefparts}.
%
In the general case, $\q_c$ 
can be non-zero.
Since we want to treat the two boundaries separately, 
we use Young's inequality, 
$\q_c(\xi_N^T\wpen_N^T\paraE\wpen_0\xi_\noll+\xi_\noll^T\wpen_0^T\paraE\wpen_N\xi_N)\leq|\q_c|\left(\xi_\noll^T\wpen_0^T\paraE\wpen_0\xi_\noll
+\xi_N^T\wpen_N^T\paraE\wpen_N\xi_N\right)$, which leads to
\begin{align}
\label{LessThanComp}
\begin{split}
W^T(M\otimes \paraE)W
\leq \nsol ^T(S^TMS\otimes \paraE)\nsol &+\left(2(\Sbar \nsol )_0 \hspace{3pt}+\q \wpen_0\hspace{1pt}\xi_\noll\hspace{1pt}\right)^T\paraE\wpen_0\hspace{1pt}\xi_\noll\\
&+\left(2(\Sbar \nsol )_N +\q\wpen_N\xi_N\right)^T\paraE\wpen_N\xi_N\end{split}
\end{align}
where $\q=\q_0+|\q_c|=\q_N+|\q_c|$, as stated in \eqref{qdef}.
Further, we note that multiplying \eqref{WdefCompact} by $(e_0^T\otimes I_\sizepara)$ and $(e_N^T\otimes I_\sizepara)$, respectively, yields the relations
$W_0= (\Sbar \nsol )_0+ \q_0 \wpen_0\xi_\noll+\q_c \wpen_N\xi_N$ and $W_N=(\Sbar \nsol )_N+\q_c\wpen_0\xi_\noll+\q_N \wpen_N\xi_N$.
Instead of using those, which contain unwanted terms from the other boundary, 
we define
\begin{align}
\label{W0NdefCOMP}
\Wtilde_0&\equiv(\Sbar \nsol )_0+\q \wpen_0\xi_\noll&\Wtilde_N&\equiv(\Sbar \nsol )_N+\q \wpen_N\xi_N.
\end{align}
Inserting the relation \eqref{LessThanComp} into  \eqref{FirstEnergyCompact}, we obtain 
\begin{align}
\label{SecondEnergyCompact}
\ddt\| \nsol \|^2_\discmark
+2W^T(M\otimes \paraE)W&\leq2\inprod{\nsol , \fh }_\discmark+ \BTsnok_L+ \BTsnok_R\end{align}
 where 
  \eqref{FirstBTCompact} and \eqref{W0NdefCOMP} together with \eqref{wpensigma} yields
 \begin{align}\label{NyaBT}
 \begin{split}
\BTsnok_L&=\hspace{10pt}\nsol _\noll^T\paraA{\nsol }_\noll \hspace{1pt}-2{\nsol }_\noll^T\paraE
 \Wtilde_\noll
 \hspace{1pt}+2(\nsol _\noll^T (\taul \hspace{1pt}-\q\sigmal\hspace{1pt})\hspace{3pt}-\Wtilde_\noll
^T\hspace{1pt} \sigmal\hspace{1pt})\xi_\noll\\
\BTsnok_R&=-{\nsol }_N^T\paraA{\nsol }_N+2{\nsol }_N^T\paraE
 \Wtilde_N
 +2(\nsol _N^T(\taur+\q\sigmar)-\Wtilde_N^T\sigmar)\xi_N.
 \end{split}
 \end{align}
If  the penalty parameters make $\BTsnok_L\leq0$ and $\BTsnok_R\leq0$ for 
  zero data, \eqref{PrimalCOMPACT} is stable.

  
Again taking the left boundary as an example, we
define $\VCOMP_0=[\nsol _0^T,\Wtilde_0^T]^T$
and write the first part of $ \BTsnok_L$ in \eqref{NyaBT} as
\begin{align}\label{part1}
\nsol _0^T\paraA{\nsol }_0 -2{\nsol }_0^T\paraE\Wtilde_0=\VCOMP_0^T\bigA\VCOMP_0.
\end{align}
Next, using
the relations \eqref{LiteViktigare}, \eqref{W0NdefCOMP} and \eqref{xiCOMPACT}, 
 recalling the assumptions
 $\GL=\factorL\paraE$ and  $\sigmal=-\paraE\wpen_0$,
 and thereafter using 
\eqref{TauUtanMu} from Theorem~\ref{ParaSystPen},
we obtain
\begin{align}\label{KVg}\begin{split}
\bigBL\VCOMP_0-\gl
&=\bigJL(\bigJL+\q \factorL\bigX_2\bigRot^{ }_+)^{-1}\xi_\noll.
\end{split}
\end{align} 
 From \eqref{TauUtanMu} we also get
\begin{align*}
\taul -\q\sigmal=-\bigX_1\bigRot^{ }_+
(\bigJL+\q \factorL\bigX_2\bigRot^{ }_+ )^{-1},&&-\sigmal=-\bigX_2\bigRot^{ }_+(\bigJL+\q \factorL\bigX_2\bigRot^{ }_+)^{-1}
\end{align*}
such that the second part of $\BTsnok_L$ in \eqref{NyaBT} becomes
  \begin{align}\label{part2}\begin{split}
2\left(\nsol _0^T (\taul -\q\sigmal)-\Wtilde_0
^T \sigmal\right)\xi_\noll
&=2\VCOMP_0^T\bigSigL(\bigBL\VCOMP_0-\gl)
\end{split}
 \end{align}
 where the relations \eqref{X1234} and \eqref{KVg} have been used,
 and where $\bigSigL=-\bigX_+\bigRot^{ }_+\bigJL^{-1}$. 
Now we can, by inserting \eqref{part1} and \eqref{part2} into  \eqref{NyaBT},  write
 \begin{align*}
  \begin{split}
\BTsnok_L&=\VCOMP_0^T\bigA\VCOMP_0+2\VCOMP_0^T\bigSigL(\bigBL\VCOMP_0-\gl)
 \end{split}
 \end{align*}
which  has exactly  the same form as $\text{BT}_L^{Disc.}$ in \eqref{HypBTInnan}.
We thus know that $\BTsnok_L\leq0$ for zero data,
since $\bigSigL$ is computed just as in the hyperbolic case. 
The same procedure can, of course, be repeated for the right boundary. 
We conclude that the scheme \eqref{PrimalCOMPACT} with the penalty parameters \eqref{TauUtanMu} is 
stable.

\subsection{Dual consistency for narrow-stencil second derivative operators}

The dual problem of \eqref{ParaSyst} is
\begin{align}
\label{ParaSystDual}
\begin{array}{rll}\dualvar_\tau
- \paraA \dualvar_x-\paraE \dualvar_{xx}=&\hspace{-7pt}\weight ,\hspace{20pt}&x\in[\xl,\xr],\hspace{20pt}\vspace{4pt}\\
\dual{\HL} \dualvar+\dual{\GL} \dualvar_x=&\hspace{-7pt}\gldual,&x= \xl,\vspace{4pt}\\
\HRdual \dualvar+\dual{\GR} \dualvar_x=&\hspace{-7pt}\grdual,&x= \xr,\end{array}
\end{align}
for $\tau\geq0$ and with $\dualvar(x,0)=\dualvar_0(x)$.
The 
spatial operator in \eqref{ParaSyst} and its dual are thus
\begin{align}
\label{LbothSYST}
\linop =\paraA\frac{\partial}{\partial x}-\paraE\frac{\partial^2}{\partial x^2},
&&
\linop ^*=-\paraA\frac{\partial}{\partial x}-\paraE\frac{\partial^2}{\partial x^2}.
\end{align}
The semi-discrete  approximation of \eqref{ParaSystDual}
is
\begin{align}
\label{DualCOMPACT}
\begin{split}\dualdisc_\tau-(D_1\otimes \paraA)\dualdisc-(D_2\otimes \paraE)\dualdisc=\gh &+\Pbar^{-1}(e_0\hspace{2pt}  \otimes\dual{\taul}\hspace{2pt} +S^Te_0\hspace{2pt}  \otimes\dual{\sigmal}\hspace{2pt}  )\dual{\xi_\noll}\\
&+\Pbar^{-1}(e_N \otimes\dual{\taur}+S^Te_N \otimes\dual{\sigmar})
\dual{\xi_N},
\end{split}
\end{align}
where
\begin{align*}
\dual{\xi_\noll}&=\dual{\HL}\dualdisc_0 +\dual{\GL} (\Sbar \dualdisc)_0 -\gldual,&\dual{\xi_N}&= \HRdual\dualdisc_N+\dual{\GR }(\Sbar \dualdisc)_N -\grdual.
\end{align*}
From \eqref{PrimalCOMPACT} we see that
the discrete operator, corresponding to $\linop$ in \eqref{LbothSYST}, is 
\begin{align}
\label{LhDiscCompact}
\begin{split}\lindisc =(D_1\otimes \paraA)-(D_2\otimes \paraE)&-\Pbar^{-1}(e_0 \otimes\taul +S^Te_0 \otimes\sigmal )(e_0^T\otimes\HL + e_0^TS\otimes\GL)\\&-\Pbar^{-1}(e_N \otimes\taur+S^Te_N \otimes\sigmar)(e_N^T\otimes\HR+e_N^TS\otimes\GR ).\end{split}
\end{align}
Using the relations in \eqref{SBPprop1} and \eqref{SBPprop2}, we obtain
\begin{align}
\lindisc ^*= \Pbar^{-1}\lindisc ^T\Pbar&=-(D_1\otimes \paraA)-(D_2\otimes \paraE)\notag\\
&-\Pbar^{-1}(e_0e_0^T\otimes \paraA)+\Pbar^{-1}\left((S^Te_0e_0^T-e_0e_0^TS)\otimes \paraE\right)\notag\\
&+\Pbar^{-1}(e_Ne_N^T\otimes\paraA)-\Pbar^{-1}((S^Te_Ne_N^T-e_Ne_N^TS)\otimes\paraE)\notag\\
&-\Pbar^{-1}(e_0\otimes\HL^T+ S^Te_0\otimes\GL^T)(e_0^T \otimes\taul^T +e_0^TS \otimes\sigmal^T )\notag\\
&-\Pbar^{-1}(e_N\otimes\HR^T+S^Te_N\otimes\GR^T )(e_N^T \otimes\taur^T+e_N^TS \otimes\sigmar^T).\notag
\end{align}
However, from \eqref{DualCOMPACT}
we see that for dual consistency $\lindisc ^*$ must have the form
\begin{align}
(\lindisc ^*)^{goal}&=-(D_1\otimes \paraA)-(D_2\otimes \paraE)\notag\\
&-\Pbar^{-1}(e_0 \otimes \dual{\taul} +S^Te_0 \otimes \dual{\sigmal} )(e_0^T\otimes\dual{\HL} + e_0^TS\otimes\dual{\GL})\notag\\&
-\Pbar^{-1}(e_N \otimes\dual{\taur}+S^Te_N \otimes\dual{\sigmar})(e_N^T\otimes\HRdual+e_N^TS\otimes\dual{\GR }).\notag
\end{align}
Demanding that $\lindisc ^*=(\lindisc ^*)^{goal}$, gives us the duality constraints
\begin{align}\label{ExplicitDualCOMP}\begin{split}
\left[\begin{array}{cc}\HL^T\taul^T+\paraA&\HL^T\sigmal^T+\paraE\\\GL^T\taul^T-\paraE&\GL^T\sigmal^T\end{array}\right]&=\left[\begin{array}{cc} \dual{\taul}\dual{\HL}& \dual{\taul}\dual{\GL}\\ \dual{\sigmal}\dual{\HL}& \dual{\sigmal}\dual{\GL}\end{array}\right]\\
\left[\begin{array}{cc}\HR^T\taur^T-\paraA&\HR^T\sigmar^T-\paraE\\\GR^T\taur^T+\paraE&\GR^T\sigmar^T\end{array}\right]&=\left[\begin{array}{cc} \dual{\taur}\HRdual& \dual{\taur}\dual{\GR}\\ \dual{\sigmar}\HRdual& \dual{\sigmar}\dual{\GR}\end{array}\right].\end{split}
\end{align}
The duality constraints in \eqref{ExplicitDualCOMP}  do not depend explicitly  on the grid size $h$. Moreover, we already know that for the wide case, the penalty parameters in 
 \eqref{WIDEtauRel} -- even though they contain the $h$-dependent constant $\p$ -- gives dual consistency. 
Since the generalized penalty parameters in \eqref{TauUtanMu} have exactly the same form (the  only difference is that they depend on {\it another} $h$-dependent constant, $\q$) they
 will also yield dual consistency.
We have thus shown that the penalty parameters in Theorem~\ref{ParaSystPen} indeed makes the scheme \eqref{PrimalCOMPACT} stable and dually consistent.

\begin{remark}
The SAT parameters in Theorem~\ref{ParaSystPen} are probably a subset of all parameters giving stability and dual consistency since the duality constraint \eqref{ExplicitDualCOMP}   could be used in combination with some other stability proof than the one presented here.
\end{remark}

\section{Computing 
$\q$}
\label{AppRinv}

We want to compute $\q=\q_0+|\q_c|=\q_N+|\q_c|$ as stated in \eqref{qdef} and are thus looking for $\q_0$, $\q_N$ and $\q_c$
specified in \eqref{qdefparts}.
For wide second derivative operators, $M$ is equal to $\PH $, and is thus well-defined. 
When using
narrow second derivative operators, $M$ is defined in \eqref{SBPprop2}  through $\compA=S^TMS$. 
However, only the first and last row of $S$ 
are clearly specified.
In for example \cite{Carpenter1999341,Mattsson2004503,Eriksson20092659}, the 
 interior of $S$ is the identity matrix,
and $S$ is then invertible. 
$\compA$ is singular
(since $\compA=(E_N-E_0)S-\PH D_2$, where
$D_2$ and the first and last row of $S$ are consistent difference operators)   
and thus an invertible $S$ implies that $M$ is singular.

If $M$ and $S$ are defined such that  $M$ is singular
and $S$ not, 
which is often the case, 
we use the following strategy to find $\q$:
The relation
$\compA=S^TMS$ leads to $M^{-1}=S\compA^{-1}S^T$, 
but
since $\compA$ is singular 
we define
 the perturbed matrix $\compAsnok \equiv \compA+\pert E_0$ and 
compute $\Msnok^{-1}=S\compAsnok ^{-1}S^T$ instead.
This is motivated by the following proposition:

\begin{proposition}
\label{propJK}
Define $\compAsnok \equiv \compA+\pert E_j$, where $E_j$  is an all-zero matrix except for the element $(E_j)_{j,j}=1$, with $0\leq j\leq N$.
The inverse of $\compAsnok $ is  
$\compAsnok ^{-1}=J/\pert+K_j$
where $J$ is an  all-ones matrix and $K_j$ is a matrix that does not depend on the scalar $\pert$.
A consequence of this structure is that 
the corners of   $\Msnok^{-1}=S\compAsnok ^{-1}S^T$ are independent of $\pert$, such that   
\begin{align}
\label{qdefpartsSNOK}
\q_0=e_0^T\Msnok^{-1}e_0,&&\q_N=e_N^T\Msnok^{-1}e_N,&& \q_c=e_0^T\Msnok^{-1}e_N=e_N^T\Msnok^{-1}e_0.
\end{align}
\end{proposition}

\noindent
Proposition~\ref{propJK}  is motivated in Appendix~\ref{qex}.
In Table~\ref{qconst} below we provide the value of $\q$ for all  second derivative operators   considered in this paper.
The wide-stencil operators  are given by $D_2=D_1^2$, where  $D_1$ has the order of accuracy  (2,1), (4,2), (6,3) or (8,4), paired as (interior order, boundary order). For these operators, the $\q $ values are  obtained directly from the matrix $\PH $.
For the narrow-stencil operators, the $\q $ values are computed according to Proposition~\ref{propJK}.
All examples in Table~\ref{qconst}, 
 except the narrow (2,0) order operator, 
refers to operators 
 given in 
 \cite{Mattsson2004503}.

\begin{table}[H]
\centering
\begin{tabular}{|clll|}\hline
{Order}&{Type}&$\q h$&{Comment}\\\hline
2,0&{wide}&2&\\
4,1&{wide}&$\frac{48}{17}\approx2.8235$&\\
6,2&{wide}&$\frac{43200}{13649}\approx3.1651$&\\
8,3&{wide}&$\frac{5080320}{1498139}\approx3.3911$&\\\hline
2,0&{narrow}&1&{See Eq. }\eqref{2nd}\\
2,1&{narrow}&2.5&\\
4,2&{narrow}&3.986391480987749   & $(N=8)$ \\
6,3&{narrow}&5.322804652661742 &$(N=12)$\\
8,4&{narrow}&633.69326893357&$(N=16)$\\\hline
\end{tabular}
\caption{The $\q$-values (scaled with $h$) for various second derivative operators.}
\label{qconst}
\end{table}

\begin{remark} The SBP operators with interior order 6 and higher have free parameters, and  if those parameters are chosen differently than in  \cite{Mattsson2004503}, that will affect  $\q$. 
\end{remark}

\begin{remark}\label{RemNonDual}

The quantity $\q$ has nothing to do with 
dual consistency, 
but 
indicates how the penalty should be chosen to give energy stability.
As an example, consider solving the scalar problem presented below in \eqref{primalSCALAR} with  Dirichlet boundary conditions, using the
scheme \eqref{primalDiscScalar}.
Using 
the same 
technique as in Section~\ref{StabComp}, we find that
the stability demands for the (left) 
penalty parameter $\taul$, in three special cases of $\sigmal$, are
\begin{align*}
&\text{Dual consistent (see Eq. \eqref{Dirichlet})}&&\sigmal=-\varepsilon&&\taul
\leq-a/2-\varepsilon\q\\
&\text{Method 1 (dual inconsistent)}&&\sigmal=0&&\taul
\leq-a/2-\varepsilon\q/4\\
&\text{Method 2 (dual inconsistent)}&&\sigmal=\varepsilon&&\taul
\leq-a/2.
\end{align*}
The two latter 
approaches
are frequently used 
but they do not yield dual consistency.

\end{remark}

\section{Examples and numerical experiments}
\label{ExNum}

In this section, we  give 
a few concrete 
examples of  the derived penalty parameters and perform some numerical simulations. 
We  demonstrate that these penalty parameters give superconvergent functional output not only for the wide second derivative operators but also for the narrow ones.
The following procedure  is used:
\begin{itemize}
\item[i)] Consider a continuous problem on the form \eqref{ParaSyst}, where 
$\GL=\factorL\paraE$ and $\GR=\factorR\paraE$
are required.
 Identify $\bigA$ and $\bigBL,\bigBR$ according to \eqref{LiteViktigare}.

\item[ii)] Factorize $\bigA$ as $\bigA=\bigX\bigRot \bigX^T,$ according to \eqref{ALambdaX}, where 
$\bigX$  must be non-singular. 
\item[iii)]

Compute $\bigJL$ and $\bigJR$. From  \eqref{RLrelation} we see that  $\bigJL$ is  the first $\sizebig_+\times \sizebig_+$ part of $\bigBL\bigX^{-T}$, and correspondingly, that $\bigJR$ is  the last $\sizebig_-\times \sizebig_-$ part of $\bigBR\bigX^{-T}$, as
\begin{align}\label{nyRLrelation}
\bigBL\bigX^{-T}\hspace{-2pt}=\hspace{-1pt}\left[\hspace{-2pt}\begin{array}{ccc}\bigJL& \bigNL& \bigJL\bigRL\end{array}\hspace{-2pt}\right],&&\bigBR\bigX^{-T}\hspace{-2pt}=\hspace{-1pt}\left[\hspace{-2pt}\begin{array}{ccc}\bigJR\bigRR&\bigNR&\bigJR\end{array}\hspace{-2pt}\right].\end{align}

\item[iv)]
The problem \eqref{ParaSyst} is discretized in space using the scheme  \eqref{PrimalCOMPACT}. Rearranging the terms
in the scheme yields $\nsol _t+\lindisc  \nsol =\RHS $,
where $\lindisc $ is given in \eqref{LhDiscCompact}, and where
\begin{align*}
\RHS =\fh -\Pbar^{-1}(e_0 \otimes\taul +S^Te_0 \otimes\sigmal )\gl-\Pbar^{-1}(e_N \otimes\taur+S^Te_N \otimes\sigmar)\gr.
\end{align*}
The  penalty parameters $\taul$, $\sigmal$, $\taur$ and $\sigmar$ are 
specified in  Theorem~\ref{ParaSystPen}.
 \item[v)]
 If $\sol _t=0$, we have a stationary problem and the linear system $\lindisc  \nsol =\RHS $ must be solved. For the time-dependent cases,
 we use the method of lines and discretize $\nsol _t+\lindisc  \nsol =\RHS $ 
 in time using a suitable solver for ordinary differential equations.
\end{itemize}
\begin{remark}
When we have a hyperbolic problem, 
step (i) is omitted and step (iv) is modified such that the scheme \eqref{HypSystDisc} is used with penalty parameters given 
in Theorem~\ref{HypSystPen}.
\end{remark}

In the simulations, we are interested  in the functional  error 
$\Jfel=\funcdisc(\nsol )-\func(\sol )$, 
where  $\func(\sol )=\inprod{\weight ,\sol }$, $\funcdisc(\nsol )=\inprod{ \gh ,\nsol }_\discmark$  
and 
$\gh_i(t)=\weight (x_i,t)$, 
but of course also
in the solution error $\fel$, where $\fel_i(t)=\nsol _i(t)-\sol (x_i,t)$.
 We also investigate  the spectra of $\lindisc $, 
 that is the eigenvalues $\lambda_j$ of $\lindisc ,$ with $j=1,2,\hdots,\sizepara(N+1)$. 
 Here we are in particular interested in 
 the spectral radius
 $\maxeig=\max_j(|\lambda_j|)$
 and in
 $\mineig=\min_j(\Re(\lambda_j))$. (For time-dependent problems $\maxeig\Delta t\lesssim
C$ is a crude estimate of the stability regions of explicit  Runge-Kutta schemes, and 
thus $\maxeig$
can be seen as a measure of 
stiffness.
The eigenvalue with the smallest real part,
$\mineig$, determines how fast a time-dependent solution converges to a steady-state solution, see \cite{SteadyState}.) 
Ideally, the penalties are chosen such that $\maxeig$ is kept small   while $\mineig$ is maximized.
For steady problems or when using implicit 
 time solvers, other properties (e.g. the condition number)
might be of greater interest.

 We start by investigating a couple of scalar cases in some detail, then give an example 
 of a system with a solid wall type of boundary condition.

\subsection{The scalar case}
\label{SecScalar}

Consider the scalar advection-diffusion equation, 
\begin{align}
\label{primalSCALAR}
\begin{array}{rll}\scalaru_t+a\scalaru_x-\varepsilon \scalaru_{xx}=&\hspace{-7pt}\force ,\hspace{20pt}&x\in[0,1],\hspace{20pt}\vspace{4pt}\\\al \scalaru+\bl \scalaru_x=&\hspace{-7pt}\gl,&x= 0,\vspace{4pt}\\\ar \scalaru+\br \scalaru_x=&\hspace{-7pt}\gr,&x= 1,
\end{array}
\end{align}
valid for $t\geq0$, with initial condition $\scalaru(x,0)=\scalaru_0(x)$ and
where $\varepsilon>0$. Using \eqref{LiteViktigare} yields
\begin{align*}
\bigA=\left[\begin{array}{cc}a&-\varepsilon\\-\varepsilon&0
\end{array}\right], 
&&\bigBL=\left[\begin{array}{cc}\al&\bl
\end{array}\right],&&\bigBR=\left[\begin{array}{cc}\ar&\br
\end{array}\right].
\end{align*}
In this case, the factorization of the matrix $\bigA$ can be parameterized as
\begin{align}\label{ScalarFactorization}
\bigA=\bigX\bigRot \bigX^T=\left[\begin{array}{cc}\frac{a+\rot}{2s_1}&\frac{a-\rot}{2s_2}\\\frac{-\varepsilon}{s_1}&\frac{-\varepsilon}{s_2}\end{array}\right]\left[\begin{array}{cc}\frac{s_1^2}{\rot}&0\\0&-\frac{s_2^2}{\rot}\end{array}\right]
\left[\begin{array}{cc}\frac{a+\rot}{2s_1}&\frac{a-\rot}{2s_2}\\\frac{-\varepsilon}{s_1}&\frac{-\varepsilon}{s_2}\end{array}\right]^T,
\end{align}
with $\rot>0$.
In particular, if $\rot=\sqrt{a^2+4\varepsilon^2}$ and if $s_{1,2}^2=\rot(\rot\pm a)/2$, then
 the above  factorization  is the eigendecomposition of $\bigA$.
The discrete scheme mimicking \eqref{primalSCALAR} is
\begin{align}
\label{primalDiscScalar}
\begin{split}
\scalarv_t+aD_1\scalarv-\varepsilon D_2\scalarv=\fh 
&+\PH ^{-1}( \taul e_0 +\sigmal S^Te_0) \left(\al \scalarv_0+\bl(S \scalarv)_0-\gl\right)\\
&+\PH ^{-1}(\taur e_N +\sigmar S^Te_N) \left(\ar \scalarv_N+\br(S \scalarv)_N-\gr\right).
\end{split}
\end{align}
To compute the penalty parameters,  $\bigJL=\frac{s_1}{\rot}\left(\al+\bl\frac{a-\rot}{2\varepsilon}\right)$ and $\bigJR=-\frac{s_2}{\rot}\left(\ar+\br\frac{a+\rot}{2\varepsilon}\right)$ are needed, which we obtain using \eqref{nyRLrelation}.
Theorem~\ref{ParaSystPen} now yields 
\begin{align}\label{OptPenScalar}
\begin{split}
\taul&=\frac{-\frac{a+\rot}{2}-\q\varepsilon}{
\al+\bl\frac{a-\rot}{2\varepsilon}-\q \bl },\hspace{45pt}\sigmal=\frac{-\varepsilon}{
\al+\bl\frac{a-\rot}{2\varepsilon}-\q \bl },\\
\taur&=\frac{\frac{ a-\rot}{2}-\q\varepsilon}{\ar+ \br\frac{ a+\rot}{2\varepsilon}+\q\br},\hspace{41pt}\sigmar=\frac{\varepsilon}{\ar+ \br\frac{ a+\rot}{2\varepsilon}+\q\br}.
\end{split}
\end{align}
Formally $0<\rot<\infty$ is necessary (since in the limits $\bigX$ becomes singular), but 
as long as 
the number of imposed boundary condition does not change or the penalty parameters go 
 to infinity, 
we can allow
$0\leq\rot\leq\infty$.
Below we present some special cases:

For Dirichlet boundary conditions we have
 $\al=\ar=1$ and $\bl=\br=0$. In this case  the penalty parameters in
\eqref{OptPenScalar} become
\begin{align}\label{Dirichlet}
\taul&=-\frac{a+\rot}{2}-\q\varepsilon,&\sigmal&=-\varepsilon,&
\taur&=\frac{ a-\rot}{2}-\q\varepsilon,&\sigmar&=\varepsilon,
\end{align}
with $0\leq\rot<\infty$. 
Translating the penalty parameters 
for the advection-diffusion case  in \cite{Berg20126846}
 to the form used here, 
it can be seen that 
they are exactly the same. 
  
With $\al=\frac{|a|+a}{2}$, $\bl=-\varepsilon$ at the left boundary and
  $\ar=\frac{|a|-a}{2}$, $\br=\varepsilon$ at the right boundary, 
we have boundary conditions of a low-reflecting far-field type. In this case,
    the penalty parameters in
\eqref{OptPenScalar} become
  \begin{align}\label{OptPenScalarFarField}
\taul&=-\frac{\frac{\rot+a}{2}+\q\varepsilon}{
\frac{\rot+|a|}{2}+\q \varepsilon },&&\sigmal=\frac{-\varepsilon}{
\frac{\rot+|a|}{2}+\q \varepsilon},&
\taur&=-\frac{\frac{ \rot-a}{2}+\q\varepsilon}{\frac{\rot+|a|}{2}+\q\varepsilon},&&\sigmar=\frac{\varepsilon}{\frac{\rot+|a|}{2}+\q\varepsilon}
\end{align}
and we see that in the limit $\rot\to\infty$, we obtain $\taul=-1$, $\sigmal=0$, $\taur=-1$ and $\sigmar=0$.
This particular choice  corresponds to the penalty $\Sigma=-I$ used in
\cite{Berg201341,Berg2014135} for systems with boundary conditions of far-field type.

\begin{remark}

If $\varepsilon=0$ in \eqref{primalSCALAR} we get the 
transport
 equation, and then only one boundary condition should be given instead of two. 
That means that 
the derivation of the penalty parameters must be redone accordingly. 
See
\cite{Berg20126846}, where this case is covered.

\end{remark}

\begin{remark}
\label{RemVarying}

The results can be extended to the case of varying 
coefficients. 
Consider the scalar diffusion problem $\scalaru_t-(\varepsilon \scalaru_{x})_x=\force $
with Dirichlet boundary conditions, where $\varepsilon(x)>0$. 
Following \cite{Mattsson2012}, we define a narrow-stencil operator mimicking $\partial/\partial x(\varepsilon\partial/\partial x)$  as
\begin{align*}
D_2^{(\varepsilon)}=\PH ^{-1}\left(-\compA^{(\varepsilon)}+(\varepsilon(1)E_N-\varepsilon(0)E_0)S\right)
\end{align*}
where $\compA^{(\varepsilon)}$ is symmetric and positive semi-definite.
It is assumed that $D_2^{(\varepsilon)}=\varepsilon D_2$ holds when $\varepsilon$ is constant. 
The discrete problem becomes
\begin{align*}
\scalarv_t-D_2^{(\varepsilon )}\scalarv=\fh 
+\PH ^{-1}( \taul e_0 +\sigmal S^Te_0) \left(\scalarv_0-\gl\right)+\PH ^{-1}(\taur e_N +\sigmar S^Te_N) \left(\scalarv_N-\gr\right).
\end{align*}
The continuous  problem is self-adjoint, so for dual consistency $\lindisc ^*= \PH^{-1}\lindisc ^T\PH=\lindisc $ is needed,
which is fulfilled 
 if $\sigmal=-\varepsilon(0)$ and $\sigmar=\varepsilon(1)$. 
Moreover, using $\compA^{(\varepsilon)}\geq\varepsilon_{\min} \compA$,
where $\varepsilon_{\min}=\min_{x\in[0,1]}\varepsilon(x)$, 
it can be shown
that the discretization will be stable if we choose
$\taul  \leq-\frac{\q}{\varepsilon_{\min}}\varepsilon(0)^2$ and  $\taur\leq-\frac{\q}{\varepsilon_{\min}}\varepsilon(1)^2$.
(The superconvergence for functionals has been confirmed numerically and the resulting "best" choices of $\taul$ and $\taur$ are similar to what we obtain in the constant case considered below.)

\end{remark}

\subsubsection{The stationary heat equation with Dirichlet boundary conditions}

We consider the 
heat equation with Dirichlet boundary conditions, i.e.
problem \eqref{primalSCALAR}
with $a=0$, $\al,\ar=1$ and $\bl,\br=0$, which we solve using the scheme \eqref{primalDiscScalar}, with the penalty parameters 
 given by
\eqref{Dirichlet}, also with $a=0$. 
To isolate the errors originating from the spatial discretization, we first look at  the
steady problem. 
Thus we let $\scalaru_t=0$ and solve $-\scalaru_{xx}=\force (x)$ numerically. The 
resulting quantities  
$\maxeig$ and $\mineig$, 
the solution error $\|\fel\|_\discmark$ and the functional error $|\Jfel|$ are given (as functions of the  parameter $\rot$) 
in Figure~\ref{FigA}.
The spectral radius $\maxeig$ grows with $\rot$, so we do not want $\rot\to\infty$. On the other hand, the decay rate $\mineig$ shrinks with $\rot$ so $\rot\to0$ should also be avoided. The errors tend to decrease with increasing $\rot$ (the errors naturally varies slightly depending on  the choice of $\force$ and $\weight$, but the example in Figure \ref{FigA} shows  a typical behavior). Thus the demand for accuracy is conflicting with the demand of keeping $\maxeig$ small  (the aim to maximize $\mineig$ is met before the aim to minimize the errors  and is therefore not a limiting factor in this case). Empirically we have found that a good compromise, which gives small errors without increasing the  spectral radius dramatically, is obtained using  $\rot\approx \q  \varepsilon$.

\begin{figure}[H]
\centering
\subfigure[Interior order 6, wide operator]{\includegraphics[width=.5\textwidth]{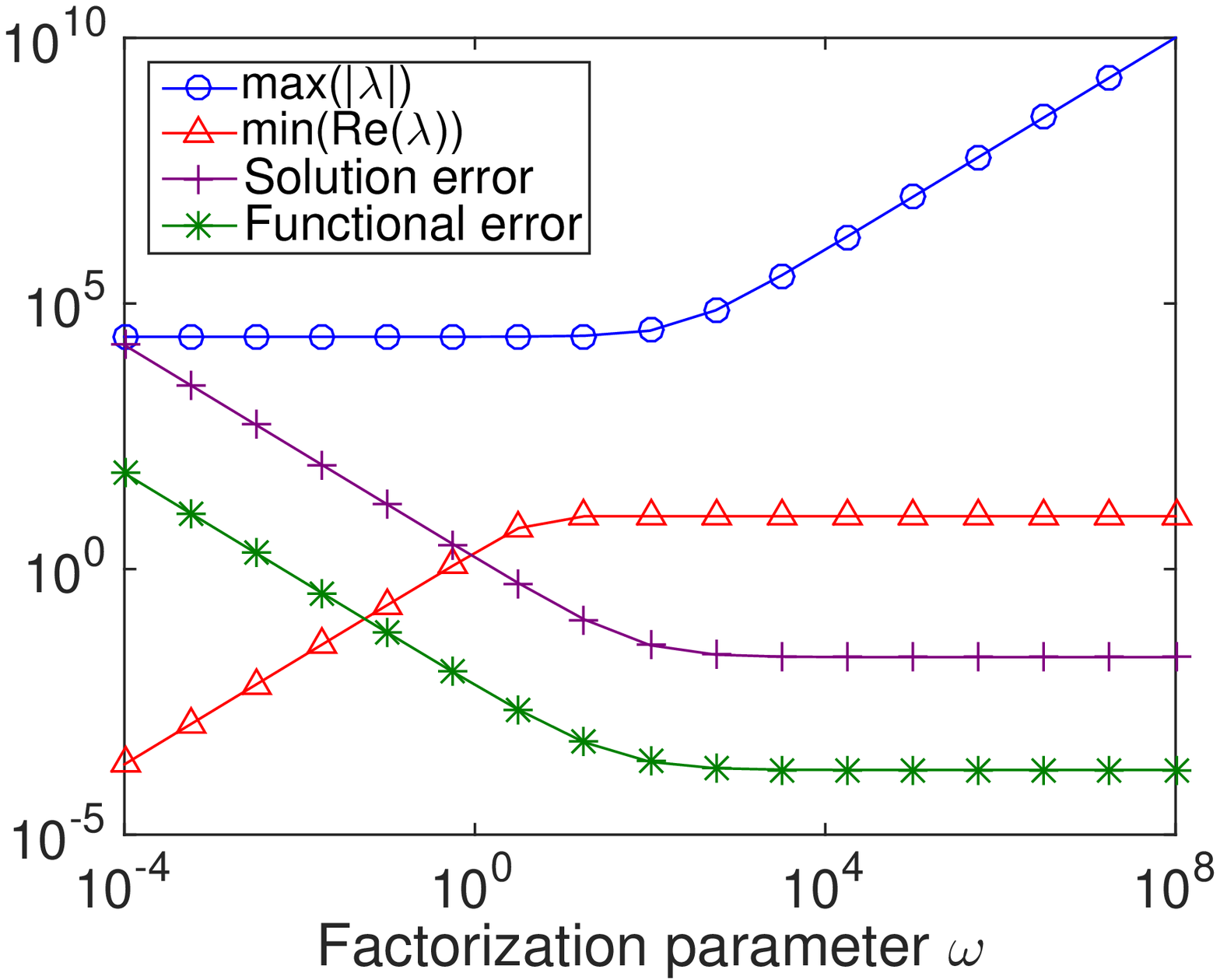}}%
\subfigure[Interior order 6, narrow operator]{\includegraphics[width=.5\textwidth]{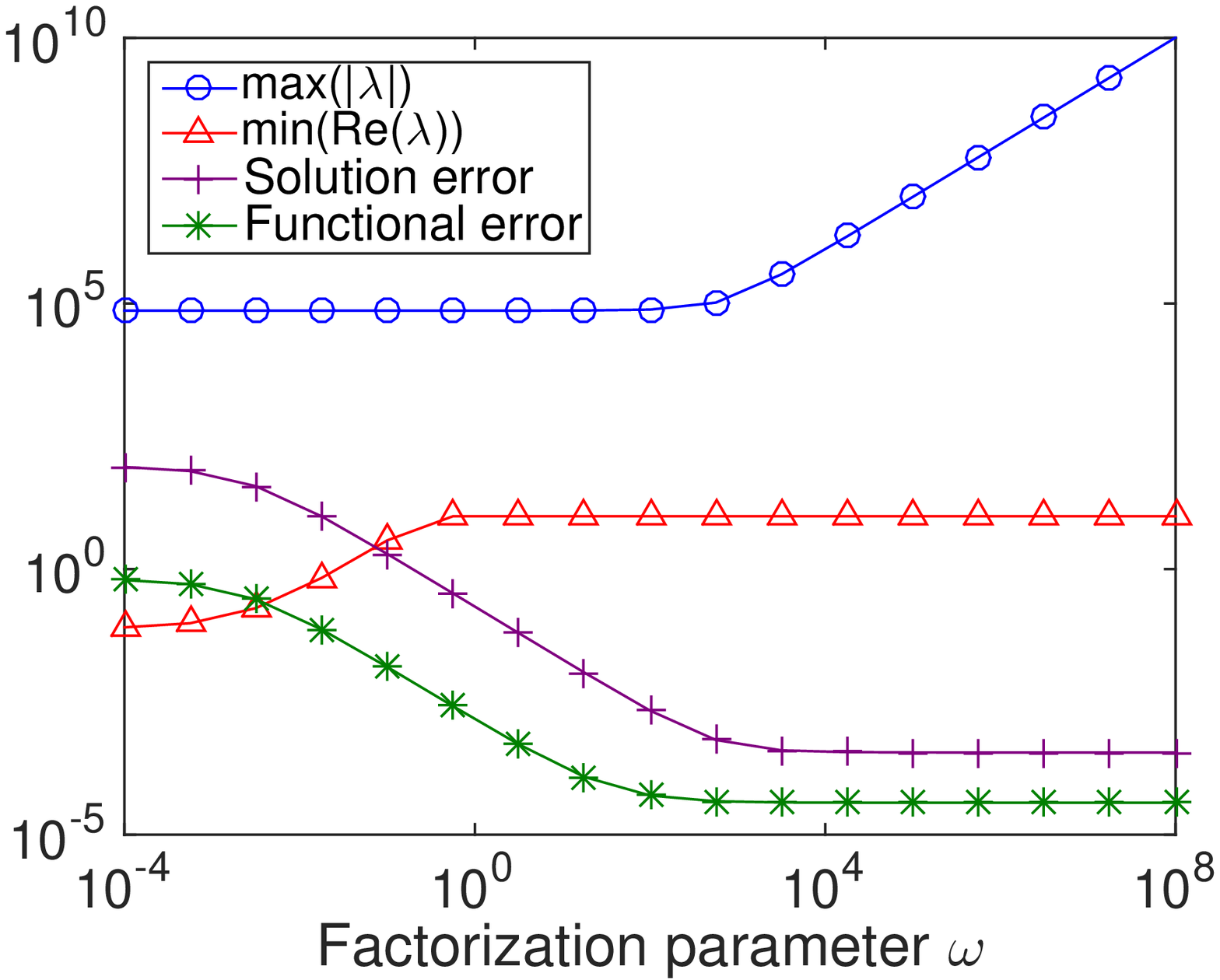}}
\caption{
Properties of $\lindisc $ and errors when solving $-\scalaru_{xx}=\force (x)$ with Dirichlet boundary conditions.
The number of grid points is $N=64$, the second derivative operator is 6th order accurate in the interior, and is either wide
or narrow.
Here $\scalaru(x)=\weight (x)=\cos(30x)$.
}\label{FigA}
\end{figure}

From this example, we make an observation. If we would  use the eigenfactorization, we would  have $\rot=\sqrt{a^2+4\varepsilon^2}=2$. However, in Figure~\ref{FigA} we see that
that choice is not especially good, since the errors then become much larger than if using  $\rot=\q\varepsilon$, which is approximately  200 and 340, respectively.
In some cases, the difference in accuracy is so severe that the choice of factorization parameter $\rot$ affects the convergence rate.
%
For the narrow operator with the order (2,0), the errors behave as  $\|\fel\|_\discmark\sim h^{3/2}$
when using 
$\rot\sim1$,
whereas we obtain the expected
 $\|\fel\|_\discmark\sim h^{2}$ 
when using 
$\rot\sim1/h$.
Similar behaviors are  observed also for narrow operators of higher order, see below. 

In Figure~\ref{FigB}(a) the errors $\|\fel\|_\discmark$ for the operators with interior order  6 are shown.  
For the narrow scheme, the convergence rate is 4.5 
when using $\rot=2\varepsilon$ and 5.5 when using  $\rot=\q\varepsilon$. 
 For the wide scheme, the order is 4 in both cases, but the error constant changes.
In the 8th order case, Figure~\ref{FigB}(b), the 
convergence rates are not affected, but in the narrow case
 the errors are around 2500 times smaller when using   $\rot=\q\varepsilon$.
 In this example,
 the functional errors are not as sensitive to $\rot$ as the solution errors.
 In the 6th order case,  
the convergence rates are slightly better than the predicted $2p=6$, both for the wide and the narrow schemes, see Figure~\ref{FigB2}(a). For the 8th order case, see Figure~\ref{FigB2}(b), 
the convergence rates are in all cases higher than $2p=8$. 
 Thus
the derived 
SAT parameters actually produce superconvergent functionals, also  for the narrow operators. 
\begin{figure}[H]
\centering
\subfigure[Interior order 6]{\includegraphics[width=.5\textwidth]{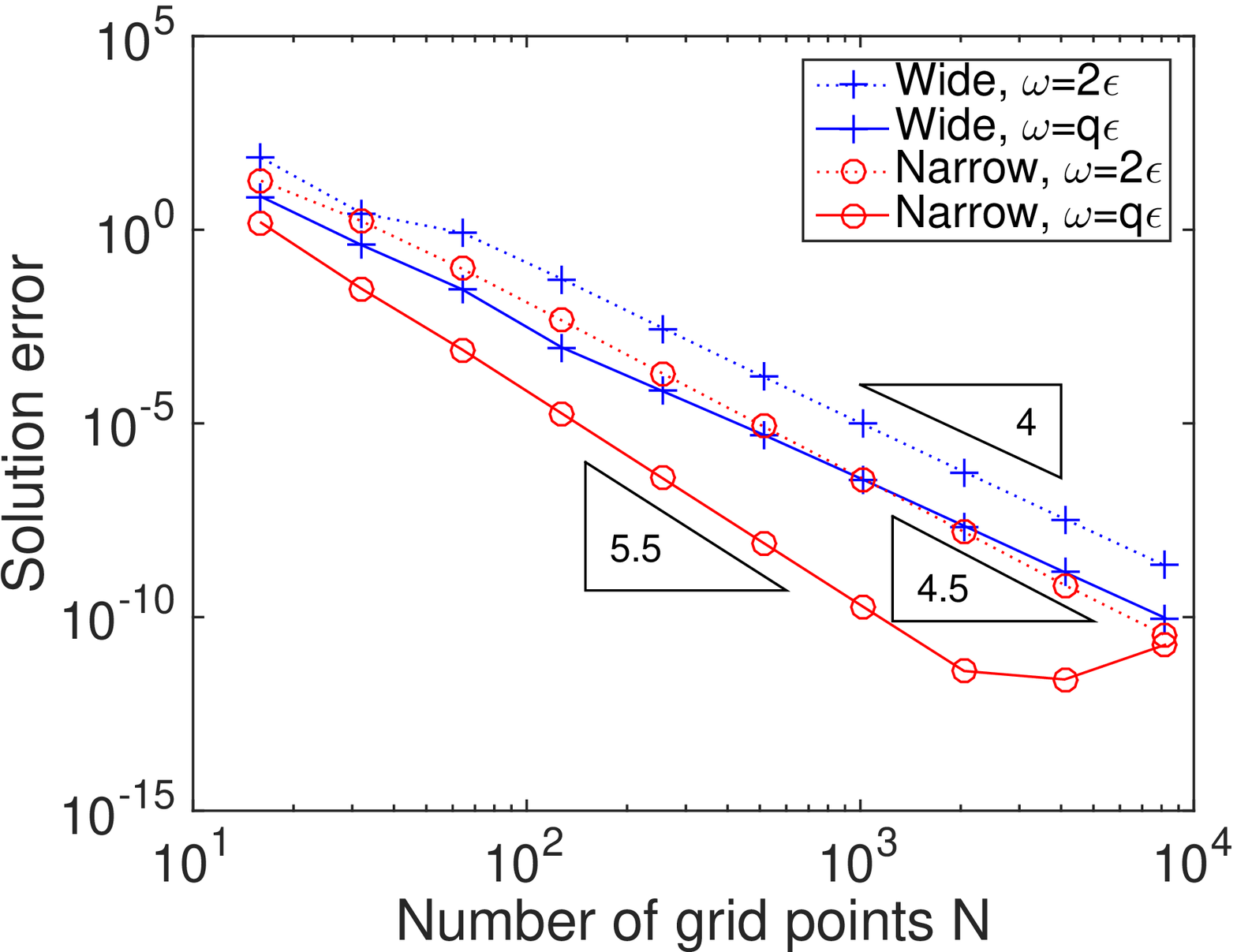}}%
\subfigure[Interior order 8]{\includegraphics[width=.5\textwidth]{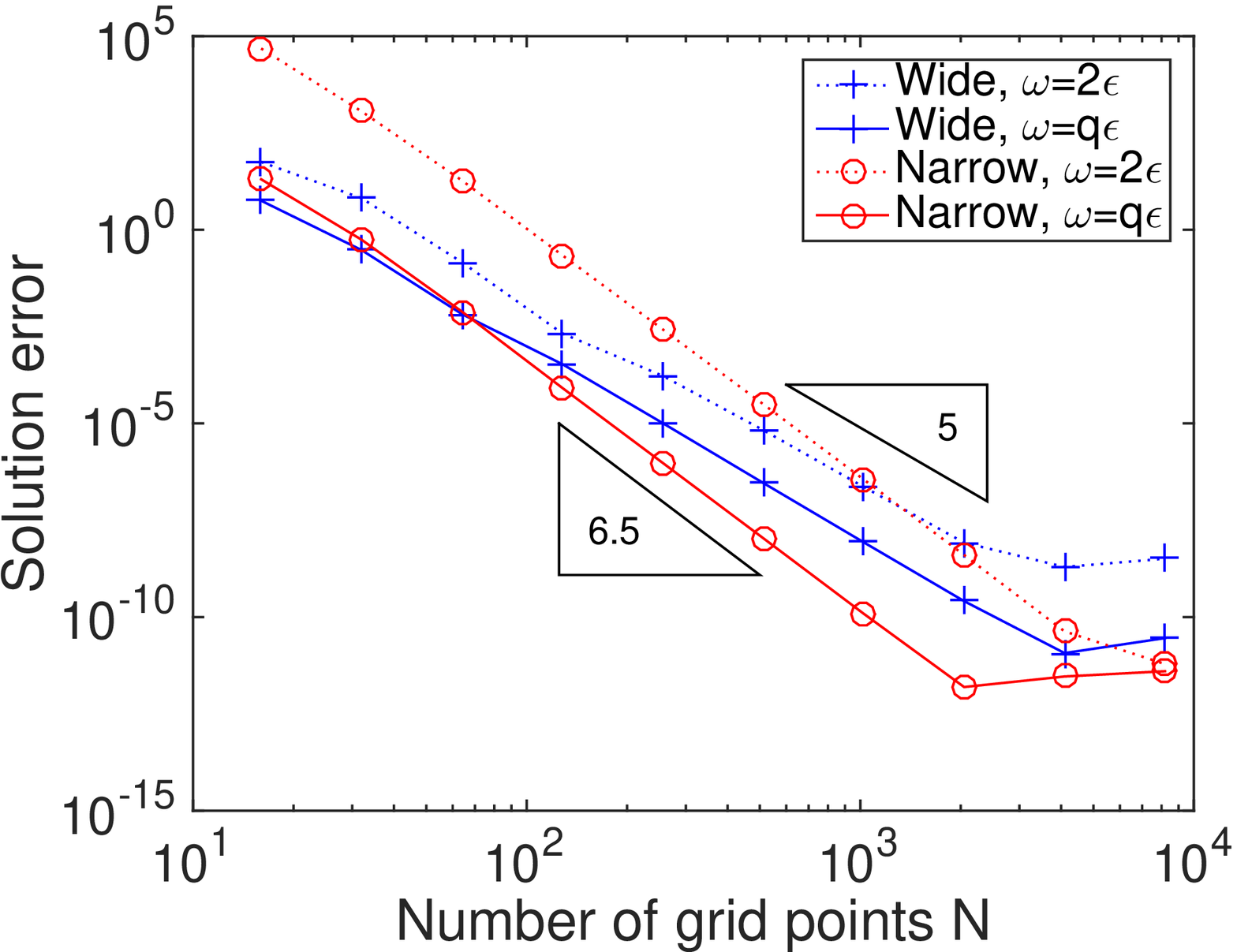}}
\caption{
The error $\|\fel\|_\discmark$, for $-\scalaru_{xx}=\force (x)$. The exact solution is $\scalaru=\cos(30x)$. 
}\label{FigB}
\end{figure}

 \begin{figure}[H]
\centering
\subfigure[Interior order 6]{\includegraphics[width=.5\textwidth]{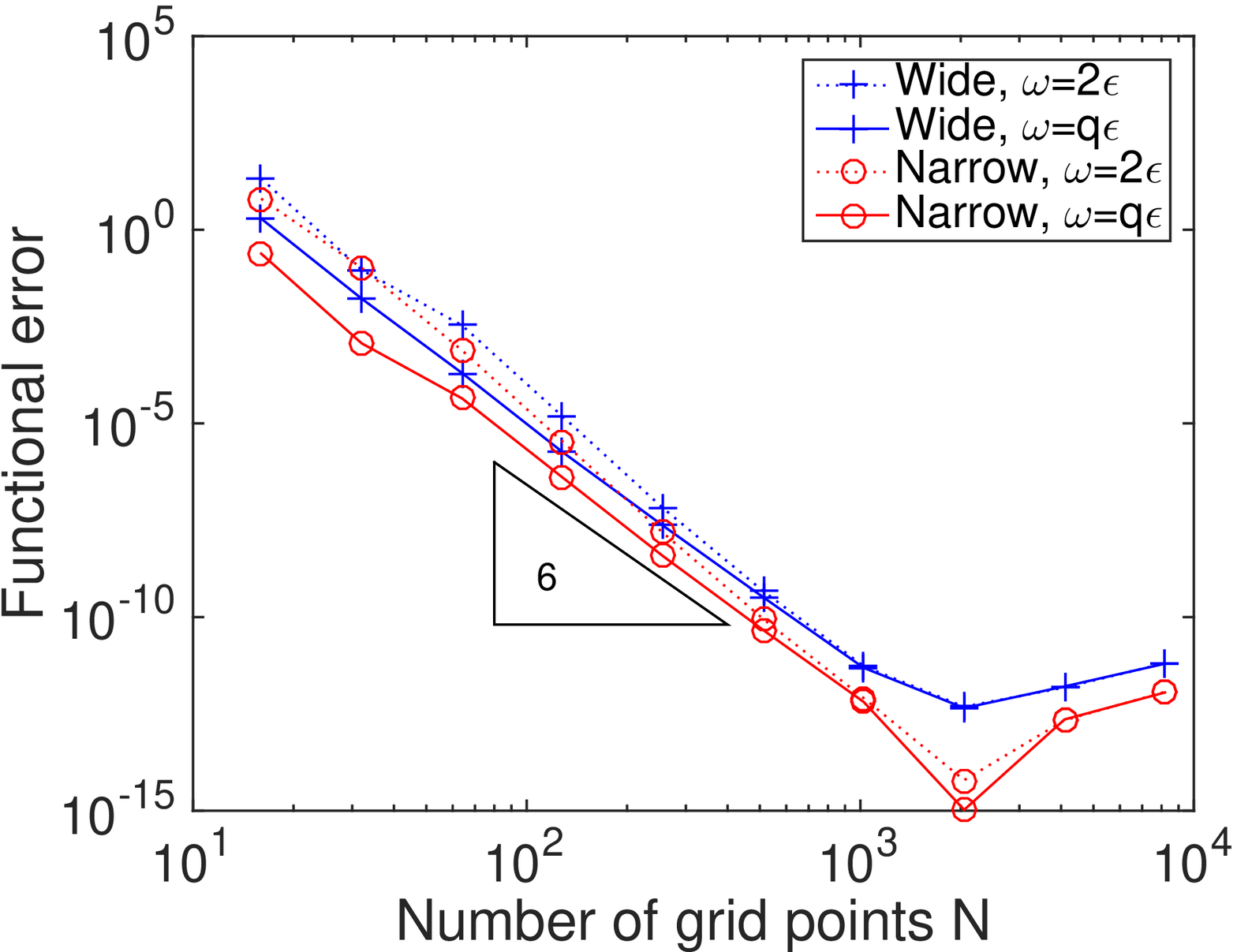}}%
\subfigure[Interior order 8]{\includegraphics[width=.5\textwidth]{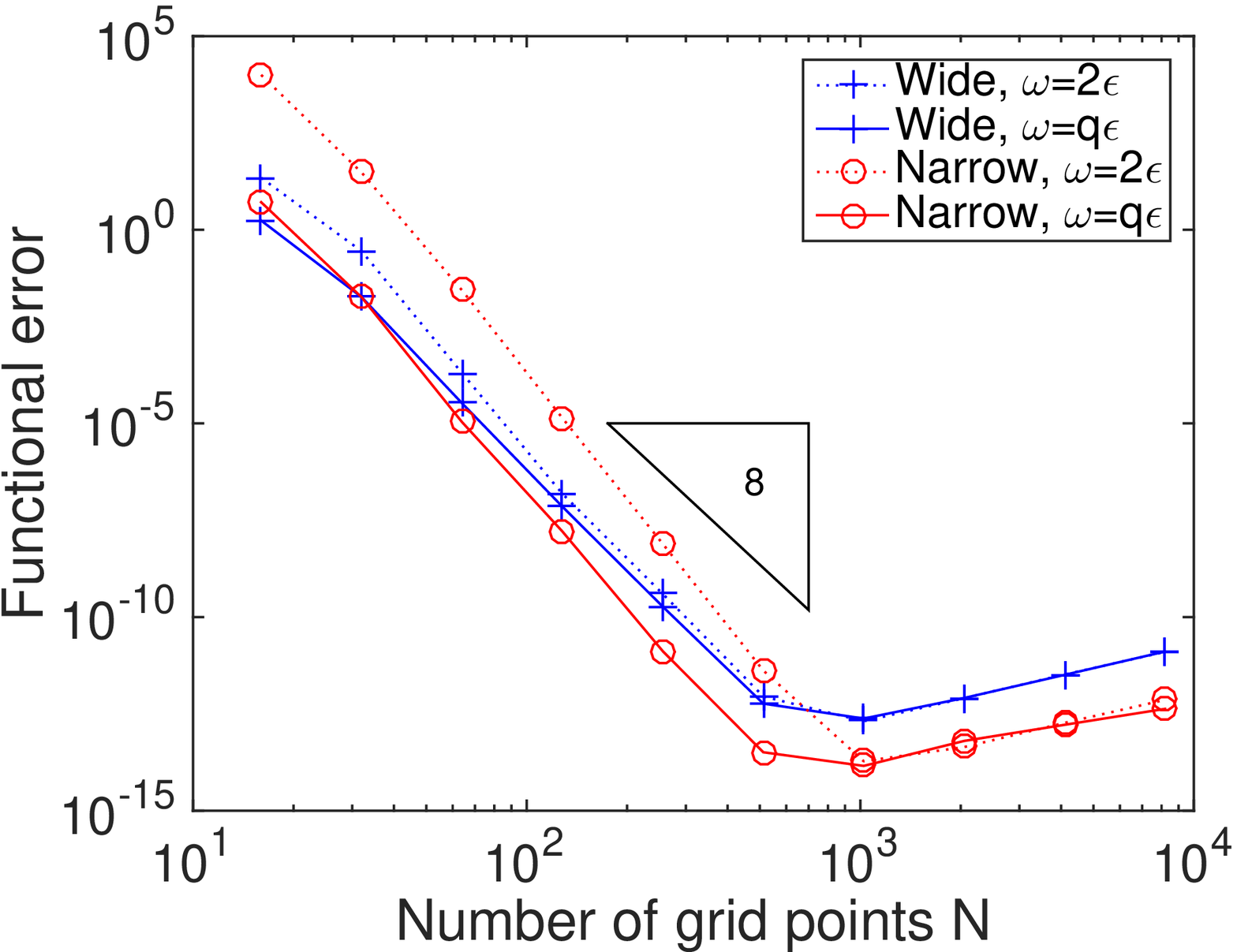}}
\caption{The functional error $|\Jfel|$, using the weight function $\weight (x)=\cos(30x)$.
}\label{FigB2}
\end{figure}

\subsubsection{The time-dependent heat equation with Dirichlet boundary conditions}
 
Next, we consider the actual heat equation. We solve $\scalaru_t=\varepsilon \scalaru_{xx}+\force (x,t)$ 
with 
$\varepsilon=0.01$ and
the 
exact solution  $\scalaru(x,t)=\cos(30x)+\sin(20x)\cos(10t)+\sin(35t)$. 
For the time propagation the classical 4th order accurate Runge-Kutta scheme  is used, with sufficiently small time steps, $\Delta t=10^{-4}$, such that the spatial errors dominate.
In Figure~\ref{Heat} the errors obtained using the narrow (6,3) order scheme 
 are shown as a function of time. 
\begin{figure}[H]
\centering
\subfigure[Solution error $\|\fel\|_\discmark$]{\includegraphics[width=.5\textwidth]{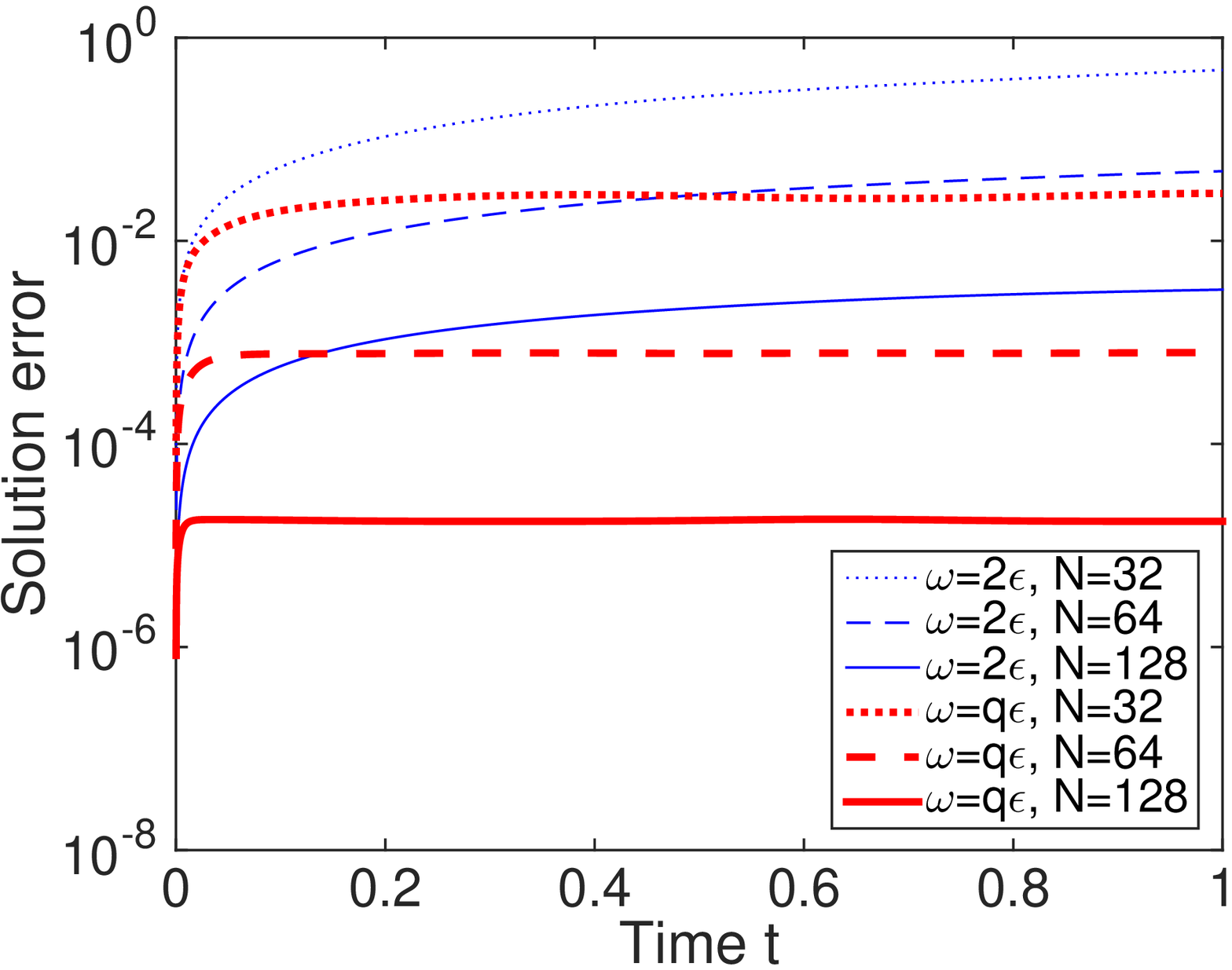}}
\subfigure[Functional error $|\Jfel|$ with $\weight (x)=1$]{\includegraphics[width=.5\textwidth]{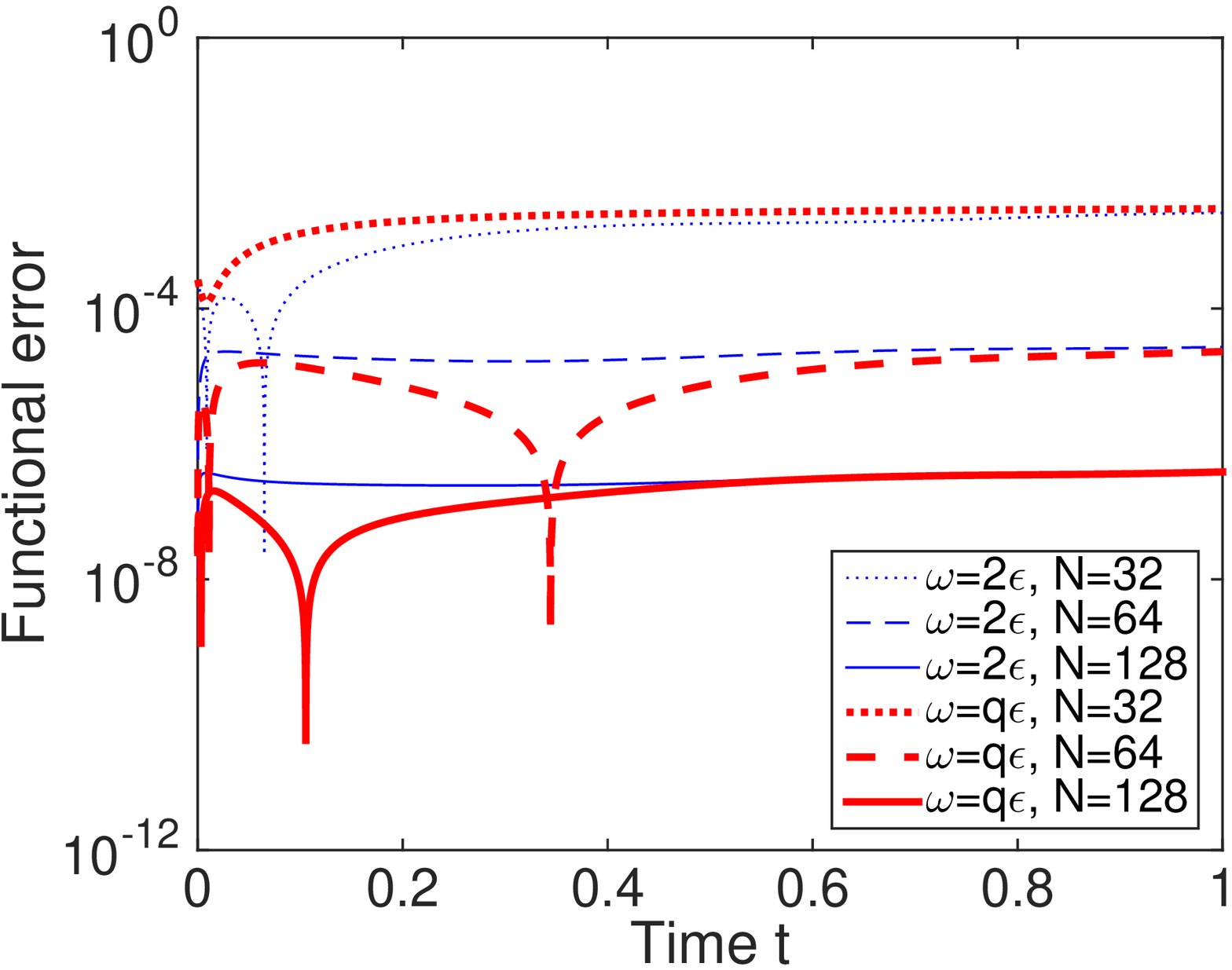}}
\caption{
Errors when solving the heat equation
using the narrow (6,3) order  scheme. 
}\label{Heat}
\end{figure}
The corresponding 
spatial
order of convergence (at 
time $t=1$) 
is shown in Table~\ref{HeatTable}. 
The simulations
 confirm the steady results, namely that 
 both 
 $\rot=2\varepsilon$ and 
  $\rot=\q\varepsilon$
 give superconvergent functionals
 but
 that
 choosing the 
factorization parameter  as $\rot\sim\varepsilon/h$  
 improves the solution significantly
 compared to when using the eigendecomposition. 
\begin{table}[H]\centering
$\begin{array}{|r|cccc|cccc|}
\hline
&\multicolumn{4}{|c|}{\rot=2\varepsilon}&\multicolumn{4}{|c|}{\rot=\q\varepsilon}\\
N&\|\fel\|_\discmark&\text{Order}&|\Jfel|&\text{Order}&\|\fel\|_\discmark&\text{Order}&|\Jfel|&\text{Order}\\\hline
32&0.480872&-&0.00258741&-&0.029297&-&0.00297573&-\\
64&0.048501&3.3096&0.00002704&6.5804&0.000790&5.2121&0.00002315&7.0064\\
128&0.003307&3.8743&0.00000038&6.1559&0.000017&5.5131&0.00000039&5.9055\\\hline
\end{array}$
\caption{The errors and convergence rates at $t=1$ for the narrow (6,3) order scheme.}
\label{HeatTable}
\end{table}

\subsubsection{The heat equation with Neumann boundary conditions}
 
We solve $\scalaru_t=\varepsilon \scalaru_{xx}+\force (x,t)$ again,
but this time with Neumann boundary conditions, and the penalty parameters are now given by \eqref{OptPenScalar} with $a=0$, $\varepsilon=0.01$, $\al=\ar=0$ and $\bl=\br=1$.
In contrast to when having Dirichlet boundary conditions, the spectral radius $\maxeig$ does not depend so strongly on $\rot$
and therefore we can let $\rot\to\infty$. 
Figure~\ref{HeatNeumann} 
shows the convergence rates for
 the  schemes with interior order 6.  The 
exact solution is $\scalaru(x,t)=\cos(30x)$ and
for the time propagation  the implicit Euler method, with $\Delta t=1$, is used (this is more than enough
since the chosen $\scalaru$ does not depend on $t$).
We note that the 
convergence rates behaves similarly to the Dirichlet case. We could also have used $\rot=\q\varepsilon$ here, it gives the same convergence rates as $\rot=\infty$.
 \begin{figure}[H]
\centering
\subfigure[Solution error $\|\fel\|_\discmark$]{\includegraphics[width=.5\textwidth]{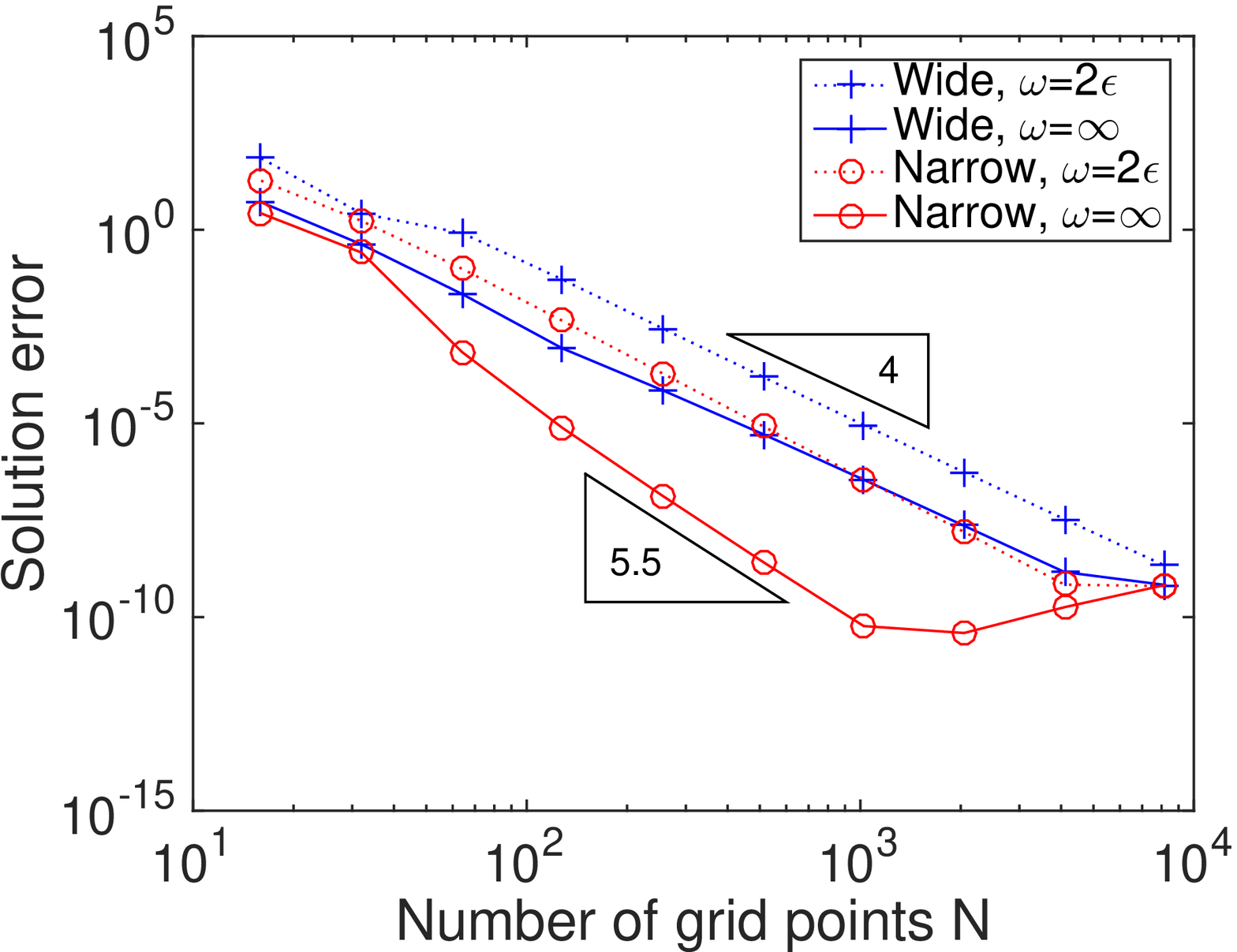}}%
\subfigure[Functional error $|\Jfel|$ with $\weight (x)=\cos(30x)$]{\includegraphics[width=.5\textwidth]{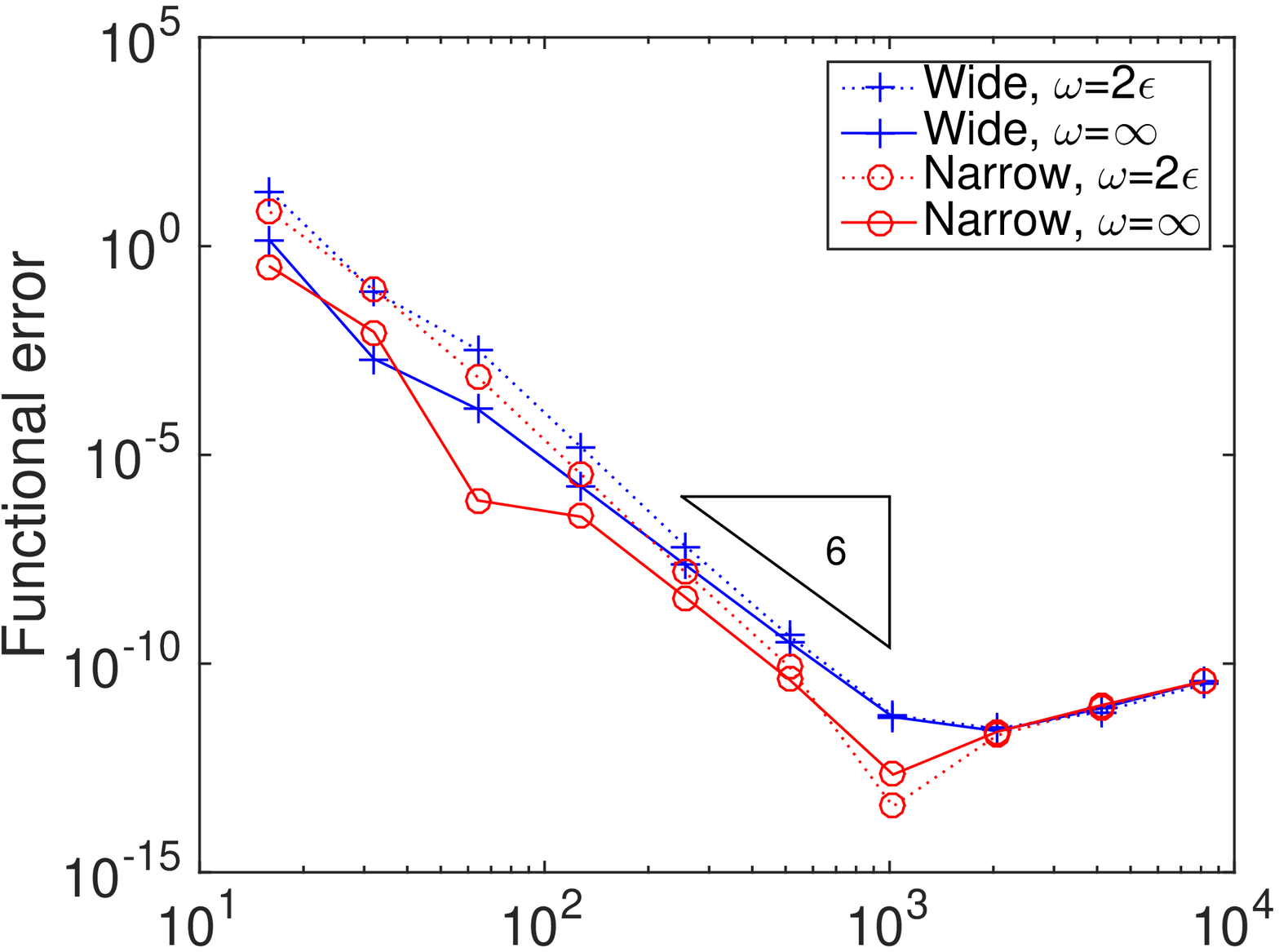}}
\caption{
Errors at time $t=100$ when solving the heat equation with Neumann boundary conditions
using the schemes with interior order 6. 
}\label{HeatNeumann}
\end{figure}

\subsubsection{The advection-diffusion equation with Dirichlet boundary conditions}

For simplicity we consider steady problems again, this time  $a\scalaru_x=\varepsilon \scalaru_{xx}+\force $. That is, we solve \eqref{primalSCALAR} 
using the scheme \eqref{primalDiscScalar},
both
with 
omitted
time derivatives.
 The penalty parameters
for Dirichlet boundary conditions 
 are given in \eqref{Dirichlet}.

First, we take a look at
an interesting special case, namely when 
$\force =0$. Then the exact solution is $\scalaru(x)=c_1+c_2\exp{(ax/\varepsilon)}$, where
the constants $c_1$ and $c_2$ are determined by the boundary conditions. 
For $\varepsilon\ll|a|$ 
the exact solution 
forms
a thin boundary layer 
at 
the 
outflow boundary,
which for insufficient resolution usually leads to oscillations in the numerical solution.
  This can be handled by upwinding or 
   artificial diffusion 
 (see e.g. 
 \cite{Alfio}).
 Here we will instead use the free parameter $\rot$ in the penalty to minimize the oscillating modes (the so-called $\pi$-modes).

We start with the wide second derivatives stencils. The ansatz $\scalarv_i=\ansatz^i$, inserted into the interior of the scheme \eqref{primalDiscScalar}, gives (for the second order case)  a numerical solution 
\begin{align*}
\scalarv_i=\widetilde{c}_1+\widetilde{c}_2(-1)^i+\widetilde{c}_3\ansatz_3^i+\widetilde{c}_4\ansatz_4^i,&&\ 
\ansatz_{3,4}=\frac{ha}{\varepsilon}\pm\sqrt{\frac{h^2a^2}{\varepsilon^2}+1}.
\end{align*}
Thus there  exist two modes with alternating signs, $\widetilde{c}_2(-1)^i$ and $\widetilde{c}_4\ansatz_4^i$.
However, 
one can show
that 
the choice $\rot=|a|$ leads to
$\widetilde{c}_2=0$ and to $\widetilde{c}_4$ being  small enough compared to $\widetilde{c}_3$ such that $\scalarv_i$ is monotone.
Empirically we have seen that this nice behavior holds also for the wide schemes with higher order of accuracy. In Figure~\ref{AdvDiff}
the result using the 
scheme with interior order 8 is shown.
The solution obtained using $\rot=|a|$ shows no oscillations, even though the grid is very coarse.
Moreover, this particular choice 
 of factorization gives functional errors almost at machine precision
 (although it should be noted that this is a special case since $\force (x)=0$ and $\weight (x)=1$).

\begin{figure}[h]
\centering
\subfigure[The number of grid points is $N=16$. ]{\includegraphics[width=.5\textwidth]{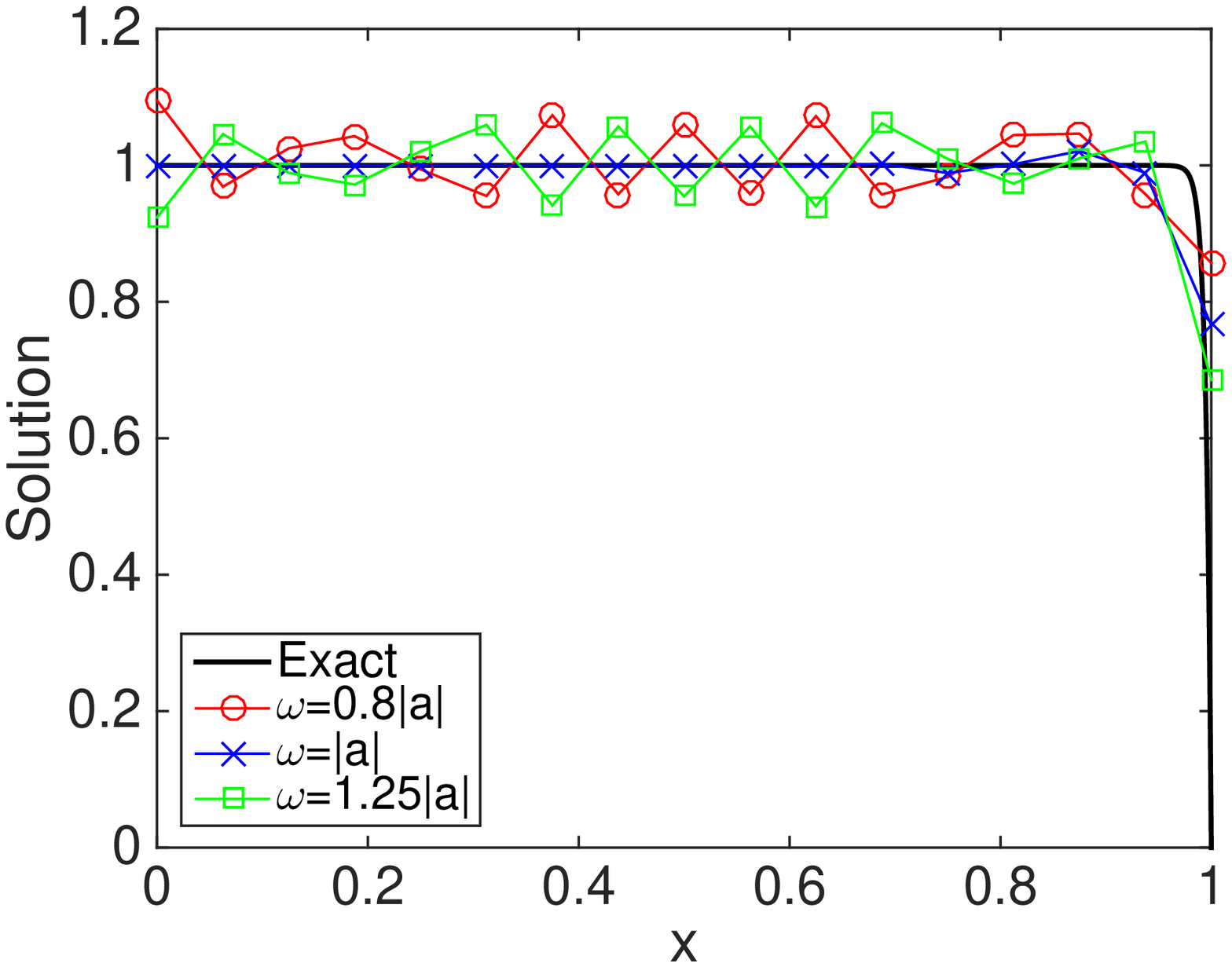}}%
%
%
%
  \subfigure[The weight function is $\weight (x)=1$.]{\includegraphics[width=.5\textwidth]{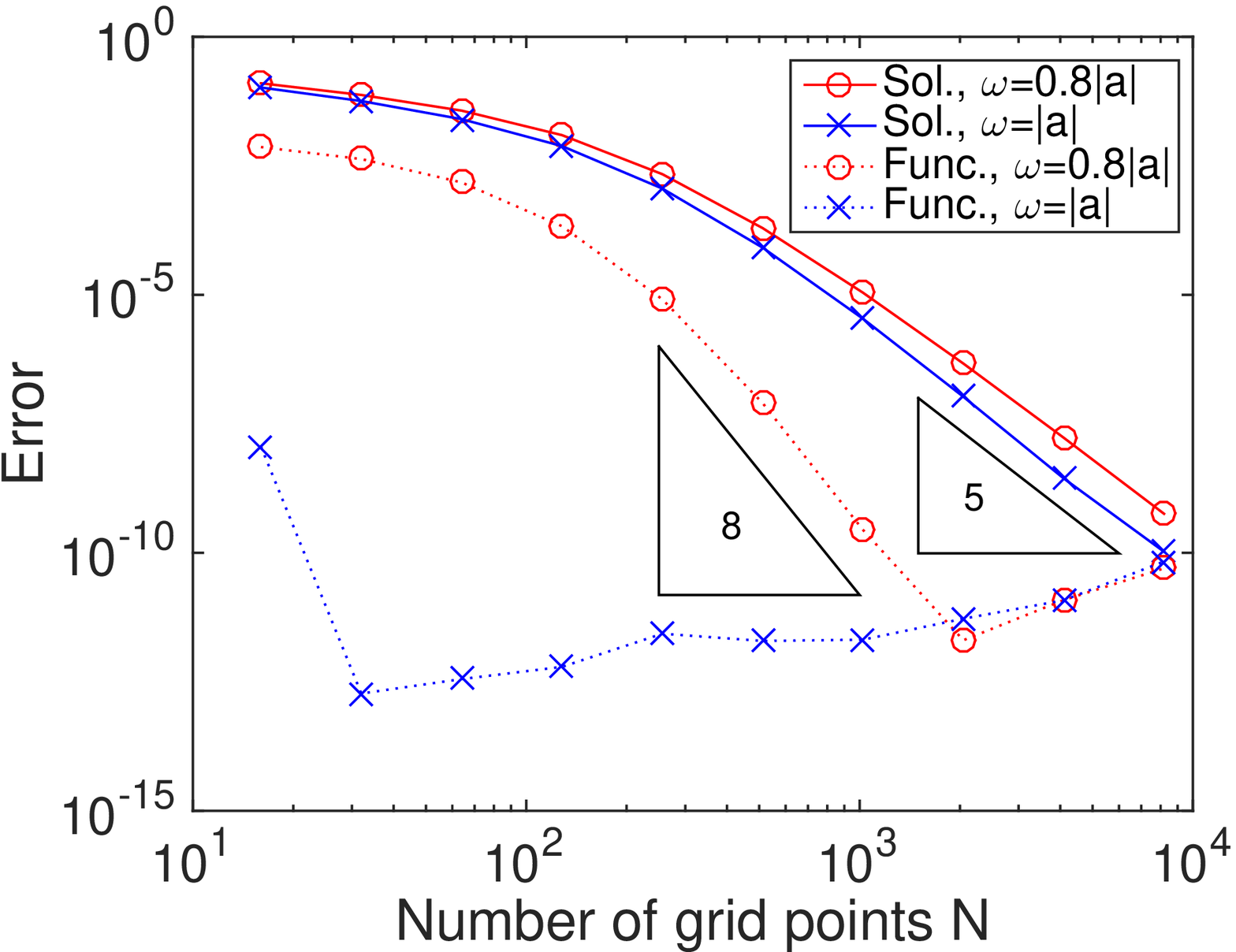}} 
\caption{
We solve $a\scalaru_x=\varepsilon \scalaru_{xx}$ with $a=1$, $\varepsilon=0.005$ using the 
wide scheme with interior order 8. 
In a) 
the solutions, 
in b) the  errors $\|\fel\|_\discmark$ and  $|\Jfel|$.  
}\label{AdvDiff}
\end{figure}

For the narrow-stencil  schemes, the existence of 
spurious oscillating modes 
depends 
on the resolution. In 
the second order case,  the interior solution is
\begin{align*}
\scalarv_i=\widetilde{\widetilde{c}}_1+\widetilde{\widetilde{c}}_2\left(\frac{1+ah/(2\varepsilon)}{1-ah/(2\varepsilon)}\right)^i,
\end{align*}which
has an oscillating component if  
$|a|h/(2\varepsilon)>1$. 
With very particular choices of the penalty parameter  this component
can be canceled (for the operators with order (2,0) and (2,1) it is achieved using $\rot=|a|/(1-\frac{2\varepsilon}{|a|h})$ and $\rot=|a|(1-\frac{\varepsilon}{|a|h})/(1-\frac{2\varepsilon}{|a|h})^2$, respectively) such that the numerical solution  becomes constant. 
As soon as $|a|h/(2\varepsilon)<1$,
 this mode should not be canceled anymore, but how to do the transition between 
 the unresolved case and  the resolved case is not obvious. 
For the higher order schemes the $\rot$ which cancels the oscillating modes are even more complicated and  in some cases negative (i.e. useless). 
In short,  these particular, canceling choices of $\rot$  are not 
 worth the effort. 
Instead, we recommend  to use $\rot\approx|a|+\q \varepsilon$  for the narrow-stencil operators, see below.


The above results 
were obtained under the assumption 
$\force =0$. Next, we use a forcing function $\force $ such that the exact solution is $\scalaru(x)=\cos(30x)$. The resulting errors, together with  $\maxeig$ and $\mineig$, are shown in 
Figure~\ref{FigAD} for $a=1$ and $\varepsilon=10^{-6}$. 
Clearly, 
$\rot\approx|a|$ is still a good choice 
since the errors are small, $\maxeig$ is  not too large and $\mineig$ is maximal.
For $\varepsilon\gg |a|h$ the curves are more 
similar to those in Figure~\ref{FigA}, and 
$\rot\approx|a|+\q \varepsilon$ will  be a better choice.
In the transition region $\varepsilon\sim |a|h$ we sometimes observe  order reduction.
This can be seen 
in Figures~\ref{FigADresolved} and \ref{FigADnotresolved} for the schemes with an interior order of accuracy 6. 
Figure~\ref{FigADresolved} shows the convergence rates when $\varepsilon=0.1$, which is large enough for the numerical solution to be well resolved. For the narrow scheme, we see an improved convergence rate for the solution error if $\rot=|a|+\q \varepsilon$ is used. The functional output converges with $2p=6$ for all schemes.
Figure~\ref{FigADnotresolved} shows the convergence rates when $\varepsilon$  is decreased to $10^{-4}$, such that 
the numerical solution is badly resolved. For all schemes, except the wide scheme with the particular choice  $\rot=|a|$, we see a pre-asymptotic order reduction of the functional.

\begin{figure}[H]
\centering
\subfigure[Wide-stencil second derivative]{\includegraphics[width=.5\textwidth]{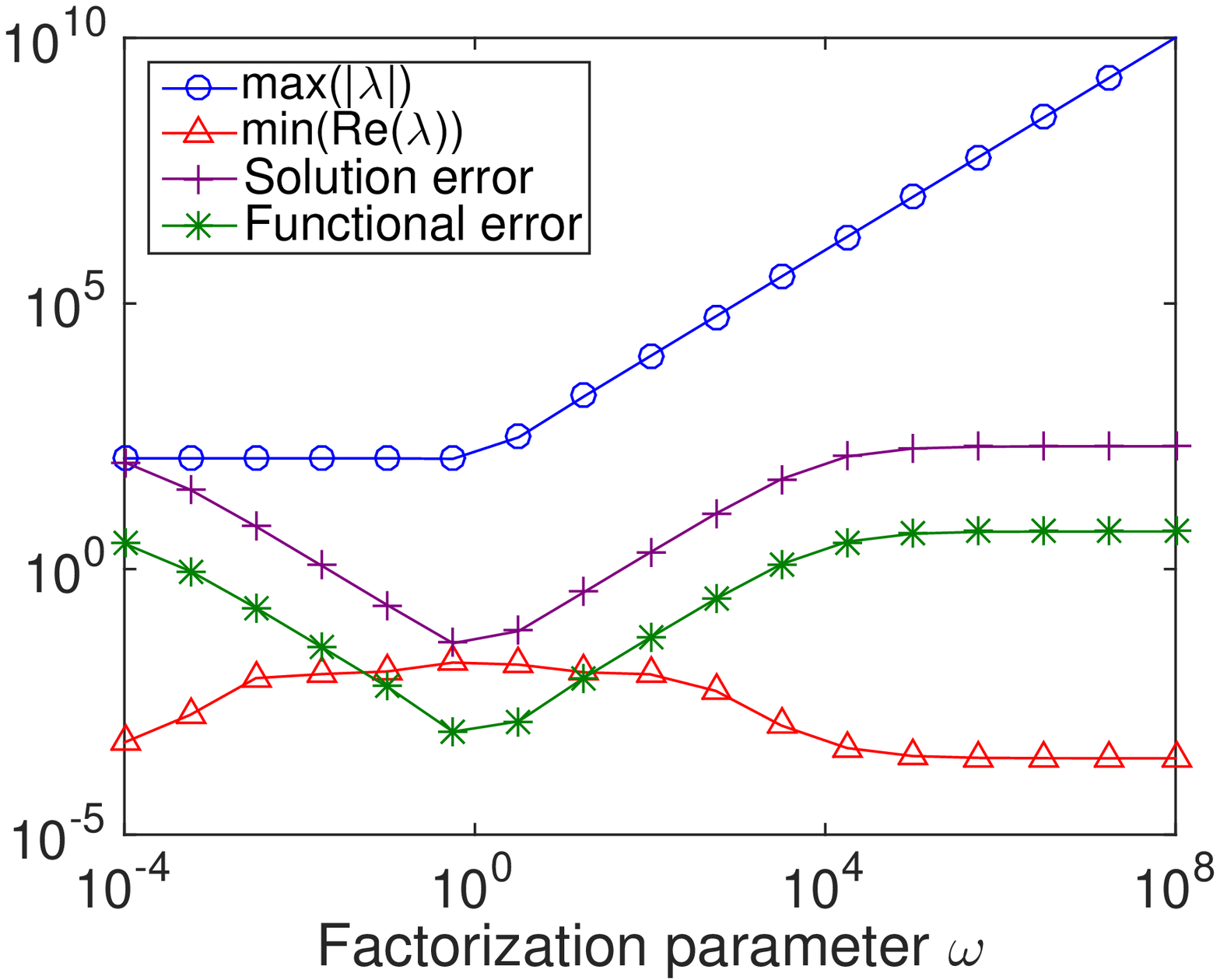}}%
\subfigure[Narrow-stencil second derivative]{\includegraphics[width=.5\textwidth]{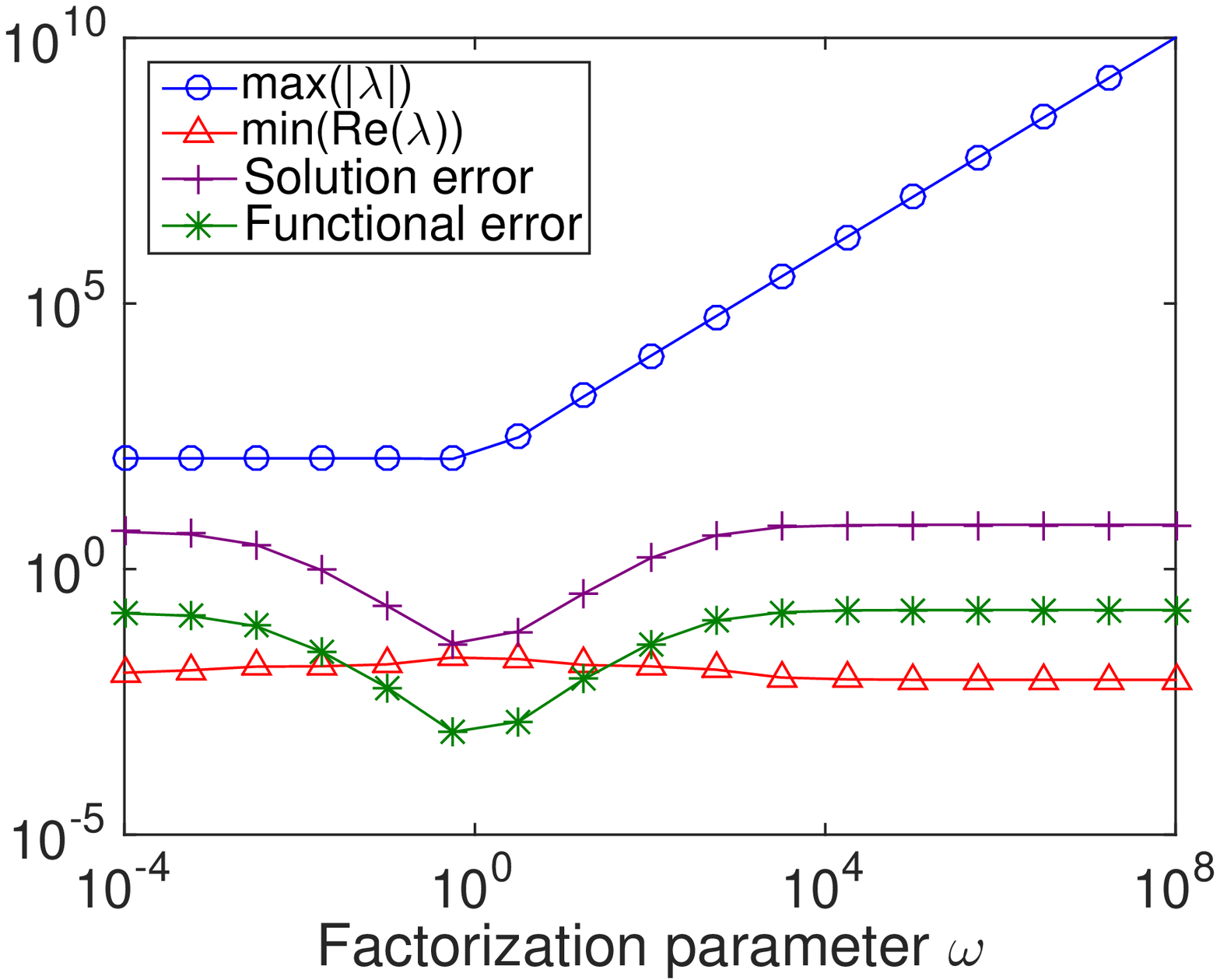}}%
\caption{We solve $a\scalaru_x=\varepsilon \scalaru_{xx}+\force (x)$ with Dirichlet boundary conditions 
and with $\scalaru(x)=\weight (x)=\cos(30x)$.
The number of grid points is $N=64$, 
the interior order is 6.
}\label{FigAD}
\end{figure}

We conclude that the penalties  in  Theorem~\ref{ParaSystPen} yields
 superconvergent functionals for the advection-diffusion equation with Dirichlet boundary conditions -- in the asymptotic limit. In the special case when having the wide scheme with $\rot=|a|$ we even get super-convergent functionals in the troublesome transition region.

\begin{figure}[H]\centering
\subfigure{\includegraphics[width=.5\textwidth]{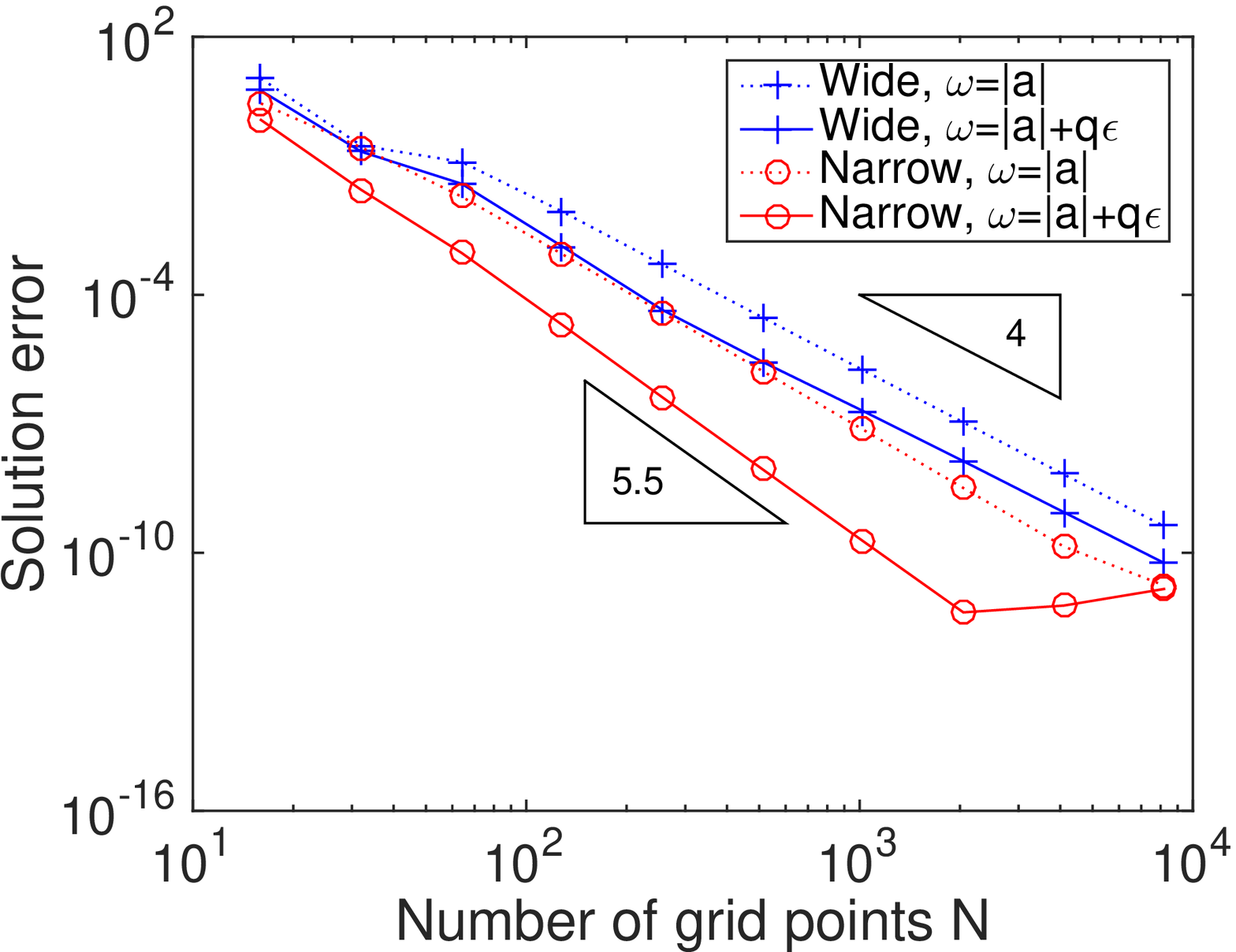}}%
\subfigure{\includegraphics[width=.5\textwidth]{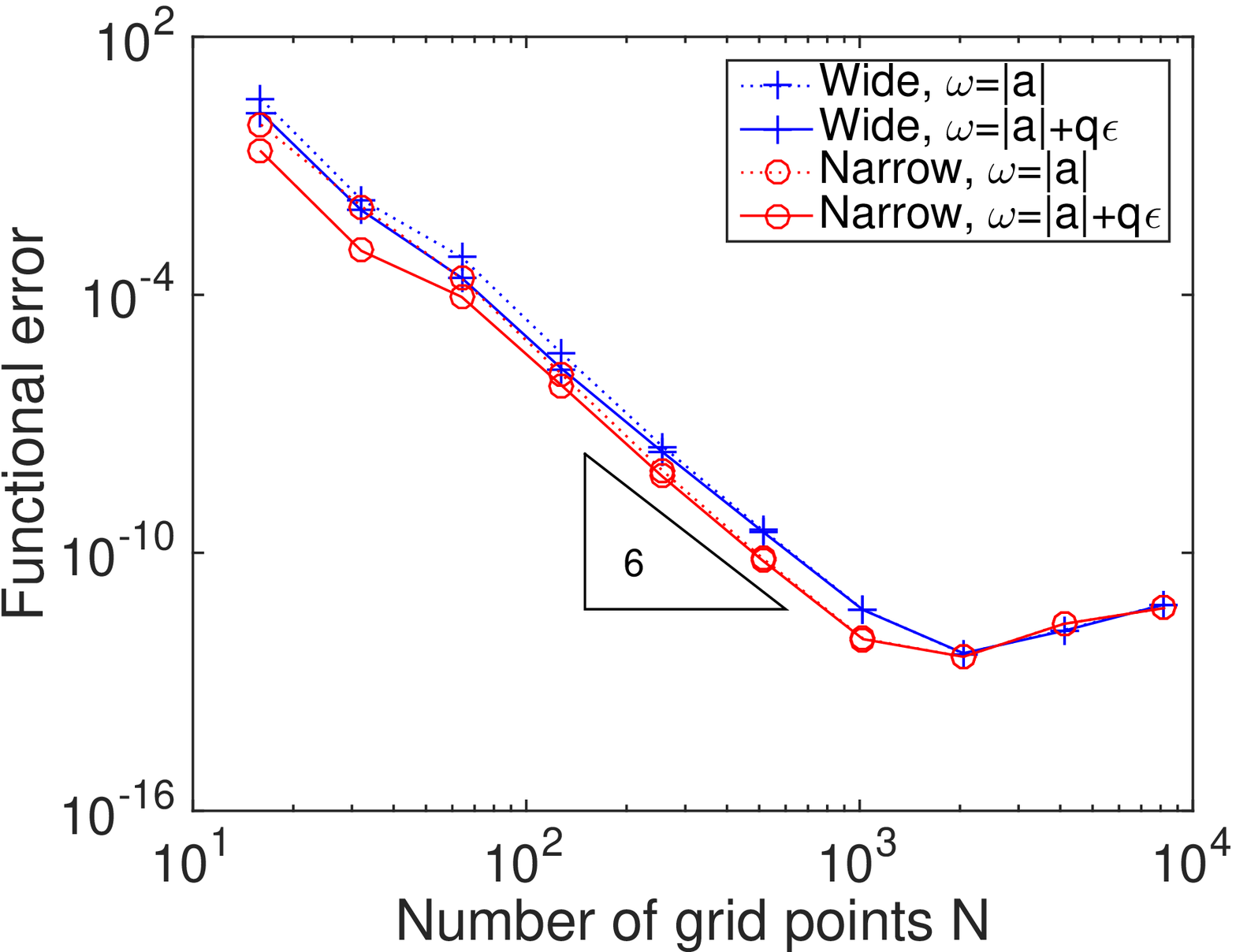}}
\caption{The  inner order of accuracy is 6, $\scalaru(x)=\weight (x)=\cos(30x)$, $a=1$ and $\varepsilon=0.1$.
}\label{FigADresolved}
\end{figure}

\begin{figure}[H]\centering
\subfigure{\includegraphics[width=.5\textwidth]{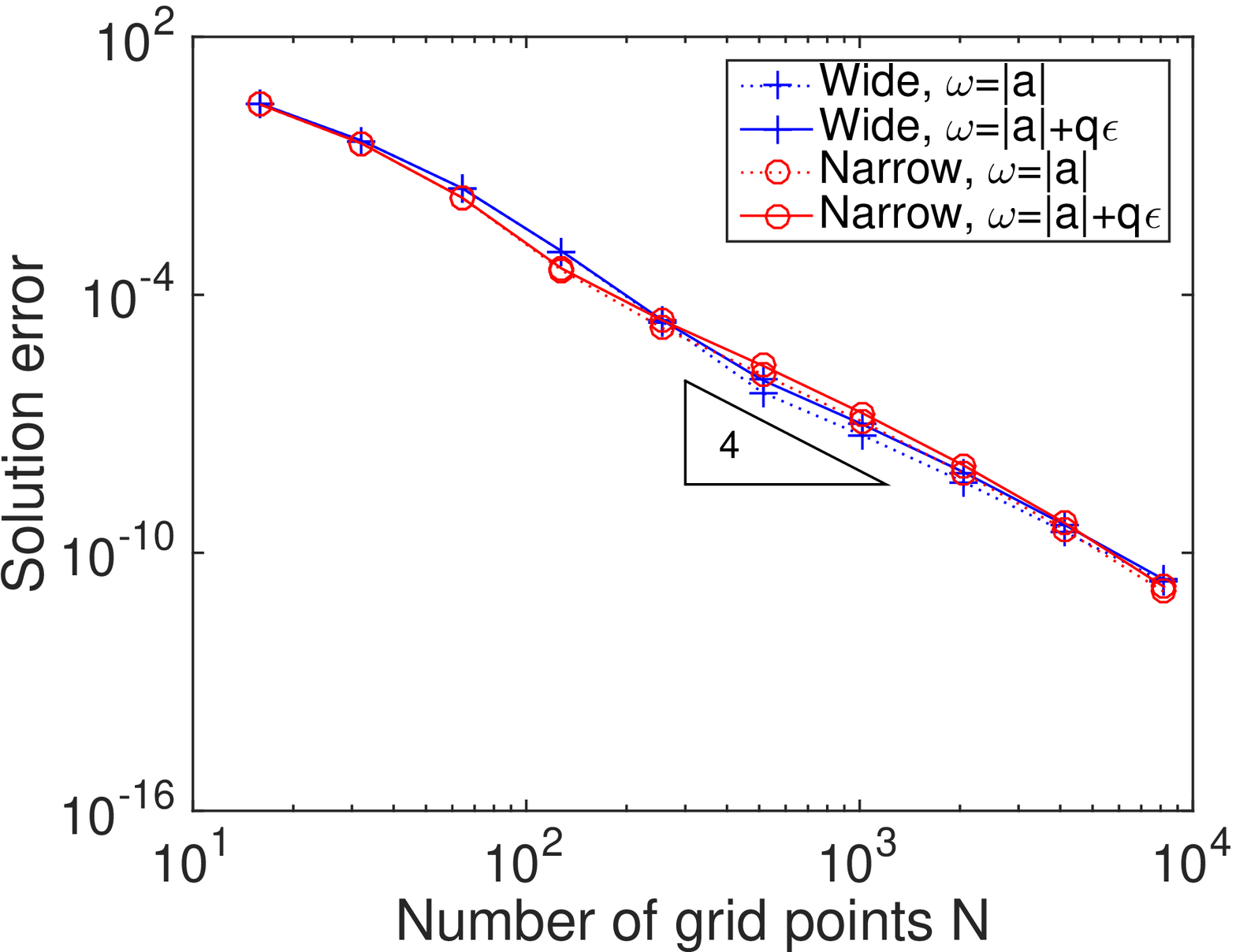}}%
\subfigure{\includegraphics[width=.5\textwidth]{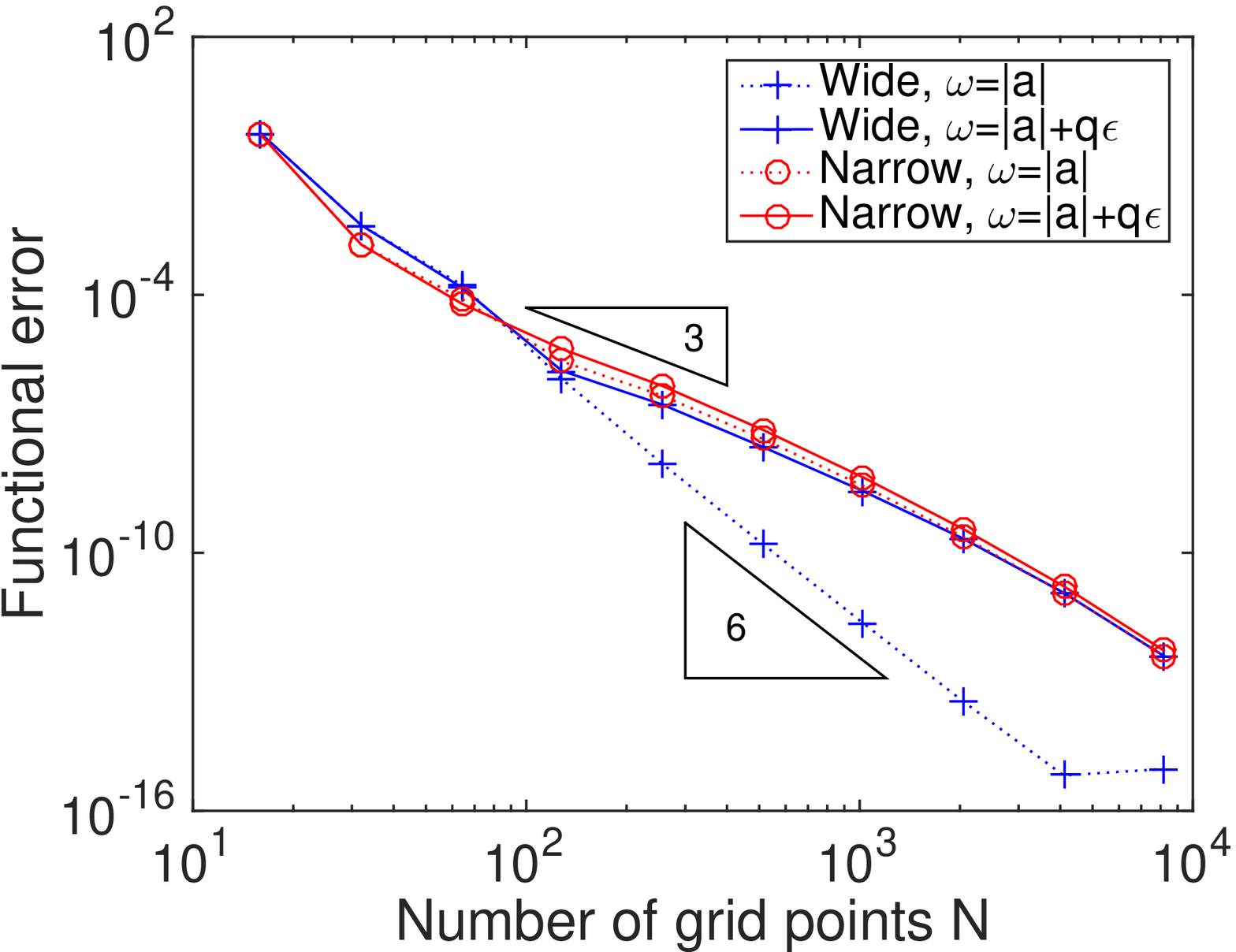}}
\caption{The  inner order of accuracy is 6, $\scalaru(x)=\weight (x)=\cos(30x)$, $a=1$ and $\varepsilon=10^{-4}$.
}\label{FigADnotresolved}
\end{figure}

\subsubsection{The advection-diffusion equation with far-field boundary conditions}

We just comment briefly on
the far-field boundary conditions  and their corresponding
SAT parameters given in \eqref{OptPenScalarFarField}. 
If $|a|h/\varepsilon$ is large,
the quantities $\maxeig$, $\mineig$ and the errors 
 barely depend on
 $\rot$ 
(except if $\force =0$, 
then the errors are smaller if $\rot\approx|a|$). 
%
For small $|a|h/\varepsilon$, large values of $\rot$ give smaller errors and slightly larger $\mineig$, whereas $\maxeig$ is  slightly increased.
In this case,
the
penalty obtained by taking the limit $\rot\to\infty$, that is $\taul=-1$, $\sigmal=0$, $\taur=-1$ and $\sigmar=0$ (corresponding to the penalty used in
\cite{Berg201341,Berg2014135}) 
is not a bad choice  and it has an appealing simplicity.
As before, 
$\rot\approx|a|+\q \varepsilon$ also  gives good results.

\subsubsection{Reflections from the scalar case}

From what we have seen from the numerical experiments so far, the best 
choice of the factorization parameter $\rot$ 
is 
not only dependent on the continuous problem at hand (i.e. the parameters $a$ and $\varepsilon$ and the type of boundary conditions), but 
 also
 on numerical quantities, such as the 
grid resolution 
and if
the stencils are wide or narrow.
In some cases the factorization 
has almost no impact,
sometimes it 
makes the system at hand extremely ill-conditioned or
even changes the order of accuracy of the scheme.

In the scalar
 case it is rather straightforward to optimize with respect to the single factorization parameter $\rot$, 
but for systems this task becomes 
 non-trivial 
 and one might  have 
  to settle for the factorizations at hand.
Nevertheless, we note that the eigendecomposition 
is not necessarily the best factorization and that it could be worth searching for other options. 
With that being said, next we consider a system and use nothing but the eigendecomposition for constructing the penalty parameters.

\subsection{A fluid dynamics system with solid wall boundary conditions}
\label{SysSec}

The symmetrized, compressible
Navier--Stokes equations in one dimension ($\Omega=[0,1]$) with frozen coefficients 
is given
by \eqref{ParaSyst},  with
\begin{align*}
\paraA=\left[\begin{array}{ccc}\bar{u}&\abar&0\\\abar&\bar{u}&\bbar\\0&\bbar&\bar{u}\end{array}\right],&&
\paraE=\varepsilon\left[\begin{array}{ccc}0&0&0\\0&\varphi&0\\0&0&\psi\end{array}\right],&&
\sol =\left[\begin{array}{c}\varrho\\u\\T\end{array}\right],
\end{align*}
where the constants $\bar{u}$, $\abar$, $\bbar$, $\varepsilon$, $\varphi$ and $\psi$ denote suitable physical quantities  and
where $\varrho$, $u$ and $T$ are scaled 
perturbations in density, velocity and temperature.
Let $\bar{u}<0$ and $\varepsilon, \varphi, \psi>0$. In this case, two boundary conditions should be given at the left boundary and three at the right boundary.
We impose solid wall
 boundary conditions (a perfectly insulated wall) at the left boundary, that is 
$u(0,t)=
T_x(0,t)=0$.
 At the right boundary, we impose free stream boundary conditions of Dirichlet type, as $\sol (1,t)=\sol _\infty$.
 These boundary conditions give a well-posed problem.
The boundary operators are
\begin{align*}
\HL&=\left[\begin{array}{ccc}0&1&0\\0&0&0\end{array}\right],&\GL&=\left[\begin{array}{ccc}0&0&0\\0&0&1\end{array}\right],&
\HR&=\left[\begin{array}{ccc}1&0&0\\0&1&0\\0&0&1\end{array}\right],&\GR&=\left[\begin{array}{ccc}0&0&0\\0&0&0\\0&0&0\end{array}\right].
\end{align*}
These boundary conditions 
 can not be rearranged to
  the  far-field form 
  and therefore the penalty used in \cite{Berg201341,Berg2014135} can not be applied.
We identify $\bigA$,
$\bigBL$ and $\bigBR$ 
according to \eqref{LiteViktigare}, 
and factorize  $\bigA$
using the eigendecomposition. 
The dual consistent
penalty parameters
are now described in \eqref{TauUtanMu},
with
\begin{align*}
\factorL&=\left[\begin{array}{ccc}0&0&0\\0&0&1/(\varepsilon\psi)\end{array}\right],
&\factorR&=\left[\begin{array}{ccc}0&0&0\\0&0&0\\0&0&0\end{array}\right].
\end{align*}
As a comparison, we  use the alternative penalty parameters (cf. Method 2 in Remark~\ref{RemNonDual})
\begin{align*}
\nondual{\penI}_0=\left[\begin{array}{cc}-\abar& 0\\0 &0\\ -\bbar &\varepsilon\psi\end{array}\right],&&
\nondual{\penS}_0=\left[\begin{array}{cc}0 &0\\\varepsilon\varphi& 0\\0& 0\end{array}\right],&&
\nondual{\penI}_N=\left[\begin{array}{ccc}\bar{u} &\abar &0\\0& \bar{u}& 0\\0 &\bbar& \bar{u}\end{array}\right],&&
\nondual{\penS}_N=\left[\begin{array}{ccc}0& 0& 0\\0& -\varepsilon\varphi& 0\\0& 0& -\varepsilon\psi\end{array}\right]
\end{align*}
which give stability (they are chosen such that the boundary terms 
in \eqref{FirstBTCompact} are non-positive for zero data) but 
they
 do not fulfill the demands for dual consistency.

In the numerical  simulations we 
use 
the exact solution 
$\varrho=\cos(7x)$, $u=\sin(13x)$ and $T=\cos(30x)$ and as weight functions 
we  use $\weight (x)=[1,\ 0,\ 0]^T$,  $\weight (x)=[0,\ 1,\ 0]^T$ and  $\weight (x)=[0,\ 0,\ 1]^T$ (such that one functional output is obtained for each variable). 
Figure~\ref{FigSYSTEMsol} shows
the resulting errors 
when using the schemes with interior order 6.
In the wide case,  the solutions do not differ much.
In the narrow case, the dual consistent solution  converges one half order slower than the dual inconsistent one (order 4 for $\varrho$ and 4.5 for $u,T$ compared to 4.5 for $\varrho$ and 5 for $u,T$), but 
the result is still as good as in the wide case. Moreover,
recall that in the scalar case the order 
could be improved by choosing another factorization than the eigendecomposition, see 
Figure~\ref{FigADresolved}(a). 
In
Figure~\ref{FigSYSTEMfunc} we see that the functionals convergence with the expected 6th order for both the dual consistent schemes, whereas the dual inconsistent schemes yield 5th order.

  \begin{figure}[H]
  \centering
  \subfigure[Interior order 6, wide operator]{\includegraphics[width=.5\textwidth]{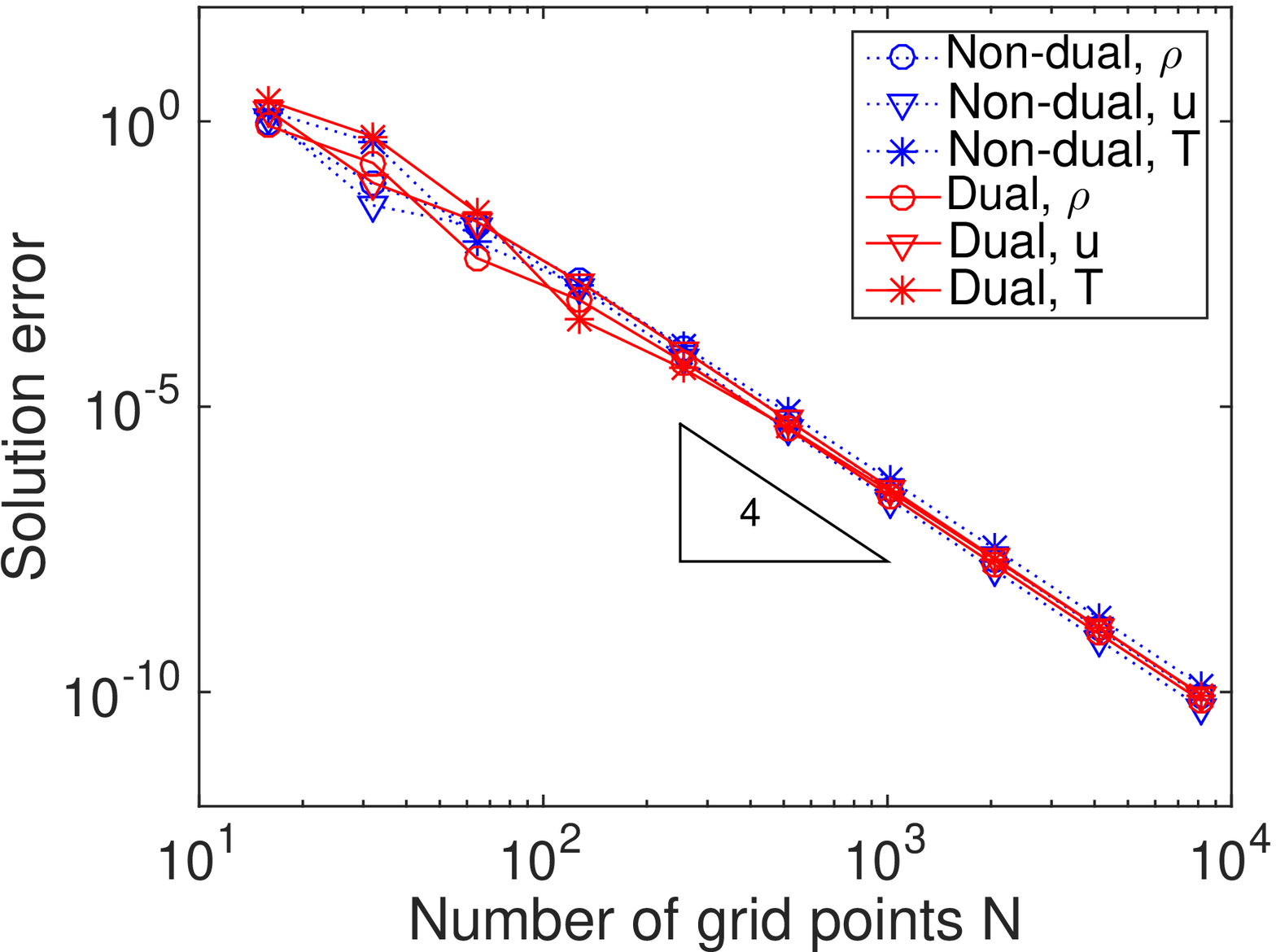}}%
\subfigure[Interior order 6, narrow operator]{\includegraphics[width=.5\textwidth]{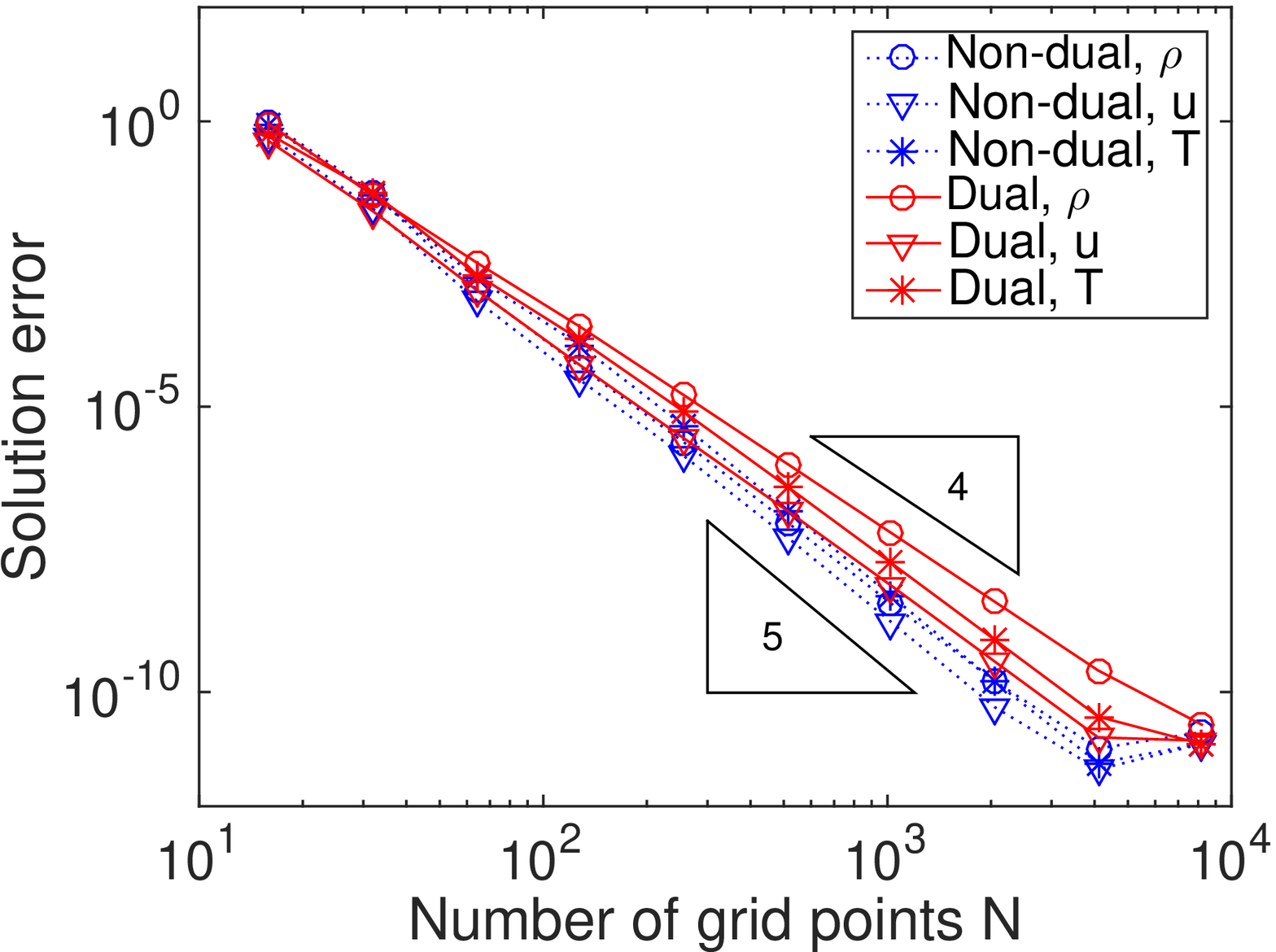}}%
\caption{Solution errors, 
for $\bar{u}=-0.5$, $\abar=0.8$, $b=0.6$, $\varphi=1$, $\psi=2$, $\varepsilon=0.01$.
}\label{FigSYSTEMsol}
\end{figure}
  \begin{figure}[H]
  \centering
\subfigure[Interior order 6, wide operator]{\includegraphics[width=.5\textwidth]{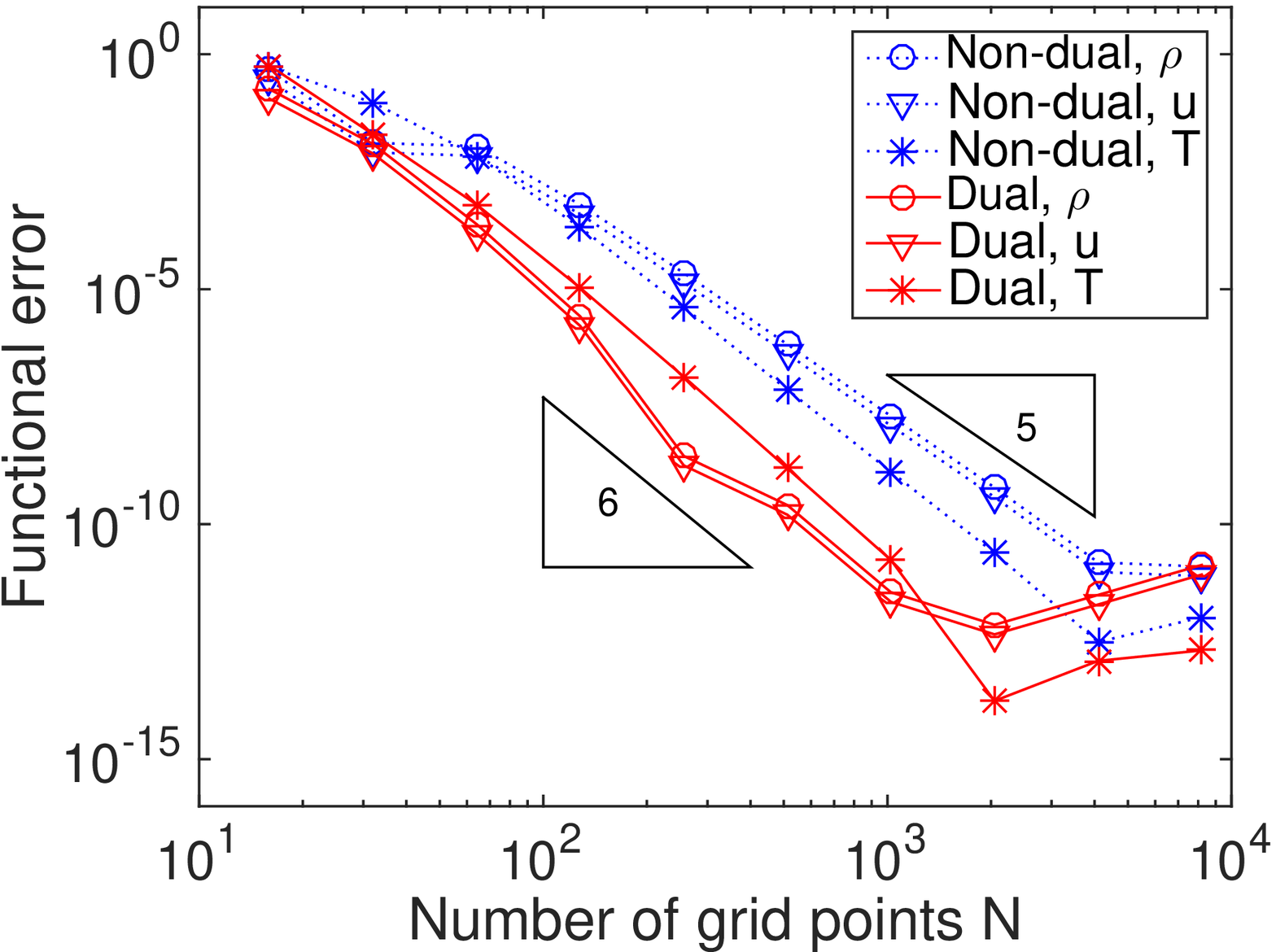}}%
\subfigure[Interior order 6, narrow operator]{\includegraphics[width=.5\textwidth]{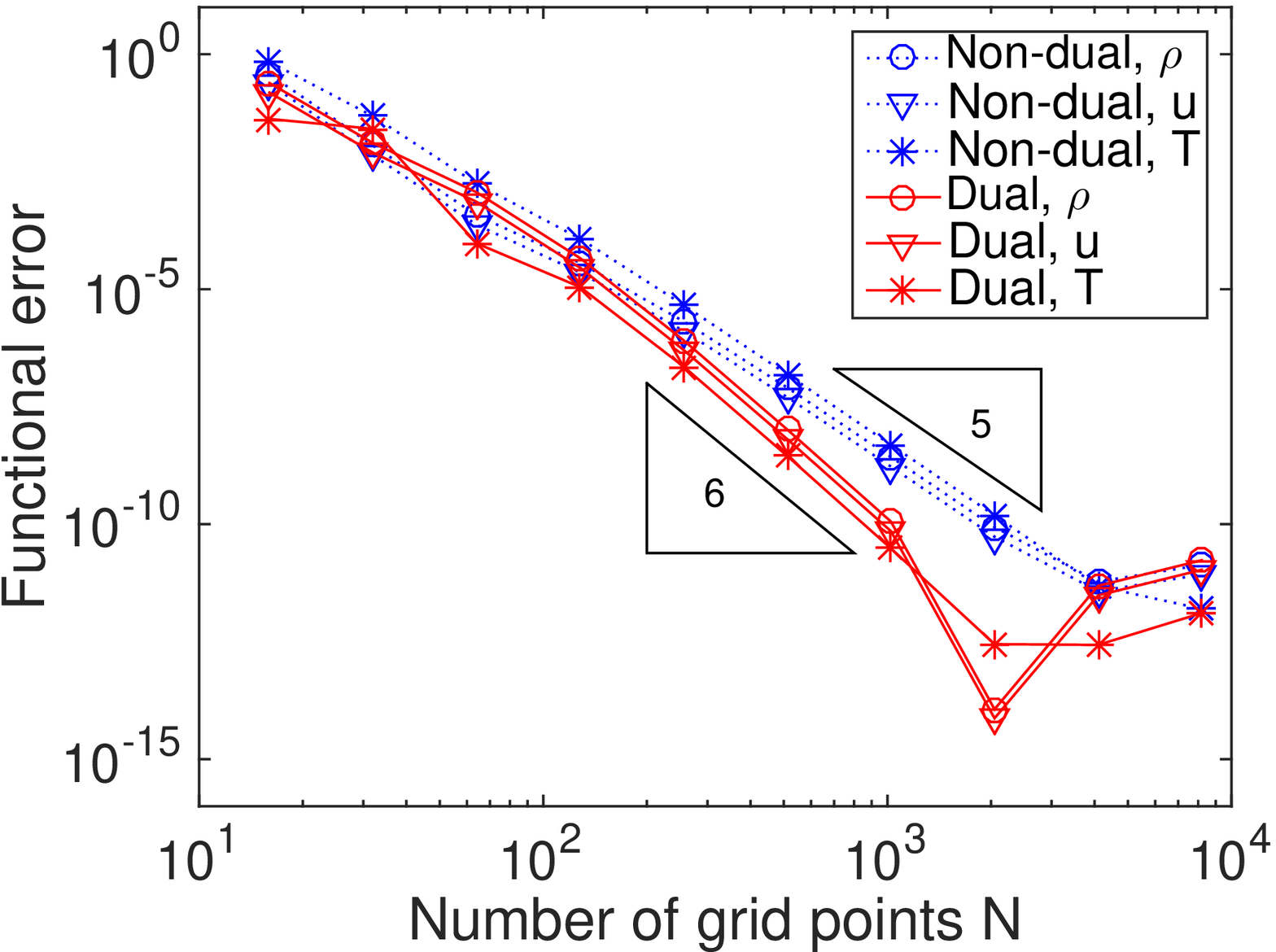}}%
\caption{Functional errors, 
for
 $\bar{u}=-0.5$, $\abar=0.8$, $b=0.6$, $\varphi=1$, $\psi=2$, $\varepsilon=0.01$.
}\label{FigSYSTEMfunc}
\end{figure}

The diffusion parameter is decreased from  $\varepsilon=0.01$ to $\varepsilon=10^{-6}$ and the resulting errors are shown in Figures~\ref{FigSYSTEMhighReSol} and \ref{FigSYSTEMhighReFunc}. 
Now the solution 
errors 
obtained using the
dual consistent schemes are  slightly better
than the ones obtained using the dual inconsistent
schemes, but the difference is small, 
see Figure~\ref{FigSYSTEMhighReSol}.
  \begin{figure}[H]
  \centering
  \subfigure[Interior order 6, wide operator]{\includegraphics[width=.5\textwidth]{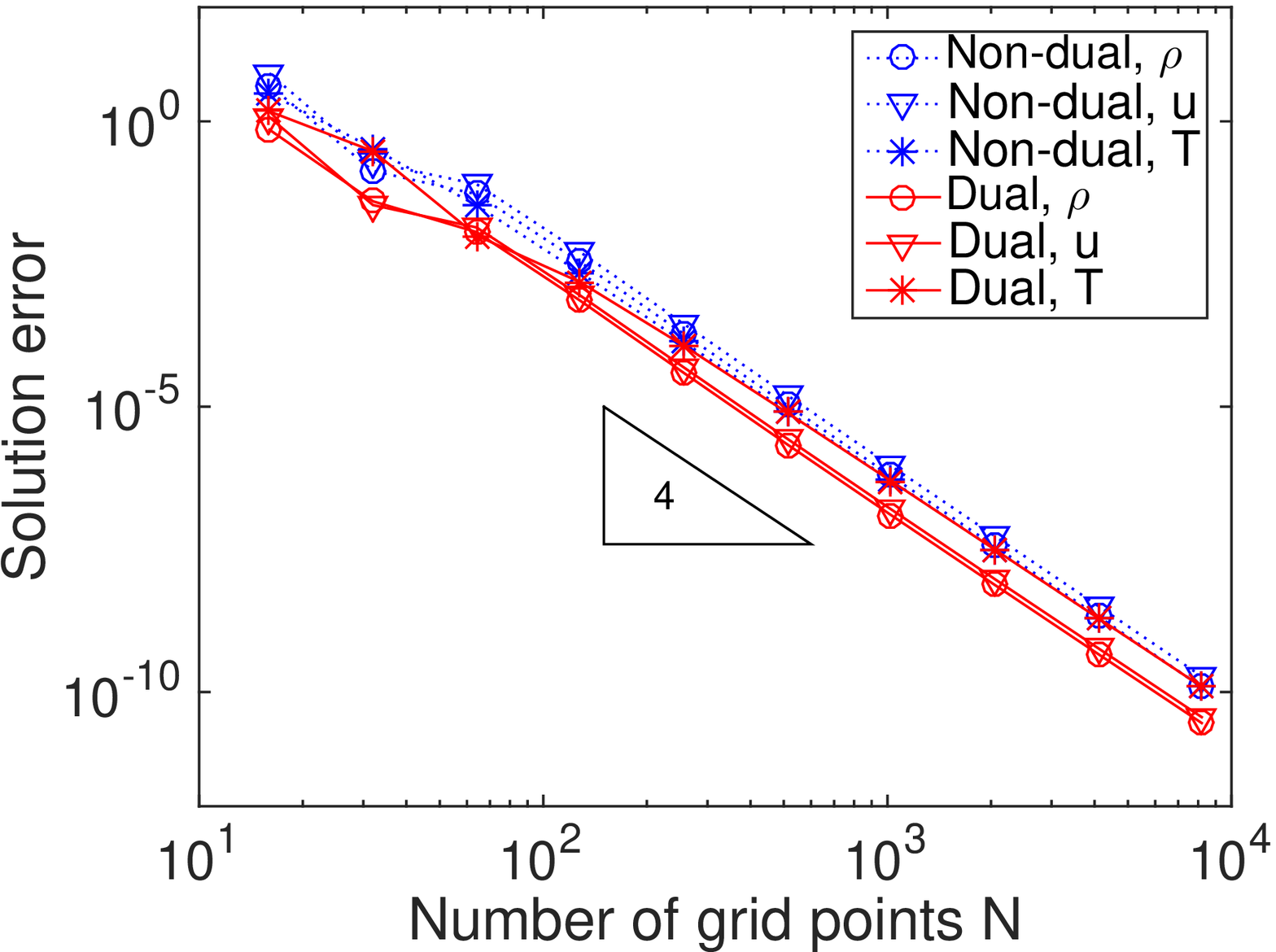}}%
\subfigure[Interior order 6, narrow operator]{\includegraphics[width=.5\textwidth]{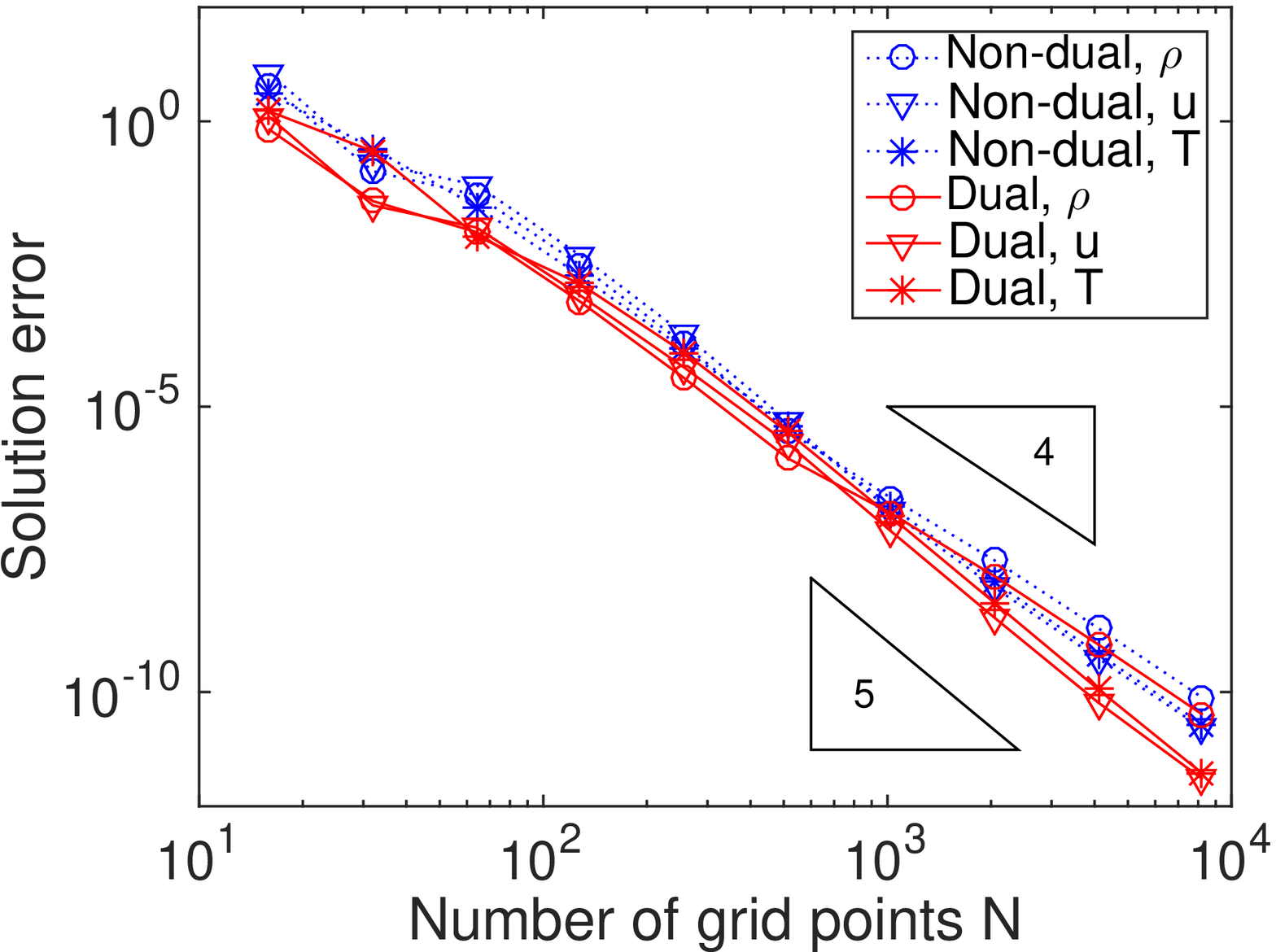}}%
\caption{
Solution errors, 
for $\bar{u}=-0.5$, $\abar=0.8$, $b=0.6$, $\varphi=1$, $\psi=2$, $\varepsilon=10^{-6}$.
}\label{FigSYSTEMhighReSol}
\end{figure}
For the functional errors the difference is more pronounced, see Figure~\ref{FigSYSTEMhighReFunc}. 
In the wide case, the dual consistent  scheme produces a perfect convergence rate of almost 7.
This behavior was observed already in the scalar case,  when the factorization parameter was chosen exactly as $\rot=|a|$ (which 
for small amounts of diffusion is very  close to the eigendecomposition).
For the narrow-stencil schemes the dual consistent scheme still produces smaller errors than the dual inconsistent scheme, but the order is reduced to 3 (a pre-asymptotic low-order tendency seen already in Figure~\ref{FigADnotresolved} in the scalar case).

  \begin{figure}[H]
  \centering
\subfigure[Interior order 6, wide operator]{\includegraphics[width=.5\textwidth]{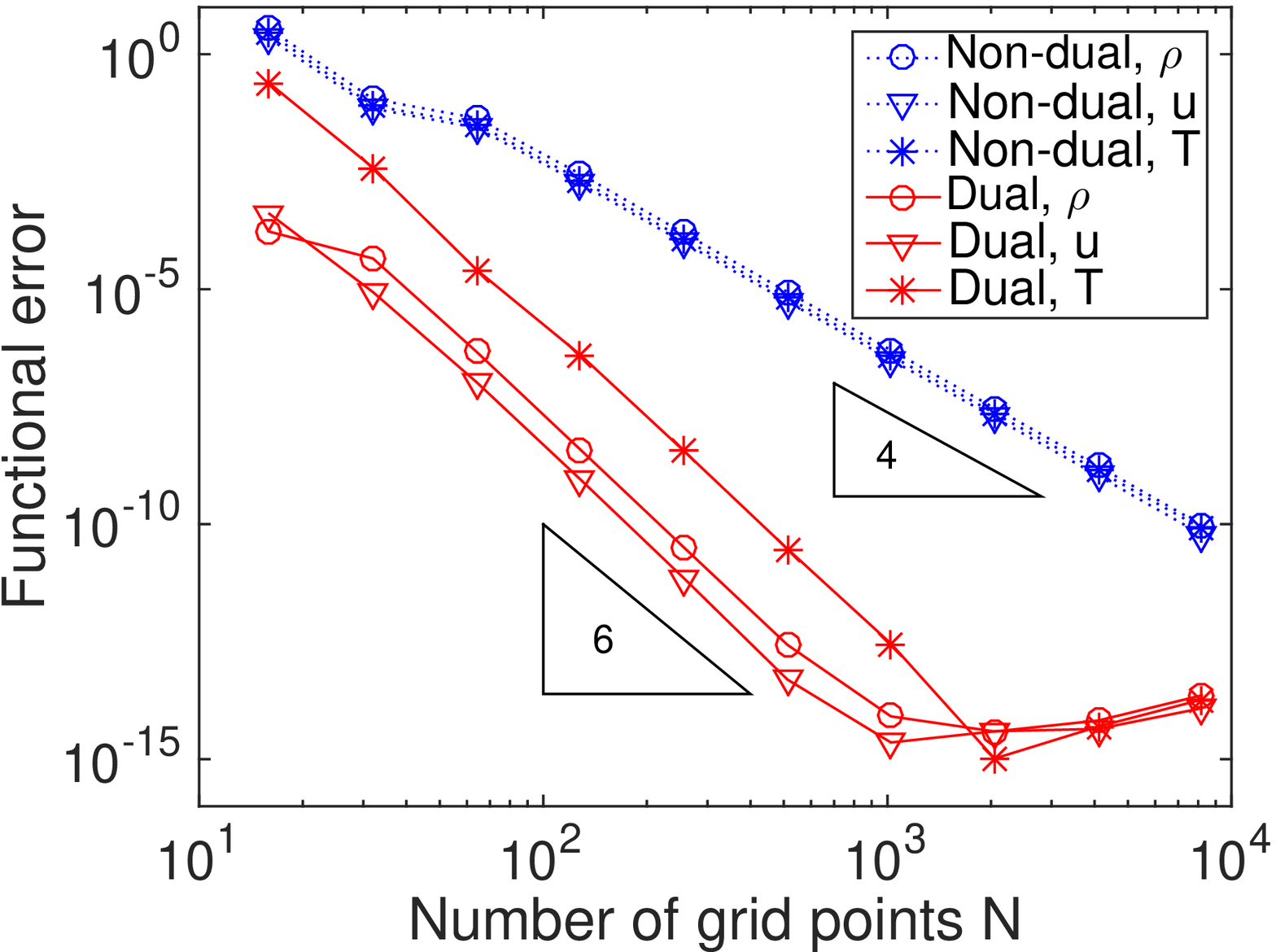}}%
\subfigure[Interior order 6, narrow operator]{\includegraphics[width=.5\textwidth]{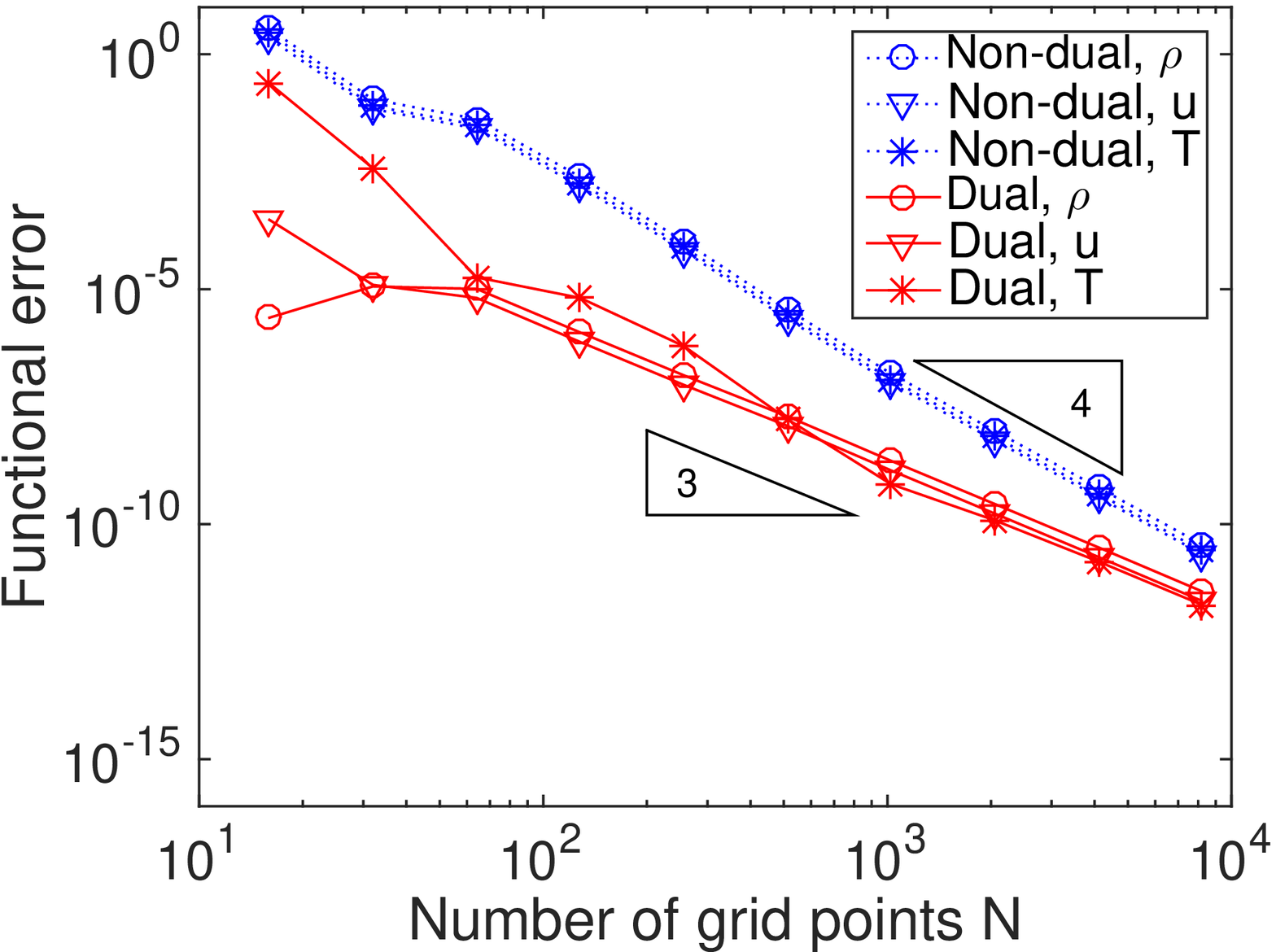}}%
\caption{
Functional errors, 
for 
 $\bar{u}=-0.5$, $\abar=0.8$, $b=0.6$, $\varphi=1$, $\psi=2$, $\varepsilon=10^{-6}$.
}\label{FigSYSTEMhighReFunc}
\end{figure}
Extrapolating from the scalar case, we  assume that it could be worth searching for better penalty parameters for the narrow-stencil schemes when having diffusion dominated problems. However, for convection dominated problems the wide scheme with a factorization close to the eigendecomposition  is 
 hard to beat.

\section{Concluding remarks}
\label{conclusions}

We use a finite difference method based on summation by parts operators, combined with a penalty method for the 
boundary conditions (SBP-SAT).
Diagonal-norm SBP operators have $2p$-order accurate interior stencils 
 and $p$-order accurate boundary closures, which limits the global  accuracy of the solution to $p+1$ (or $p + 2$ for parabolic problems under
certain conditions).
Recently, it has been shown that SBP-SAT
schemes can give functional
estimates that are $\mathcal{O}(h^{2p})$. To achieve this superconvergence, the SAT parameters
must be carefully chosen to ensure that the discretization is dual-consistent.

We first look at hyperbolic systems 
and derive
stability requirements and  duality constraints for the SATs. Then
we present a recipe to choose these SAT parameters such that both these (independent) demands are fulfilled.
When wide-stencil second derivative operators are used,
the results automatically extend to parabolic problems. We   generalize the recipe  such that it holds also for narrow-stencil second derivative operators. 

The $2p$ order convergence of SBP-SAT functional estimates is confirmed numerically 
for a
variety of scalar examples, as well as for an incompletely parabolic system.
For low-diffusion  advection-diffusion problems,
the superconvergence is sometimes seen first asymptotically. Generally speaking, the narrow-stencil schemes are better for diffusion dominated problems whereas the wide schemes are preferable for advection dominated problems.

In most cases the derived dual consistent SAT parameters have some remaining  degree of freedom. 
 The free parameters 
 can be used to improve the accuracy of the primary solution or to tune numerical quantities such as spectral radius, decay rate or condition numbers.
 Optimal choices within these families are suggested for the scalar problems, however, to do the same for systems is considered a task for the future.

\appendix

\section{Reformulation of the first order form discretization}
\label{ReformWide}

We derive
the scheme \eqref{FinalWide} with penalty parameters
\eqref{WIDEtauRel},
using the hyperbolic results.

\

\noindent{\bf Step 1:} Consider the 
problem \eqref{ParaSystAsHypSyst}, which is
 a first order system. We represent the solution $\bigU$ by a discrete solution vector  $\bigV=[\bigV_0^T,\bigV_1^T,\hdots,\bigV_N^T]^T$, where $\bigV_i(t)\approx \bigU(x_i,t)$ and discretize \eqref{ParaSystAsHypSyst} exactly as was done in \eqref{HypSystDisc} for the hyperbolic case, that is as
\begin{align}
\label{Approach1}
\begin{split}(I_N\otimes \bigI)\bigV_t+(I_N\otimes \bigR)\bigV+(D_1\otimes \bigA)\bigV=\bigFdisc&+(\PH ^{-1}e_0 \otimes\bigSigL )(\bigBL{\bigV}_0 -\gl)\\
&+(\PH ^{-1}e_N \otimes\bigSigR )(\bigBR{\bigV}_N -\gr).
\end{split}
\end{align}
As 
proposed in Theorem~\ref{HypSystPen}, we let
$\bigSigL=-\bigX_+\bigRot^{ }_+\bigJL^{-1}$ and $\bigSigR=\bigX_-\bigRot^{ }_-\bigJR^{-1}$.

\

\noindent{\bf Step 2:} 
We discretize \eqref{ParaSyst} directly by approximating $\sol $ by $\nsol $ and $\sol _x$ by $\wide{W}$. 
We obtain
\begin{subequations}
\label{Approach2}
\begin{align}
\label{Approach2a}
\begin{split}
\nsol _t
+(D_1\otimes \paraA)\nsol -(D_1\otimes \paraE)\wide{W}&=\fh +(\PH ^{-1}e_0 \otimes\sigltop )(\HL{\nsol }_0 +\GL \wide{W}_0 -\gl)\\
&\hspace{23.5pt}+(\PH ^{-1}e_N \otimes\sigrtop )(\HR{\nsol }_N+\GR \wide{W}_N -\gr),
\end{split}\\
\label{Approach2b}
\begin{split}
(I_N\otimes \paraE)\wide{W}
-(D_1\otimes \paraE)\nsol &=(\PH ^{-1}e_0 \otimes\siglbot )(\HL{\nsol }_0 +\GL \wide{W}_0 -\gl)\\
&+(\PH ^{-1}e_N \otimes\sigrbot )(\HR{\nsol }_N+\GR \wide{W}_N -\gr).
\end{split}
\end{align}
\end{subequations}
If $
\bigSigL=[\sigltop^T,\siglbot^T]^T$ and $\bigSigR=[\sigrtop^T,\sigrbot^T]^T$, then
\eqref{Approach2}
  is 
  a permutation of  \eqref{Approach1}.

\

\noindent{\bf Step 3:}
The scheme in \eqref{Approach2} is a  system of differential algebraic equations, 
so
we would like to
cancel the variable $\wide{W}$ 
and get a system of ordinary differential equations instead.
%
Multiplying 
 \eqref{Approach2b} 
by $\Dbar=(D_1\otimes I_\sizepara)$ and adding the result to 
\eqref{Approach2a}, 
yields
\begin{align}
\begin{split}\nsol _t
+(D_1\otimes \paraA)\nsol -(D_1
^2\otimes \paraE)\nsol =\fh &+(\PH ^{-1}e_0 \otimes\sigltop+D_1\PH ^{-1}e_0 \otimes\siglbot )\wide\chi_0\\
&+(\PH ^{-1}e_N \otimes\sigrtop +D_1\PH ^{-1}e_N \otimes\sigrbot)\wide\chi_N,
\end{split}
\notag
\end{align}
where
\begin{align}
\label{chi}
\wide\chi_0=\HL{\nsol }_0 +\GL \wide{W}_0 -\gl,&&
\wide\chi_N=\HR{\nsol }_N+\GR \wide{W}_N -\gr.
\end{align}
Next, using the properties in \eqref{SBPprop1}, 
together with the fact that  $\PH $ is diagonal, 
we  compute
\begin{align}D_1\PH ^{-1}e_0
=\PH ^{-1}(-\p I_N-D_1^T)e_0,&&
D_1\PH ^{-1}e_N=\PH ^{-1}(\p I_N-D_1^T)e_N,\notag\end{align}
where $\p$ is the scalar $\p=e_0^T\PH ^{-1}e_0=e_N^T\PH ^{-1}e_N$ given in \eqref{defp}.
This yields
\begin{align}
\label{StillSomeW}
\begin{split}\nsol _t
+(D_1\otimes \paraA)\nsol -(D_1
^2\otimes \paraE)\nsol 
&=\fh +\Pbar^{-1}(e_\noll \otimes(\sigltop-\p\siglbot)-D_1^Te_\noll \otimes\siglbot )\wide\chi_\noll\\
&\hspace{27pt}+\Pbar^{-1}(e_N \otimes(\sigrtop+\p\sigrbot) -D_1^Te_N \otimes\sigrbot)\wide\chi_N,
\end{split}
\end{align}
where $\Pbar=(\PH \otimes I_\sizepara)$.
However, the boundary condition deviations $\wide\chi_0$ and $\wide\chi_N$ 
still contain $\wide{W}$, so we
multiply 
\eqref{Approach2b} by $(e_0^T\otimes I_\sizepara)$ and $(e_N^T\otimes I_\sizepara)$, respectively, to get
\begin{align}
\label{2ndWideRel}
 \paraE \wide{W}_0
- \paraE(\Dbar \nsol )_0
&=\p \siglbot \wide\chi_0,&
 \paraE \wide{W}_N
- \paraE(\Dbar \nsol )_N
&=\p \sigrbot \wide\chi_N.
\end{align}
Next, we need 
boundary condition deviations without $\wide{W}$, and define
\begin{align*}
\wide\xi_\noll&=\HL{\nsol }_0 +\GL (\Dbar \nsol )_0 -\gl,&\wide\xi_N&= \HR{\nsol }_N+\GR (\Dbar \nsol )_N -\gr.
\end{align*}
Recall that 
$\GG_{L,R}=\factor_{L,R}\paraE$. 
Using \eqref{2ndWideRel}, we can now relate $\wide\xi_{0,N}$ above to $\wide\chi_{0,N}$ in \eqref{chi} as
\begin{align}
\label{xichirel}
\wide\xi_\noll
&=(I_{\sizebig_+}-\p \factorL\siglbot )\wide\chi_0
,&
\wide\xi_N
&= (I_{\sizebig_-}-\p\factorR \sigrbot )\wide\chi_N,
\end{align}
where $I_{\sizebig_+}$ and $I_{\sizebig_-}$ are identity matrices of sizes corresponding to the number of positive (${\sizebig_+}$) and negative  ($\sizebig_-$) eigenvalues of $\bigA$, respectively.
Inserting $\wide\chi_{0,N}$ from \eqref{xichirel} into \eqref{StillSomeW} allows us
to finally write the scheme without any $\wide{W}$
terms and we obtain \eqref{FinalWide},
with
\begin{align}
\label{WIDEtaumuRel}
\begin{split}
\wide\penI_\noll&=(\sigltop-\p\siglbot)
(I_{\sizebig_+}-\p \factorL\siglbot )^{-1},\hspace{53pt}\wide\penS_\noll=-\siglbot(I_{\sizebig_+}-\p \factorL\siglbot )^{-1},
\\
\wide\penI_N&=(\sigrtop+\p\sigrbot)
(I_{\sizebig_-}-\p\factorR \sigrbot )^{-1},\hspace{40pt}
\wide\penS_N=-\sigrbot(I_{\sizebig_-}-\p\factorR \sigrbot )^{-1}.
\end{split}
\end{align}
From Step 1 and 2 we know that
\begin{align*}
\left[\begin{array}{c}\sigltop\\\siglbot\end{array}\right]=-\left[\begin{array}{c}\bigX_1\bigRot^{ }_+\bigJL^{-1}\\\bigX_2\bigRot^{ }_+\bigJL^{-1}\end{array}\right],&&\left[\begin{array}{c}\sigrtop\\\sigrbot\end{array}\right]=\left[\begin{array}{c}\bigX_3\bigRot^{ }_-\bigJR^{-1}\\\bigX_4\bigRot^{ }_-\bigJR^{-1}\end{array}\right],
\end{align*}
where $\bigX_{1,2,3,4}$ are given in \eqref{X1234}.
Inserting the above relation into \eqref{WIDEtaumuRel},
we obtain the penalty parameters presented in \eqref{WIDEtauRel}.

\section{Motivation of Proposition~\ref{propJK}}
\label{qex}

In Proposition~\ref{propJK} we claim that the inverse of
 $\compAsnok = \compA+\pert E_j$ is
$J/\pert+K_j$. We motivate this  below, for $j=0$.
%
First,  we name the parts of $ \compA$  and present the 
structure of $K_0$ as
\begin{align*}
 \compA=\left[\begin{array}{cc}\Atop&\Avec^T\\
 \Avec&\Adel\end{array}\right],&& K_0=\left[\begin{array}{cc}\noll&\stympadnolla^T\\
 \stympadnolla&\Kdel\end{array}\right].
\end{align*}
Since $\compA$ consists of consistent difference operators, it does not "see" constants.  
Therefore, 
 $\compA J=\stornollmatris$ (since  $J$ is an all-ones matrix) 
  and
  $\Avec+\Adel\stympadetta=\stympadnolla$, where
$\stympadetta=[1, 1, \hdots, 1]^T$. 
 Moreover, due to the special structure of $K_0$, we know that $E_0K_0=\stornollmatris$.
%
Thus we have
\begin{align*}
(\compA+\pert E_0)(J/\pert+K_0)
=\compA K_0+ E_0J
 =\left[\begin{array}{cc}1&\Avec^T\Kdel+\stympadetta^T\\
 \stympadnolla&\litenidentitet\end{array}\right]=\storidentitet.
\end{align*}
The simplest possible example is
the 
narrow   (2,0)
order operator in Table~\ref{qconst}, 
specified by
\begin{align}
\label{2nd}
D_2=\frac{1}{h^2}\left[\begin{array}{ccccc}0&0\\1&-2&1\\&\ddots&\ddots&\ddots\\&&1&-2&1\\&&&0&0\end{array}\right],
&&\text{with}&&
\PH =h\left[\begin{array}{ccccc}1/2\\&1\\
&&\ddots\\&&&1\\&&&&1/2\end{array}\right].
\end{align}
Using \eqref{SBPprop2} and the above structure of $K_0$, respectively, we obtain
\begin{align}
\notag
S=\frac{1}{h}\left[\begin{array}{ccccc}-1&1\\\times&\times&\times&\times&\times\\
\vdots&\vdots&\vdots&\vdots&\vdots\\\times&\times&\times&\times&\times\\&&&-1&1\end{array}\right],&&
K_0=h\left[\begin{array}{ccccc}0&0&\hdots&0&0\\0&1&\hdots&1&1\\
\vdots&\vdots&\ddots&\vdots&\vdots\\0&1&\hdots&N-1&N-1\\0&1&\hdots&N-1&N\end{array}\right].
\end{align}
The interior rows of $S$ are marked by $\times$'s  
because they are unknown. 
Next, we compute 
\begin{align}
\Msnok^{-1}&=S\compAsnok ^{-1}S^T
=S\left(J/\pert+K_0\right)S^T
=\frac{1}{h}\left[\begin{array}{ccccc}1&\times&\hdots&\times&0\\\times&\times&\hdots&\times&\times\\
\vdots&\vdots&&\vdots&\vdots\\\times&\times&\hdots&\times&\times\\0&\times&\hdots&\times&1\end{array}\right]\notag.
\end{align}
Just as $\compA$,
 the difference stencils in the first and last row of $S$ do not "see" $J$. Therefore, the corner elements of $\Msnok^{-1}$ only depend on $K_0$ and are  independent of $\pert$. We conclude that 
when $M$ is singular and $S$ is non-singular
the constants in \eqref{qdefparts} can be computed using \eqref{qdefpartsSNOK}.
In this case we get $\q_0=\q_N=1/h$ and $\q_c=0$, such that $\q=1/h$. 

In addition to the operator discussed above, 
we 
use the diagonal-norm operators  in \cite{Mattsson2004503}.
For the higher order accurate operators found in \cite{Mattsson2004503}, $\q$ varies with $N$. 
For example, for the narrow  (4,2) order accurate operator, we have
\begin{align*}
\begin{array}{cccc}
N&\q_0h&\q_ch&\q h\\
\hline
   8&   3.986350339808304  & 0.000041141179445 &  3.986391480987749\\
   9&   3.986350339313381 &  0.000002953803786  & 3.986353293117168\\
  10&   3.986350339310830 &  0.000000212073570  & 3.986350551384400\\
  11&   3.986350339310817  & 0.000000015226197  & 3.986350354537014\\
  12&   3.986350339310817   &0.000000001093192  & 3.986350340404008\\
\end{array}
\end{align*}
Since the values do not differ so much, it is practical to use the largest value, the one for $N=8$, regardless of the number of grid points.


\bibliographystyle{plain}      
\bibliography{OptPenShort}   %

\end{document}